\documentclass{amsart}

\usepackage{mathabx,hyperref,mathtools,amsrefs,mathrsfs,lineno,comment}
\usepackage{amsfonts,amssymb}
\usepackage{color,url}
\usepackage{tikz}
\usepackage{enumitem}

\usepackage{accents}
\usetikzlibrary{cd,decorations.pathmorphing, decorations.pathreplacing, decorations.shapes,decorations.markings,arrows,calc,intersections}

\DeclareSymbolFont{bbold}{U}{bbold}{m}{n}
\DeclareSymbolFontAlphabet{\mathbbold}{bbold}

\newcommand\IH{{\mathbb{H}}}
\newcommand\C{{\mathbb{C}}}
\newcommand\ID{{\mathbb{D}}}
\newcommand\Q{{\mathbb{Q}}}
\newcommand\R{{\mathbb{R}}}
\newcommand\N{{\mathbb{N}}}

\newcommand\Pone{{\widehat\C}}
\newcommand\Mod{{\operatorname{PMod}}}
\newcommand{\rs}{\widehat{\C}}
\newcommand{\Aut}{{\operatorname{Aut}}}

\newcommand{\Teich}{\text{Teich}}
\newcommand{\Moduli}{\text{Moduli}}
\DeclareMathOperator{\arcsinh}{arcsinh}
    
\newtheorem{thm}{Theorem}[section]

\newtheorem{cor}[thm]{Corollary}
\newtheorem{lem}[thm]{Lemma}
\newtheorem{prop}[thm]{Proposition}

\newtheorem{setup}[thm]{Setup}
\newtheorem{mainthm}{Theorem}
\newtheorem{maincor}[mainthm]{Corollary}
\newtheorem{mainconj}[mainthm]{Conjecture}

\theoremstyle{definition}
\newtheorem{defn}[thm]{Definition}

\theoremstyle{remark}

\newcommand{\sfX}{{\sf X}}
\newcommand{\sfx}{{\sf x}}

\newcommand{\sfb}{{\sf b}}
\newcommand{\sfa}{{\sf a}} 

\newcommand{\cT}{\mathcal{T}}
\newcommand{\cS}{\mathcal{S}}

\newcommand{\cusps}[1]{\text{\sc cusps}(#1)}

\newcommand{\graphs}{\text{\sc{graphs}}}

\renewcommand\gg{>\!\!>}
\renewcommand\ll{<\!\!<}

\begin{document}
\author{Laurent Bartholdi} 
\email{laurent.bartholdi@gmail.com}

\author{Dzmitry Dudko}
\email{dzmitry.dudko@gmail.com}

\author{Kevin M. Pilgrim}
\email{pilgrim@iu.edu}

\title[Correspondences on Riemann surfaces]{Correspondences on Riemann surfaces: (non-uniform) hyperbolicity and graph attractors}

\begin{abstract}
   We consider certain correspondences on a Riemann surface, and show that they admit a weak form of hyperbolicity: sufficiently long loops get shorter under lifting at a fixed point and closing.  In terms of their algebraic encoding by bisets, this translates to contraction of fundamental group elements along sequences arising from iterated lifting. 

   As an application, we show that apart from the usual Latt\`es counterexamples, for any rational map on $\Pone$ with $4$ post-critical points, there is a finite invariant collection of isotopy classes of curves into which every curve is attracted under iterated lifting.  More generally, among graphs of given complexity, there exists a finite invariant collection of isotopy classes of graphs into which every graph is attracted. Applied to sufficiently rich graphs, the graph attractor provides a finite set of topological normal forms for the rational map.

   We also present a strategy towards proving the same statements for maps with more than $4$ post-critical points.
 \end{abstract}



\date{compiled \today, modified 2025-VII-03}

\maketitle
\setcounter{tocdepth}{1}


\section*{Prologue}
Thanks to recent efforts, many key aspects of the theory of surface homeomorphisms have been extended to the non-invertible setting of post-critically finite (PCF) branched self-coverings of $S^2$, now known as \emph{Thurston maps}.  Common themes include obstructions to geometrization in the form of multicurves with certain invariance properties~\cite{MR1251582, farb:margalit:mcg, MR3071467}, actions on Teichmüller space~\cite{ Belk:2023aa, selinger:pullback}, on curve and tree complexes~\cite{MR4510015}, and on sets of graphs and related objects such as train tracks~\cite{zbMATH07239271, MR4068871}.  An analogous decomposition theory has been developed~\cite{MR2020454} with application to algorithmic aspects~\cite{MR3781419, MR3434502, MR3443373, MR4278340}.

The case of PCF hyperbolic polynomials is well understood, but stands apart from both homeomorphisms and general Thurston maps.  The forward-invariant \emph{Hubbard tree} of such a polynomial serves as a complete combinatorial invariant \cite{MR2650982}.  Under lifting by such a map, all curves (and all nearby structures like generic trees) converge exponentially fast to a finite invariant set of curves (and to a small combinatorial neighborhood of the Hubbard tree) \cite{MR4510015}; see also Remark~\eqref{rem:2} in~\S\ref{ss:thurston}. This is in contrast with both the case of typical homeomorphisms, for which curves become exponentially more complicated under iteration (though they do converge to an invariant train-track), and with many obstructed Thurston maps, which may have wandering or infinitely many cycles of curves.

In the setting of a PCF rational map $f$, a good substitute for train tracks of homeomorphisms, as well as Hubbard trees of polynomials, seems to be an invariant planar graph $\Gamma$ which is rich enough to allow reconstruction of $f$ from its restriction to $\Gamma$ and yet small enough so as to have minimal entropy among all such invariant graphs; see \cite{MR4423808, MR3774834, cui2024invariantgraphsjuliasets}.

The main objective of our paper is to develop a satisfactory similar theory beyond the polynomial case.  In this realm,  a central difficulty arises from the existence of obstructed twists $g_0fg_1$ of a general PCF rational map $f$, which implies that the mapping class biset of $f$ is not contracting in the sense of Nekrashevych, and which greatly complicates the analysis; see again Remark~\eqref{rem:2}.  Our work analyzes the tension between two opposing forces: the ``non-uniform contraction'' induced by lifting, and the ``additive correction'' required to ensure that the process we analyze is iterative, see~\eqref{eq:contract}. 

One of our main results, Theorem~\ref{thm:main X ray}, states that non-uniform contraction eventually dominates when $f$ has $4$ postcritical points. Its corollaries support the existence of minimal-entropy invariant graphs, but this is still open, even if the post-critical set has $3$ points. When $f$ has $4$ post-critical points, the corresponding pullback map on Teichm\"uller space is the lift to its universal cover of an analytic correspondence between Riemann surfaces. We therefore cast our arguments in this natural, more general setting.

\section{Introduction}\label{secn:intro}
An \emph{analytic covering self-correspondence} is a pair of maps
\[ F=(\phi, \rho \colon \cT\rightrightarrows\cS),\qquad\text{ simply written }\qquad F=\phi\circ \rho^{-1}\colon \cS \multimap \cS\]
between finite area hyperbolic Riemann surfaces, where $\phi\colon \cT\to\cS$ is a covering map, $\rho\colon  \cT\to\cS$ is any analytic map, and throughout this text `$\multimap$' denotes a multivalued map; see~\S\ref{secn:correspondences}.  The dynamics of such correspondences is the subject of much recent attention; see e.g.~\cites{MR1370680, MR4505392, MR1168195, MR4071411, MR4543152, MR4104599}. Of particular interest for us are examples arising as correspondences on moduli space induced by so-called Thurston maps with four postcritical points; see~\S\ref{subsec:thmap} and~\S\ref{secn:examples}.

If we endow $\cS$ with its hyperbolic metric, and pull it back via $\phi$ to $\cT$, then the map $\rho$ becomes distance non-increasing, by the Schwarz-Pick Lemma.  In many cases, $\rho$ strictly decreases distances --- that is, it is a contraction; this occurs if and only if $\rho$ is not a covering.  In this case, we call the correspondence \emph{admissible}; see~\S\ref{subsecn:admissible}.

It is natural to analyze a dynamical system through its effect on the fundamental group.  To this end, we suppose $F$ has a \emph{fixed point}, $\tilde{\star} \in \cT$.  This means 
\[ \star\coloneqq\phi (\tilde{\star})=\rho(\tilde{\star}).\]
Equivalently, and conveniently, a fixed point is an element $\star \in \cS$ together with a choice of preimage $\tilde{\star}\in \phi^{-1}(\star)$ for which $\rho(\tilde{\star})=\star$. We also fix a finite set
\begin{equation}\label{eq:genset}
  \sfX\subseteq\{\text{paths in $\cS$ that start at $\star$ and end at some element of $\rho(\phi^{-1}(\star))$}\}.
\end{equation}

We consider infinite sequences $g^{(0)}, g^{(1)}, g^{(2)}, \ldots$ that are orbits under iteration of the following one-to-finite multivalued map on the fundamental group $\pi_1(\cS, \star)$,
\begin{equation}\label{eq:X_ray}
g^{(n)} \multimap \rho\circ \phi^*(g^{(n)}).\overline{\sfx_n} \eqqcolon g^{(n+1)}.
\end{equation}
We call such an infinite sequence an \emph{$\sfX$-ray in the direction of $\star$ above $g^{(0)}$}; see Definition~\ref{def:xray}. Here and below, 
\begin{itemize}
\item $\phi^* (g^{(n)})\colon[0,1]\to\cT$ is the unique lift of $g^{(n)}$ under $\phi$ starting at the given lift $\tilde{\star}$ of the basepoint $\star$ as in the definition of fixed point; 
\item $\rho\circ \phi^*(g^{(n)})\colon[0,1]\to\cS$ is its image under $\rho$;
\item $\sfx_n\in \sfX$ is an arbitrary element of $\sfX$ which joins the basepoint $\star$ to the endpoint of $\rho\circ \phi^* (g^{(n)})$;
\item $\overline{\sfx_n}$ is the path $\sfx_n$ traversed in the opposite direction;
\item `$.$' denotes concatenation of paths.
\end{itemize}

If $M\geq \max\{d_\cS(\star, s) : s \in \rho(\phi^{-1}(\star))\}$ and $\sfX$ consists of all paths in $\cS$ that start at $\star$, end at some element of $\rho(\phi^{-1}(\star))$, and have length at most $M$, then $\sfX$ is necessarily finite, by properness of the hyperbolic metric, and any finite orbit $g^{(0)},\dots,g^{(n)}$ can be extended to an $\sfX$-ray.

For a path $\gamma$, we denote by $|\gamma|$ its hyperbolic length. This induces a proper norm $|g|$ on the group $\pi_1(\cS, \star)$.  Since $\rho$ is not a covering, then it is a contraction, so we have 
\[ |\rho\circ \phi^*(g^{(n)})|<|g^{(n)}|\]
and therefore
\begin{equation}\label{eq:contract}
  |g^{(n+1)}| \leq |\rho\circ \phi^*(g^{(n)}).\overline{\sfx_n}|<|g^{(n)}|+\max_{\sfx\in\sfX}|\sfx|.
\end{equation}
Thus, the  contraction induced by pulling back via $F$ is opposed by an additive correction of magnitude at most $\max_{\sfx\in\sfX}|\sfx|$ coming from concatenation with $\sfx\in \sfX$. Our main result is that the contraction eventually dominates:

\begin{mainthm}[Finite Attractor for $\sfX$-rays]\label{thm:main X ray}
  Assume that $F=\phi, \rho\colon \cT \rightrightarrows\cS$ is an admissible correspondence, and $\star \in \cS$ is a fixed point of $F$.  Then for every finite $\sfX$ as in~\eqref{eq:genset} there is a finite attractor $A(\sfX) \subset \pi_1(\cS,\star)$ such that for every $\sfX$-ray $g^{(0)}, g^{(1)}, g^{(2)}, \ldots$ in the direction of $\star$, the elements $g^{(n)}$ belong to $A(\sfX)$ for all $n$ sufficiently large.
\end{mainthm}

\subsection{Application to Thurston theory}\label{ss:thurston} See Theorem~\ref{thm:gga} for a detailed formulation. Let $f\colon (\rs, P)\righttoleftarrow$ be a rational map, with $P$ a finite subset of $\rs$ that contains the critical values of $f$ and satisfies $f(P)\subseteq P$, a \emph{post-critically finite} map.  By taking preimages, the map $f$ induces a pullback operation on the set of \emph{multicurves}: isotopy classes of collections of pairwise disjoint and non-isotopic essential curves in $\rs-P$. Somewhat similarly, $f$ induces a multivalued map on certain sets of trees whose vertices contain $P$, and on the set of graphs arising as \emph{spines} (graphs that are deformation retracts) for $\rs-P$. Under iteration:
\begin{itemize}
\item if $T_n$ is a tree on $\rs$ with $T_n \supset P$, then $T_{n+1}$ can be chosen to be any subtree of $f^{-1}(T_n)$, pruned so that the \emph{complexity}---defined as the sum of the number of edges and vertices---remains bounded under iteration;
\item if $\Gamma_n\subset \rs-P$ is a spine, then $\Gamma_{n+1}$ can be chosen to be any subspine of $f^{-1}(\Gamma_n)$. 
\end{itemize}

If $f$ is a homeomorphism of infinite order, then it is well known that $f^{-n}(\gamma)$, for an essential closed curve, will typically get more and more complicated as $|n|$ increases --- it converges to a \emph{lamination} on $\rs\setminus P$. If $f$ has degree $\ge2$ and is expanding (for example, rational), however, there is a tension between $f^{-1}$ being contracting on the one hand, and $f^{-1}(\gamma)$ consisting of, in the worst case, $\deg(f)$ branches of $\gamma$. Our main corollary, below, shows that the contraction of $f^{-1}$ wins, at least in some settings.

If $\#P=4$, and $f$ is rational and not a so-called Latt\`es example, then the correspondence $F=\phi,\rho$ on moduli space $\cS\coloneqq\mathscr M_{0,4}=\rs - \{0,1,\infty\}$ induced by $f$ is admissible. This allows us to deduce:
\begin{maincor}\label{cor:attractor}
  Assume $\#P=4$ and that $f$ is a non-Latt\`es rational map. Then there is a finite attractor for the pullback iteration on multicurves, on trees, and on spines: there is a finite collection of such objects such that every multicurve/tree/spine, upon pulling back sufficiently many times, reaches the attractor.
\end{maincor}

We sketch in~\S\ref{ss:higher dimension} how our approach could help generalizing Corollary~\ref{cor:attractor} to arbitrarily large $P$.

We note that we have only mentioned curves in this brief subsection, but that our arguments apply equally well to graphs of given complexity (say, given number of vertices and edges). This should have useful applications to the study of graphs of minimal entropy capturing the dynamics of $f$.

\subsection*{Remarks}
\begin{enumerate}
\item The existence of a finite attractor for curves, as in Corollary~\ref{cor:attractor}, was conjectured for general $P$ by the third author in the 2000s, see~\cite{kmp:tw:published}. Its extension to graphs is natural; unlike the pullback map on curves, the pullback map on e.g. spine graphs is always non-trivial.

\item The conjecture is known in the case of post-critically finite polynomials~\cite{MR3199801,MR4510015}. In fact, for hyperbolic polynomials Equation~\eqref{eq:contract} takes the stronger form
  \begin{equation}\label{eq:lambdacontraction}
    |g^{(n+1)}|\le\frac1\lambda|g^{(n)}|+\max_{x\in\sfX}|\sfx|\quad\text{for some }\lambda>1
  \end{equation}
  with respect to the word metric, and all our results easily follow from~\eqref{eq:lambdacontraction}. On the other hand, \eqref{eq:lambdacontraction} cannot hold if a map has an obstructed twist.\label{rem:2}

\item The failure of~\eqref{eq:lambdacontraction} in the general setting motivated the first two authors to put forth~\cite[Conjecture~7.10]{MR4213769}, which can be thought of as a version of Corollary~\ref{cor:attractor} with quantified number of steps to reach the attractor.
  
\item In the case of curves, Corollary~\ref{cor:attractor} is known in several cases; see~\cite{MR4472056} for a survey. 

\item For the case of trees, the pullback relation we consider is quite similar in spirit to the ``ivy iteration'' method of~\cite{MR4068871}; their implementation finds trees invariant-up-to-isotopy under a quadratic rational map $f$.  

\item A result of Hlushchanka shows that a critically fixed rational map is determined by the planar isomorphism class of its so-called charge graph; see~\cite{Hlushchanka:2019aa}. Hlushchanka and Prochorov show, using intersection theory applied to arcs on the sphere, that an induced  pullback iteration on graphs analogous to the ones considered here is shown to converge  to the charge graph; see~\cite{Hlushchanka:2022aa}.

\item The exclusion of Latt\`es examples is necessary --- for example, the flexible Latt\`es examples fix every curve.  

\item See~\cite{MR4423808} for a survey of recent results on finding invariant graphs; its main result asserts that given any critically finite rational map $f$, there is some iterate $n$ for which $f^n$ preserves exactly some finite graph containing the postcritical set of $f$. The authors note there that one cannot always specify the precise form of this graph; e.g.~\cite{MR4423808}*{Remark~3.3(2)} shows that for $f(z)=z^2+i$, there is no Jordan curve $C$ in the sphere containing the postcritical set of $f$ such that for some iterate $n$, one has $f^n(C)\subset C$.  

\item We note the following subtlety in Corollary~\ref{cor:attractor}, namely, that for trees, the size of the attractor necessarily grows upon passing to iterates. For example, suppose $f$ is a rational map with finite post-critical set $P$ and Julia set the whole sphere, and $T\subset \rs$ is any tree with vertex set containing $P$.  Taking preimages, one obtains for each $n\in \N$ a cell structure on $\rs$.  Since $f$ is uniformly expanding with respect to a suitable orbifold length metric, the diameters of the two-cells tend to zero as $n\to\infty$. It follows that for some sufficiently large iterate $n$ depending on the choice of $T$, the tree $T$ is invariant up to isotopy under pullback via $f^n$; compare~\cite{MR3727134}*{\S15}.  Corollary~\ref{cor:attractor} implies that at least in the case $\#P=4$, only finitely many such trees $T$  are contained in periodic orbits under pullback by $f$ itself. Similar results apply to a rational map $f$ with Sierpi\'nski carpet Julia set; one first collapses its Fatou components to points to obtain a map which is expanding on the whole sphere; compare~\cite{MR3774834}.

\item We conjecture that Theorem~\ref{thm:main X ray} and Corollary~\ref{cor:attractor} hold for all post-critically finite non-Latt\`es maps; see Conjecture~\ref{conj:main}. In~\S\ref{ss:higher dimension}, we present a strategy towards its proof.
\end{enumerate}

For $P$ of general cardinality, the connection to group theory arises as follows; see~\cite{kmp:tw:published}*{\S7}. Suppose that $f$ is an arbitrary Thurston map with postcritical set $P$.  Since the set of isotopy classes of curves has finitely many orbits under the action of the pure mapping class group $G\coloneqq\text{PMod}(S^2,P)$, a sequence $C_0, C_1, C_2, \ldots$ of curves arising as an orbit under pullback by $f$ can be encoded by a a sequence of pairs $(C^{(0)}, g^{(0)}), (C^{(1)}, g^{(1)}),  (C^{(2)}, g^{(2)}), \ldots$ where the $C^{(n)}$'s are drawn from a finite set of models comprising a transversal to the action of $G$, and the $g^{(n)}$'s lie in $G$.  Thus control of the orbit of $C_0$ under taking iterated  pullbacks boils down to controlling the sequence $g^{(0)}, g^{(1)}, g^{(2)}, \ldots$, which we show forms an $\sfX$-ray.  

\begin{proof}[Outline of proof of Corollary~\ref{cor:attractor}]
  The first step is to note that the map $f$ induces an admissible correspondence on the moduli space $\cS\coloneqq\mathscr M_{0,4}=\rs - \{0,1,\infty\}$, with a distinguished fixed point $\star$ determined by $f$ since $f$ is assumed rational.

  The elements $g^{(n)}$ comprising the terms of an $\sfX$-ray are then loops in $\cS$ based at $\star$.  

  The property of being an $\sfX$-ray can be abstracted in the setting of a \emph{biset}, which is an algebraic object encoding the correspondence aforementionned.

  This biset appears in two guises, one in which elements are paths (with concatenation of paths as operation), and one in which they are Thurston maps (with composition of maps as operation); these two points of view correspond to the avatars of the fundamental group as a set of paths or a set of deck transformations.

  Working in the realm of paths, we apply Theorem~\ref{thm:main X ray} to conclude the proof.
\end{proof}

\subsection{Outline of the proof of Theorem~\ref{thm:main X ray}}\label{ss:outline}
Our main result is proven as a consequence of the following:
\begin{mainthm}\label{thm:main}
  Suppose $F=\phi, \rho\colon \cT \rightrightarrows\cS$ is an admissible correspondence, and $\star\in \cS$ is a fixed point of $F$.  Let $\star$ also denote the constant path at $\star$.  Fix a finite collection $\sfX$ of curves as in~\eqref{eq:genset}. Then there are constants $\kappa>0,N\in\N,\epsilon>0,\xi>0$ such that for every $g\in \pi_1(\cS, \star)$ and every $\sfX$-ray of loops $(g^{(n)})$ in the direction of $\star$ above $g$, we have
  \begin{enumerate}
  \item $|g^{(n+1)}|\le|g^{(n)}|+\xi$ for all $n$;
   \item if $|g|>\kappa$ then $|g^{(n)}|\le |g|-\epsilon$ for some $n\in\{0,\dots,N\}$.
  \end{enumerate}
\end{mainthm}

Note that Conclusion~(1) follows easily from the definition of $\sfX$-ray and is stated only for reference, while (2) is not.  From this it will follow (Corollary~\ref{cor:contraction}) that there is a finite attractor $A\subset G$, depending on $\star$ and $\sfX$ but not on $g$, such that the terms $g^{(n)}$ of any $\sfX$-ray $g^{(0)}, g^{(1)}, g^{(2)}, \ldots$ in the direction of $\star$ eventually lie in $A$. 

Let us now outline the proof of Theorem~\ref{thm:main}. Throughout this subsection, we think of $n$ as a finite, very large number (so as to dwarf any universal constant $\mathcal O(1)$), and consider $m\to\infty$.

The proof proceeds by contradiction. Let us assume that Theorem~\ref{thm:main} does not hold. We begin by a ``soft'' preparation of the sequence of curves violating Theorem~\ref{thm:main}: there is a \emph{tight sequence} $(g_m)_m$ in $G\coloneqq  \pi_1(\cS, \star)$ with
\begin{equation}
\label{ed:dfn:tight}
 |g_m|>m\text{ and }|g_m^{(n)}|>|g_m|-1/m\qquad\text{for all }n\in\{0,\dots,m\}.
\end{equation}
In other words, for a sufficiently large $m\gg 1$, the lift $g_m \leadsto g^{(n)}_{m}$ of $g_m$ does not lose any length. Consider then the \emph{thin-thick decomposition} of $g_m$,
\begin{equation}\label{eq:g_m:decomp}
 g_m = \ell_{0,m}.r_{1,m}\ldots r_{k_m,m}. \ell_{k_m,m},
\end{equation}
in which the $r_{i,m}$, called \emph{roundabouts}, are in the thin part of $\cS$ while the $\ell_{i,m}$ are in the thick part.

Since $(g_m)_m$ is a tight sequence and $F^{-1}$ is uniformly contracting on the thick pieces $\ell_{i,m}$, they can neither be long nor numerous; in other words, we may assume $k_m$ is uniformly bounded, and $|\ell_{i,m}|=\mathcal O(1)$ for all $i$; see Proposition~\ref{prop:tightfacts}. We obtain
\begin{equation}\label{eq:pre:ignore thick}
  |g_m|=\sum_{i=1}^{k_m} |r_{i,m}|+\mathcal O(1), \qquad \sum_{i=1}^{k_m}  |r_{i,m}| \to \infty.
\end{equation}
We let $k((g_m))$ be the maximal number of roundabouts in $(g_m)_m$ whose lengths increase to infinity (with respect to all subsequences $m_t\to\infty$), see~\eqref{eq:k_g}. We then define $K<\infty$ as the maximum over all $k((g_m))$. 

Let us next impose an additional combinatorial assumption that $(g_m)_m$ is \emph{very tight}, see Definition~\ref{defn:verytight}: \emph{the number of roundabouts in $(g_m)_m$ whose lengths increase to infinity equals $K$}. For very tight sequences,~\eqref{eq:pre:ignore thick} takes the following form: by ignoring ``short'' roundabouts and for all sufficiently large $m\gg 1$, we have
\begin{equation}\label{eq:ignore thick}
  |g_m|=\sum_{i=1}^{K} |r_{i,m}|+\mathcal O(1), \qquad |r_{i,m}| \to \infty\quad  \text{for all $i$},
\end{equation}
see Lemma~\ref{lemma:ldvt}.

Since $\cS$ has dimension $1$, every roundabout $r_{i,m}$ lies in a small neighborhood of a puncture $x_{i,m}$. Denoting by $\sphericalangle  r_{i,m}$ the \emph{winding number} of $r_{i,m}$ around $x_{i,m}$, we have the following ``tropical'' \emph{log-formula}, see Lemma~\ref{lem:lengthest}:
\begin{equation}\label{eq:LogFormula}
  |r_{i,m}|=2\log_+|\sphericalangle  r_{i,m}| + \mathcal O(1). 
\end{equation} 

Let us analyze what happens under lifting by $F^{-n}$: above each puncture $x_{i,m}$, the map $F^n$ is given locally by a Puiseux series $z \mapsto c\cdot z^{t^{(n)}_{i,m}}+\cdots$. 
We consider the thick-thin decompositions of the lifts $g_m^{(n)}$. Since $(g_m)_m$ is very tight, the thin-thick decomposition persists under lifting, see~\S\ref{subsecn:liftingunder}, and we may write each $g_m^{(n)}$ in thin-thick form
\begin{equation}\label{eq:persistence}
  g_m^{(n)}=\ell_{0,m}^{(n)}.r_{1,m}^{(n)} \ldots r_{k_m,m}^{(n)}.\ell_{K,m}^{(n)},
\end{equation}
with every roundabout $r_{i,m}^{(n)}$ an $F^n$-preimage of its corresponding roundabout $r_{i,m}$. Equation~\eqref{eq:LogFormula} then gives:
\begin{equation}\label{eq:before_extraction}
\begin{aligned}
  |g_m^{(n)}|&=\sum_{i=0}^K|\ell_{i,m}^{(n)}|+\sum_{i=1}^K2\log_+|t_{i,m}^{(n)}\sphericalangle r_{i,m}|+\mathcal O(1)\\
  &=|g_m| + \sum_{i=0}^K|\ell_{i,m}^{(n)}| + \sum_{i=1}^K2\log t_{i,m}^{(n)}+\mathcal O(1)
\end{aligned}
\end{equation}
for every fixed $n$ and all sufficiently large $m \gg 1$, see Lemma~\ref{lem:tightlift}.
 
In Lemma~\ref{lemma:extraction} and its Corollary~\ref{cor:m_n}, we develop an \emph{extraction argument} to show that the total length $\sum_{i=0}^K|\ell_{i,m}^{(n)}|$ is bounded over sufficiently large $m$ and ``many'' $n$. Roughly, if one of the $\ell_{i,m}^{(n)}$ had unbounded length, then it would develop a new roundabout whose length increases to infinity, contradicting the choice of $K$. Consequently, there are sequences $(m(s)),(n(s))\to \infty$ with $n(s)\ll m(s)$ such that  
\begin{equation}\label{eq:extraction}
  \begin{aligned}
    |g_{m(s)}^{(n(s))}|&=|g_{m(s)}|+\sum_{i=1}^K2\log t_{i,{m(s)}}^{(n(s))}+\mathcal O(1)\\
    &=|g_{m(s)}|+2\log\Big(t_{1,{m(s)}}^{(n(s))}\cdots t_{K,{m(s)}}^{(n(s))}\Big)+\mathcal O(1).
  \end{aligned}
\end{equation}

Past this preparation step, we arrive at the heart of the argument, requiring precise analytic estimates. We claim (see Lemma~\ref{lem:tdown}) that the local degrees $t_{i,m}^{(n)}$ satisfy
\begin{equation}\label{eqn:tdecay}
  \begin{aligned}
    t^{(n)}_{1,m}&=\mathcal O(\exp(-\eta n)),\\
    t^{(n)}_{i,m}t^{(n)}_{i+1,m}&=\mathcal O(\exp(-\eta n))\text{ for all }i\in\{1,\dots,K-1\},
  \end{aligned}
\end{equation}
for some $\eta>0$. Consider indeed two consecutive roundabouts, and the part of $g_m$ between them. We compare that curve to the (infinite) geodesic connecting the punctures $x_{i,m}$ and $x_{i+1,m}$, see the curve $p$ on Figure~\ref{fig:homotopy} for a visualization. Because of its thick part, each lift of $p$ under $F$ loses a definite amount (proportional to $\eta$) of length, so its lift under $F^n$ shrinks by a linear function of $n$. This translates back to the claimed estimate on local degrees of $F^n$.

Multiplying an appropriate subset (depending on the parity of $K$) of the inequalities in~\eqref{eqn:tdecay}, we arrive at $t^{(n)}_{1,m}\cdots t^{(n)}_{K,m}=\mathcal O(\exp(-\eta n))$ so~\eqref{eq:extraction} implies $|g_m^{(n)}|\le |g_m| - 2\eta n + \mathcal O(1)$. This is our desired contradiction.

The key point of the argument is that the linear loss of length $-\eta n$ in~\eqref{eqn:tdecay} dominates the constant error term $\mathcal O(1)$ in~\eqref{eq:extraction} coming from the thick parts $\ell_{i,m}^{(n)}$ and the log-formula~\eqref{eq:LogFormula}.

\subsection{Strategy towards an extension to higher dimension}\label{ss:higher dimension}
We believe that Theorem~\ref{thm:main X ray} extends to some higher-dimensional correspondences of interest; most significantly to correspondences controlling various algebraic properties of postcritically finite rational maps:
\begin{mainconj}\label{conj:main}
  Let $f\colon(S^2,P)\righttoleftarrow$ be a Thurston map, and let $F$ be its associated modular correspondence on $\mathscr M_{0,d}=\Moduli(S^2, P)$, as in~\eqref{eq:modcorr}. Then Theorem~\ref{thm:main X ray} holds with respect to any fixed point $\star$ of $F$.
\end{mainconj}
This conjecture implies Corollary~\ref{cor:attractor} for the rational map associated with $\star$.

We elaborate below on the strategy towards proving Conjecture~\ref{conj:main}, based on the arguments outlined in~\S\ref{ss:outline}. We expect that the setting of tight and very tight sequences together with the Extraction Argument is applicable, see~\eqref{eq:strategy:g}; a key challenge is to establish a variant of~\eqref{eqn:tdecay}. We anticipate that an appropriate decomposition theory for generalized roundabouts would resolve the involved difficulties.

Let us consider a tight sequence $(g_m)_m$ in $G\coloneqq  \pi_1(\mathscr M_{0,d}, \star)$ as in~\eqref{ed:dfn:tight}; we write $g_{m}\colon [0,1]\to \mathscr M_{0,d}$. For every $s\in[0,1]$, the point $g_{m}(s)$ represents a marked Riemann sphere $\tau_{m,s}=(S^2,P_{g_m(s)})$ and we denote by  $\widetilde{\mathscr C}_{m,s}$ the multicurve (possibly empty) representing the Thin-Thick decomposition of $\tau_{m,s}$.

Since $g_m(s)$ is a tight sequence, lifting  via $F^{-n}\colon g_m(s)\leadsto g_m^{(n)}(s)$ almost preserves the infinitesimal length of $g_m$ for most $s$. This infinitesimal length is described by the quadratic differential $q_{m,s}\equiv q_{m,s} (z)dz^2$ on $\tau_{m,s}$ dual to $\dot g_{m}$; the induced map $F^n\colon \tau_{m,s}\to\tau^{(n)}_{m,s}$ should almost preserve $q_{m,s}$. It is shown in~\cite{HSS} that $q_{m,s}$ can be decomposed relatively to the multicurve $\widetilde{\mathscr C}_{m,s}$ on $\tau_{m,s}$. It is a \emph{meta-principle} that, in this decomposition, most of the mass of $q_{m,s}$ should be either within fat annuli (thin parts of the decomposition) or within the small spheres with a non-hyperbolic orbifold (namely, on which the induced map has degree $1$ or is doubly covered by a torus endomorphism). We expect that there is a well-defined minimal submulticurve $\mathscr C_{m,s}\subset \widetilde{\mathscr C}_{m,s}$ such that the decomposition of  $(\tau_{m,s}, {\mathscr C}_{m,s})$ combines all small spheres of  $(\tau_{m,s}, \widetilde{\mathscr C}_{m,s})$ with a non-parabolic orbifold that ``can be combined''; compare with the Levy multicurve from~\cite{BD4}. Finally, we \emph{conjecture} that there is a decomposition of $g_m$ into \emph{generalized roundabouts} $r_{i,m}$
\begin{equation}
\label{eq:strategy:g}
g_m=\ell_{0,m}. r_{1,m}\dots r_{K,m}.\ell_{K,m}\hspace{2cm}\text{similar to~\eqref{eq:g_m:decomp},}
\end{equation}
where $r_{i,m}$ are maximal subcurves of the $g_{m}$ so that $\mathscr C_{m,s}$ is non-empty and depends continuously (i.e., is stable) on $s$ within $r_{i,m}$. The curves $\ell_{i,m}$ may or may not be within a compact subpart of $\mathscr M_{0,d}$; however, we expect that the $F^n\colon g_m\leadsto g_m^{(n)}$ have a definite contraction at the $\ell_{i,m}$, i.e., a non-quantified variant of~\eqref{eqn:tdecay} holds at the $\ell_{i,m}$. Namely, for $s$ in $\ell_{i,m}$, the multicurve ${\mathscr C}_{m,s}$ changes within small spheres with a hyperbolic orbifold and thus $\ell_{i,m}$ is either within a compact part of  $\mathscr M_{0,d}$ or is close to its strata (a certain boundary region of the Weil-Petersson compactification) where $F^n$ has definite contraction. An extension of $F$ to the Weil-Petersson compactification follows from~\cite{selinger:pullback}. (The above definition of $r_{i,m}$ should be taken as guidance and may be subject to necessary adjustments.)

Assuming~\eqref{eq:strategy:g}, a variant of~\eqref{eq:before_extraction} and~\eqref{eq:extraction} is provided by Minsky's Product Region Theorem~\cite{MR1390649} as follows. Let us suppose that the decomposition of $(\tau_{i,m}, \mathscr C_{i,m})$ is represented by ``small'' punctured spheres $S_1,\dots,S_{k'}$ connected by ``fat'' annuli $A_1,\dots,A_{k''}$ in a tree-like fashion. Let us consider the projections $r'_{i,m,j}$ and $r''_{i,m,j}$ of $r_{i,m}$ to  to the factors corresponding to the small spheres and wide annuli respectively; then
\begin{equation}\label{eq:minsky}
  |r_{i,m}|=\max\Big[\max_{j=1}^{k'}|r'_{i,m,j}|,\ \  \max_{j=1}^{k''}|r''_{i,m,j}|\Big]+\mathcal O(1).
\end{equation}
Note that the lengths $|r''_{i,m,j}|$ are given by Fenchel-Nielsen coordinates, and are naturally measured in copies of $\mathbb H$. If the maximum is realized by the projections $r''_i$ to annuli, then a log-formula equivalent to~\eqref{eq:LogFormula} applies. The local degrees $t_i^{(n)}$ are to be replaced by Thurston matrices $\Theta^{(n)}_{i,m}$ recording the degrees by which the different annuli map to each other. This should lead to the estimate similar to~\eqref{eq:extraction}:
\[|g^{(n)}_{i,m}|=|g_{i,m}|+\sum_{i=1}^{K} 2\log\|\Theta_{i,m}^{(n)}\|_{i,m}+\mathcal O(1),\]
where the norm $\|\Theta_{i,m}^{(n)}\|_{i,m}$ comes from~\eqref{eq:minsky}. We expect that an analogue of Estimates~\eqref{eqn:tdecay} can be derived following the lines of Lemma~\ref{lem:tdown}; this would provide the desired conclusion.

As it seems to us, a key challenge is when the maximum is realized by the projections $r'_i$ in small spheres. We anticipate that a reduction to a lower-dimensional problem can be achieved, by better understanding the $r'_i$, the interactions between them and their neighbours along $g$, and the back-and-forth between parameter and dynamical spaces. The theory of decompositions, already well developed in~\cites{MR2020454,MR4213769}, should prove helpful.

\subsection{Table of contents}
\begin{enumerate}
\item[\S\ref{secn:correspondences}] gives background on correspondences, with a subsection on bisets.
\item[\S\ref{subsec:thmap}] describes how Thurston maps induce correspondences on moduli space, with examples in~\S\ref{secn:examples}.
\item[\S\ref{secn:hyperbolic}] develops needed background in hyperbolic geometry and applies this to correspondences on Riemann surfaces.
\item[\S\ref{secn:xrays}] introduces $\sfX$-rays.
\item[\S\ref{secn:proof}] proves Theorem~\ref{thm:main}.
\item[\S\ref{ss:gga}] shows how certain pullback orbits of graphs may be modeled by $\sfX$-rays; it concludes with the proof of Corollary~\ref{cor:attractor}.
\end{enumerate}

\subsection{Conventions and notation}
\begin{itemize}
\item Absolute constants are written $\mathcal O(1)$, those that are functions of $n$ but of no other variable are written $\mathcal O_n(1)$, and constants that vanish as $n\to\infty$ are written $o(1)_{n\to\infty}$. Quantities $A,B$ whose ratio $A/B$ is bounded as $n\to\infty$ are written $A\ll_n B$, or just $A\ll B$ if the variable tending to infinity is different from $n$ and obvious from the context.
\item Composition of functions $f\circ g$ means $g$ is applied first. 
\item Concatenation of paths is denoted $\alpha.\beta$, with $\alpha$ traversed first. 
\item A smooth map $f$ between Riemannian manifolds is a \emph{contraction} if $\|Df\|<1$ everywhere, where $\|Df\|$ is the operator norm of its derivative.  A smooth map is a \emph{uniform contraction} on a subset $K$ if for some $0 \leq c < 1$ we have $\|Df\|<c<1$ at all points of $K$. In particular, the restriction of a contraction to a compact subset is a uniform contraction.
\item For a rectifiable path $\gamma$ in a hyperbolic surface, we denote by $|\gamma|$ its hyperbolic length; and if $g$ denotes a homotopy class of paths then $|g|$ denotes the minimal hyperbolic length of a representative of $g$. A path obtained from $\gamma$ by removing bounded segments at its extremities is written $\gamma-\mathcal O(1)$.
\item Isotopy of maps, and in particular of paths, is written $\simeq$. Usually the context makes it obvious relative to what the isotopies should be considered.
\end{itemize}

\subsection{Acknowledgments} We are grateful to Mikhail Hlushchanka and Curtis McMullen for valuable feedback. 

Dzmitry Dudko is supported by NSF grant DMS-2055532.
Laurent Bartholdi is supported by SNSF grant TMAG-2\_216487/1.
Kevin M. Pilgrim is supported by Simons Collaboration Grant 615022 and NSF grant DMS-2302907, and  also acknowledges support from the Institute for Mathematical Sciences at SUNY Stony Brook and the Max-Planck Instit\"ut f\"ur Mathematik-Bonn.  All authors acknowledge support from the ERC AdG grant 101097307 and Saarland University, which supported scientific visits during which much of this article was developed.

\section{Correspondences}\label{secn:correspondences}

\subsection{Terminology}
\begin{enumerate}
\item A \emph{topological self-correspondence} is a pair of continuous maps $\phi, \rho\colon \cT \rightrightarrows\cS$ of topological spaces.  We think of it as a multivalued map $F\colon \cS \multimap \cS$ given by $F=\phi \circ \rho^{-1}$. The $n$th iterate of $F$ in this ``forward'' direction will be denoted by $F^{(n)}$.  
\item A \emph{covering correspondence} is a topological self-correspondence for which $\phi$ is a finite covering map. 
\end{enumerate}

A remark about directionality is in order. For a covering correspondence $F$, local branches of $F^{-1}$ exist, so it might seem more natural to define $F$ going in the other direction. Indeed, there are interesting examples where $\rho$ is constant, making $\rho^{-1}$ rather pathological~\cite{MR2508269}. However, the simplest examples are when $\rho$ is injective, such as in the case of the Rabbit Example in~\S\ref{secn:examples} below.  Also, with this convention of direction, $F$ has the path-lifting property. 

\begin{enumerate}
\setcounter{enumi}{2}
\item If $\phi$ and $\rho$ are both continuous, both analytic, etc. we use the adjective  \emph{topological}, \emph{analytic}, etc.\ to describe the correspondence.
\item A \emph{backward orbit} of $s\in\cS$ is a sequence $t_0, t_1,\ldots,t_n\in \cT$ with $\phi(t_0)=s$ and $\rho(t_i)=\phi(t_{i+1})$ for all $i=0,\dots,n-1$.  Abusing notation, suppressing the fact that the $t_n$'s are part of the definition, it is a sequence $s_0, s_1,\ldots,s_n\in \cS$ with $F(s_{i+1})=s_i$ for all $i=0,\dots,n-1$.
\item The \emph{backward orbit} $F^{-\N}(s_0)$ of $s_0\in \cS$ is the set of all backward orbits of $s_0$.
\item A \emph{fixed point} of a self-correspondence is a point $\tilde{\star}\in \cT$ such that $\phi(\tilde{\star})=\rho(\tilde{\star})$.  Equivalently, a fixed point is  a point $\star\in \cS$ together with a choice of $\tilde{\star}\in \phi^{-1}(\star)$ satisfying $\rho(\tilde{\star})=\star$. We again abuse notation and refer to $\star$ as a fixed point of the self-correspondence if such a $\tilde\star$ exists, with the understanding that a chosen lift $\tilde{\star}$ is part of the data. 
\end{enumerate}

If $\phi, \rho\colon \cT \rightrightarrows\cS$ is a covering correspondence, then any path $\gamma$ in $\cS$ can be lifted under $\phi$, uniquely once a preimage of $\gamma(0)$ under $\phi$ has been chosen; and then mapped by $\rho$ to obtain another path in $\cS$. If a choice of preimage of $\gamma(0)$ is available from the context, we refer to this new path as $F^*(\gamma)$.

\subsection{Admissible complex correspondences}\label{subsecn:admissible}
Our main results concern analytic correspondences on Riemann surfaces satisfying a type of ``critical finiteness'' property that is strictly weaker than that defined in~\cite{MR1168195}.  In our setting, cusps also play a distinguished role. Thus in order to distinguish our setting from that of others, we encapsulate these properties into the notion of an ``admissible complex correspondence'', which we now define precisely. Suppose that
\begin{enumerate}
\item $\widehat{\cT}, \widehat{\cS}$ are compact Riemann surfaces (and in particular are connected);
\item $\cusps\cT, \cusps\cS$ are finite subsets of $\widehat{\cT}$ and $\widehat{\cS}$, respectively;
\item $\cT\coloneqq\widehat{\cT}\setminus \cusps\cT$, $\cS\coloneqq\widehat{\cS}\setminus \cusps\cS$, so that now the sets $\cusps{\cT}, \cusps{\cS}$ become the ends of $\cT$ and $\cS$, respectively;
\item $\cT$ and $\cS$ are hyperbolic Riemann surfaces; 
\item $\phi\colon \cT\to\cS$ is a analytic covering map of finite degree $D$ larger than one; 
\item $\rho\colon \cT\to\cS$ is analytic and is not a covering map.
\end{enumerate}

We refer to a correspondence $\phi, \rho\colon \cT \rightrightarrows\cS$ satisfying (1)-(6) above as an \emph{admissible complex correspondence}.  

\subsection{Properties}
An admissible complex  correspondence satisfies several properties.
\begin{enumerate}
\item The maps $\phi, \rho$ extend over their sets of ends to analytic maps (which we denote by the same symbols) $\phi, \rho\colon \widehat{\cT}\to\widehat{\cS}$, so $\phi(\cusps{\cT}) \subseteq \cusps{\cS}$ and $\rho^{-1}(\cusps{\cS}) \subseteq \cusps{\cT}$.
\item The multivalued map $F\coloneqq\phi \circ \rho^{-1}$ has the property that its \emph{forward post-critical set} 
\[\bigcup_{n \geq 0} F^{(n)}(\{\text{branch values of }\phi\})\]
is contained in the (sometimes larger) finite set $\cusps{\cS}$.
\item  the map $\rho\colon \cT\to\cS$ is a    contraction: if we give $\cT,\cS$ the hyperbolic metric with constant curvature $-1$, then $\|d\rho(t)\|_{\cT, \cS}<1$ at all points $t\in \cT$; in particular, every periodic point of the multivalued map $F\colon \cS \multimap \cS$ given by $F=\phi\circ\rho^{-1}$ is repelling. Thus, each admissible complex correspondence is \emph{weakly hyperbolic}, in the sense that $\|d\rho(t)\|_{\cT, \cS}<1$ at all points $t\in \cT$.

\item We have $\deg(\rho) < \deg(\phi)$, by the following argument due to Walter Parry. By removing a set of area zero from $\cS$, we may find a simply-connected subset $U\subseteq\cS$ above which $\rho$ is a trivial covering with preimages $U_i$ for $i=1,\ldots,\deg(\rho)$. Then
  \begin{align*}
    \deg(\phi)\text{area}(\cS)&=\text{area}(\cT)=\text{area}(U_1)+\cdots +\text{area}(U_{\deg(\rho)})\\&>\deg(\rho)\text{area}(U)
    =\deg(\rho)\text{area}(\cS),
  \end{align*}
  since $\rho$ is a strict contraction.
  
\item Combined with~\cite{MR3667901}*{Theorem 5.3}, the previous Property (4) implies that the number of points of period $n$ grows exponentially with $n$. 

\item Given $s_0\in \cS$, the accumulation set $\text{acc}(F^{-\N}(s_0)) \subset \overline{\cS}$ of its    backward orbit (defined as the set of all points $s\in \overline{\cS}$ such that there exists a backward orbit of $s_0$ accumulating at $s$) is nonempty, by compactness of $\overline{\cS}$.   It is also independent of $s_0$.

  To see this, let $s_0,s_1,\dots$ be a backward orbit of $s$. Suppose first $s\in \text{acc}(F^{-\N}(s_0)) \cap \cS$.  Pick some other $s_0'\in \cS$.  Join $s_0$ to $s_0'$ by a path of $\cS $-hyperbolic length say $L$.  By weak hyperbolicity and path-lifting, there is a backward orbit $s_0', s_1', s_2', \ldots$ such that the hyperbolic distances between $s_n, s_n'$ satisfy $|s_n-s_n'|<L$ for each $n$. If some subsequence $s_{n_k}\to s\in \cS$ then since $\|dF^{-1}\|_{\cS, \cS} < \lambda < 1$ in e.g.\ a $2L$-neighborhood of $s$ we have $|s_{n_k}-s_{n_k}'|\to 0$. If some subsequence $s_{n_k}\to s\in \cusps{\cS}$ then we also have $s_{n_k'}\to s$ since hyperbolic balls about $s_{n_k}$ of fixed radius $2L$ become smaller and smaller in any compatible metric on $\overline{\cS}$ as $s_{n_k}\to s$. A similar argument treats the case when $s\in \cusps{\cS}$.

Thus the \emph{limit set} $\Lambda(\phi, \rho\colon \cT \rightrightarrows\cS) \subset \overline{\cS}$, defined as $\text{acc}(F^{-\N}(s_0))$ for some (equivalently, any) $s_0\in \cS$ is well-defined, and backward-invariant in the sense that $F^{-1}(\Lambda)=\Lambda$.  

\item Measure-theoretically, the dynamics of an admissible complex correspondence is similar to that of a rational function.\footnote{We thank M. Londhe for useful conversations on this topic.}  It is known~\cite{MR4104599} that for each choice of $s_0$ outside a polar (``small'', in particular, of Lebesgue measure zero) set of exceptional points, the sequence of atomic measures $\mu_n\coloneqq D^{-n}\sum_{a\in F^{(-n)}(s_0)}\delta_a$ converges to a measure $\mu$.  The measure $\mu$ is independent of the choice of $s_0$, is supported on the closure of the limit set as defined above, and assigns mass zero to any polar subset of $\cS$.

In particular, the limit set is a nonempty, uncountable subset of $\cS$.  

\item An analytic correspondence can naturally be iterated: via functorial pullbacks, for each $n\in \N$ there is a Riemann surface $\cT^{(n)}$,  a finite cover $\phi^{(n)}\colon \cT^{(n)}\to\cS$,  and a analytic map $\rho^{(n)}\colon \cT^{(n)}\to\cS$.  We obtain in this manner a multivalued map $F^{(n)}\colon \cS \multimap \cS$ given locally by $\phi^{(n)}\circ (\rho^{(n)})^{-1}$.  Every fixed point for $F$ yields a fixed point for $F^{(n)}$. 

\item Most of the examples that we are aware of are of the following type: $\cS=\rs\setminus\{0,1,\infty\}$, and $\cT$ is a finite-type surface. There are then finitely many admissible correspondences of given degree, and they are determined by group-theoretical data (the deck group of $\phi$ and the map induced by $\rho$ on fundamental groups). See for example the following correspondence~\cite{MR4213769}*{Figure~4}:
  \[\begin{tikzpicture}
  \gdef\pathSR{(160:3) .. controls (180:2.6) .. (200:3)}
  \gdef\pathRS{(-40:3) .. controls (-15:2) and (15:2) .. (40:3)}
  \gdef\surfacelabels{
    \foreach \a/\x/\y/\z/\t in {0/B/C/F/D,120/A/B/D/E,240/C/A/E/F} {
      \draw[rotate=\a] node at (-20:2.5) {$\x$};
      \draw[rotate=\a] node at (20:2.5) {$\y$};
      \draw[rotate=\a] node at (172:3.0) {$\z$};
      \draw[rotate=\a] node at (188:3.0) {$\t$};
    }
  }
  \begin{scope}
  \surfacelabels
  \gdef\surfacecoords{
  \foreach \a/\i in {0/0,120/1,240/2} {
    \begin{scope}[rotate=\a]
      \path[name path=v0] \pathSR coordinate (R\i);
      \path[name path=w0] (0,0) -- (-3,0);

      \path[name path=v1] \pathRS coordinate (S\i);
      \path[name path=w1] (0,0) -- (3,0);
      \path[name intersections={of=v0 and w0,by=p0},name intersections={of=v1 and w1,by=p1}] (p0) coordinate (P\i) -- (p1) coordinate (Q\i);

      \path[name path=w2] (0,0) -- (30:3);
      \path[name intersections={of=v1 and w2,by=p2}] (p2) coordinate (T\i);
      \path[name path=w2] (0,0) -- (-30:3);
      \path[name intersections={of=v1 and w2,by=p2}] (p2) coordinate (U\i);
    \end{scope}
  }
  \foreach \a/\i/\j in {0/0/1,120/1/2,240/2/0} {
    \begin{scope}[rotate=\a]
      \gdef\pathTU{(T\i) .. controls (60:1.8) .. (U\j)}
      \path[name path=v0] \pathTU;
      \path[name path=w0] (0,0) -- (60:3);
      \path[name intersections={of=v0 and w0,by=p0}] (p0) coordinate (V\i);
    \end{scope}
  }  
  }
  \surfacecoords
  \gdef\surfaceleft{
  \foreach \a/\i/\j/\k in {0/0/1/2,120/1/2/0,240/2/0/1} {
    \begin{scope}[rotate=\a]
      \gdef\pathQV{(Q\i) .. controls (30:1.5) .. (V\i)}
      \clip (0,0) -- \pathQV -- cycle;
      \fill[gray!20] (0,0) -- (T\i) -- (V\i) -- cycle;
    \end{scope}
    \begin{scope}[rotate=\a]
      \gdef\pathVQ{(V\i) .. controls (90:1.5) .. (Q\j)}
      \clip (0,0) -- \pathVQ -- cycle;
      \fill[gray!20] (0,0) -- (U\j) -- (Q\j) -- cycle;
    \end{scope}
    \begin{scope}[rotate=\a]
      \clip (0,0) -- \pathRS -- cycle;
      \clip (S\i) -- \pathQV -- cycle;
      \fill[gray!20] (0,0) -- (T\i) -- (Q\i) -- cycle;
    \end{scope}
    \begin{scope}[rotate=\a]
      \gdef\pathVTa{(V\i) .. controls (50:1.8) and (40:1.7) .. (T\i)}
      \clip (0,0) -- \pathTU -- cycle;
      \fill[gray!20] (S\i) -- \pathVTa -- cycle;
    \end{scope}
    \begin{scope}[rotate=\a]
      \gdef\pathVUa{(V\i) .. controls (70:1.8) and (80:1.7) .. (U\j)}
      \clip (0,0) -- \pathVUa -- cycle;
      \fill[gray!20] \pathVQ -- (U\j) -- cycle;
    \end{scope}
    \begin{scope}[rotate=\a]
      \gdef\pathVUb{(V\i) .. controls (75:2.3) .. (U\j)}
      \clip (0,0) -- \pathVUb -- cycle;
      \fill[gray!20] (P\k) -- \pathTU -- cycle;
    \end{scope}
    \begin{scope}[rotate=\a]
      \gdef\pathVTb{(V\i) .. controls (45:2.3) .. (T\i)}
      \clip (P\k) -- \pathVTb -- cycle;
      \fill[gray!20] (0,0) -- (V\i) -- (S\i) -- (Q\i) -- cycle;
    \end{scope}
    \begin{scope}[rotate=\a]
      \clip (V\j) -- \pathSR -- cycle;
      \fill[gray!20] (T\j) -- (P\i) -- (S\j) -- cycle;
    \end{scope}
    \begin{scope}[rotate=\a]
      \clip (V\j) -- \pathSR -- cycle;
      \fill[gray!20] (V\j) -- (P\i) -- (U\k) -- cycle;
    \end{scope}
    \begin{scope}[rotate=\a]
      \clip (V\k) -- \pathRS -- cycle;
      \fill[gray!20] (U\i) -- (P\j) -- (R\j) -- cycle;
    \end{scope}
  }
  \foreach \a/\i/\j/\k/\c in {0/0/1/2/,120/1/2/0/,240/2/0/1/} {
    \begin{scope}[rotate=\a]
      \draw \pathRS;
      \draw \pathSR;
      \draw (P\i) -- (Q\i);
      \draw (0,0) -- \pathTU -- cycle;
      \draw \pathVUa;
      \draw \pathVUb;
      \draw \pathVTa;
      \draw \pathVTb;
      \draw \pathVQ;
      \draw \pathQV;
      \draw (V\i) -- (R\k);
      \draw (V\i) -- (S\i);
      \draw (U\j) -- (P\k);
      \draw (T\i) -- (P\k);
      \fill (42:2.1) circle (0.06);
      \fill [\c] (78:2.1) circle (0.06);
    \end{scope}
  }
  \node at (0.4,0.1) {\small $\infty$};

  \node[right] at (U0) {\small $\infty$};
  \node[right] at (Q0) {\small $0$};
  \node[right] at (T0) {\small $\infty$};
  \node[above right=-1mm] at (S0) {\small $0$};
  \node[above] at (P2) {\small $\infty$};
  \node[above left=-1mm] at (R2) {\small $0$};
  \node[above left=-1mm] at (U1) {\small $\infty$};
  \node[above left] at (Q1) {\small $0$};
  \node[above] at (T1) {\small $\infty$};
  \node[left] at (S1) {\small $0$};
  \node[left] at (P0) {\small $\infty$};
  \node[left] at (R0) {\small $0$};
  \node[below] at (U2) {\small $\infty$};
  \node[below left] at (Q2) {\small $0$};
  \node[below left=-1mm] at (T2) {\small $\infty$};
  \node[below left=-1mm] at (S2) {\small $0$};
  \node[below] at (P1) {\small $\infty$};
  \node[below right=-1mm] at (R1) {\small $0$};
  }
  \surfaceleft
\end{scope}

\begin{scope}[shift={(6.5,0)}]
  \surfacecoords
  \foreach \a/\i/\j/\k in {0/0/1/2,120/1/2/0,240/2/0/1} {
    \begin{scope}[rotate=\a]
      \clip (0,0) -- \pathTU -- cycle;
      \fill[gray!20] (S\i) -- \pathVTa -- cycle;
    \end{scope}
    \begin{scope}[rotate=\a]
      \clip (0,0) -- \pathTU -- cycle;
      \fill[gray!20] (R\k) -- \pathVUa -- cycle;
    \end{scope}
    \begin{scope}[rotate=\a]
      \gdef\pathVTd{(V\i) .. controls (48:2.6) .. (T\i)}
      \clip (0,0) -- \pathVTd -- cycle;
      \fill[gray!20] (P\k) -- \pathVTb -- (S\i) -- cycle;
    \end{scope}
    \begin{scope}[rotate=\a]
      \gdef\pathVUd{(V\i) .. controls (72:2.6) .. (U\j)}
      \clip (0,0) -- \pathVUd -- cycle;
      \fill[gray!20] (P\k) -- \pathVUb -- (R\k) -- cycle;
    \end{scope}
  }
  \gdef\pathVUc{(V\i) .. controls (95:1.4) .. (U\j)}
  \gdef\pathVTc{(V\i) .. controls (25:1.4) .. (T\i)}
  \begin{scope}
    \foreach \a/\i/\j/\k in {0/0/1/2,120/1/2/0,240/2/0/1} {
      \clip[rotate=\a] \pathVUc -- (T\j) -- (U\k) -- (T\k) -- (U\i) -- (T\i) -- cycle;
      \clip[rotate=\a] \pathVTc -- (U\i) -- (T\k) -- (U\k) -- (T\j) -- (U\j) -- cycle;
      \clip[rotate=\a] \pathRS -- (R\k) -- (S\j) -- (R\i) -- (S\k) -- cycle;
    }
    \fill[gray!20] (0,0) circle (3);
  \end{scope}
  \foreach \a/\i/\j/\k/\c in {0/0/1/2/,120/1/2/0/,240/2/0/1/} {
    \begin{scope}[rotate=\a]
      \draw \pathRS;
      \draw \pathSR;
      \draw \pathTU;
      \draw \pathVUa;
      \draw \pathVUb;
      \draw \pathVUc;
      \draw \pathVUd;
      \draw \pathVTa;
      \draw \pathVTb;
      \draw \pathVTc;
      \draw \pathVTd;
      \fill (42:2.1) circle (0.06);
      \fill [\c] (78:2.1) circle (0.06);
      \path[name path=v0] \pathVQ; \path[name path=w0] (0,0) -- (U\j);
      \fill[name intersections={of=v0 and w0,by=p0}] (p0) circle (0.06);
      \path[name path=v0] \pathQV; \path[name path=w0] (0,0) -- (T\i);
      \fill[name intersections={of=v0 and w0,by=p0}] (p0) circle (0.06);
      \path[name path=v0] (U\j) -- (P\k); \path[name path=w0] (R\k) -- (V\i);
      \fill[name intersections={of=v0 and w0,by=p0}] (p0) circle (0.06);
      \path[name path=v0] (T\i) -- (P\k); \path[name path=w0] (S\i) -- (V\i);
      \fill[name intersections={of=v0 and w0,by=p0}] (p0) circle (0.06);
      \node at (Q\i) {$\times$};
      \node at (P\i) {$\times$};
      \node at (S\i) {$\times$};
      \node at (R\i) {$\times$};
      \node at (0:1.1) {$\times$};
    \end{scope}
  }
  \node at (0,0) {$\times$};
  \node[above right] at (0,0) {$\overline{\zeta_6}$};

  \node[right] at (U0) {\small $\infty$};
  \node[right] at (Q0) {$\overline{\zeta_6}$};
  \node[right] at (T0) {\small $1$};
  \node[above right=-1mm] at (S0) {$\zeta_6$};
  \node[above] at (P2) {$\zeta_6$};
  \node[above left=-1mm] at (R2) {$\zeta_6$};
  \node[above left=-1mm] at (U1) {\small $0$};
  \node[above left] at (Q1) {$\overline{\zeta_6}$};
  \node[above] at (T1) {\small $\infty$};
  \node[left] at (S1) {$\zeta_6$};
  \node[left] at (P0) {$\zeta_6$};
  \node[left] at (R0) {$\zeta_6$};
  \node[below] at (U2) {\small $1$};
  \node[below left] at (Q2) {$\overline{\zeta_6}$};
  \node[below left=-1mm] at (T2) {\small $0$};
  \node[below left=-1mm] at (S2) {$\zeta_6$};
  \node[below] at (P1) {$\zeta_6$};
  \node[below right=-1mm] at (R1) {$\zeta_6$};

  \node[above right] at (25:1.17) {\small $0$};
  \node[above right] at (95:1.17) {\small $1$};
  \node[left] at (145:1.17) {\small $1$};
  \node[left] at (215:1.17) {\small $\infty$};
  \node[below right] at (265:1.17) {\small $\infty$};
  \node[below right] at (335:1.17) {\small $0$};
  \node[above right] at (V0) {\small $\infty$};
  \node[left=1mm] at (V1) {\small $0$};
  \node[below right] at (V2) {\small $1$};
\end{scope}

\begin{scope}[shift={(3.8,-3.5)}]
  \begin{scope}
    \clip (-1,0) rectangle (1,-1);
    \fill[gray!20] (0,0) circle (1);
  \end{scope}
  \draw (0,0) circle (1);
  \draw (-1,0) -- (1,0);
  \fill (-0.7,0) circle (0.06) node[above] {$0$};
  \fill (0,0) circle (0.06) node[above] {$1$};
  \fill (0.7,0) circle (0.06) node[above] {$\infty$};
\end{scope}

\draw[<->,very thick] (2.4,0) -- (3.3,0);
\draw[->,very thick] (1.7,-2.6) -- node[below left=-1mm] {$30:1$} (2.5,-3.2);
\draw[->,very thick] (5.7,-2.2) -- node[below right=-1mm] {$16:1$} (4.9,-3.0);

\end{tikzpicture}
\]
  A wealth of high-genus examples appear in the context of nearly Euclidean maps, see~\S\ref{secn:examples}.\label{item:8}
  
\item All correspondences in the setting of~\eqref{item:8} can be composed, by taking fibre products, and also lead in this manner to arbitrarily high genus examples.
  
\end{enumerate}

\subsection{Dynamical regularity} 
Suppose $\phi, \rho\colon \cT \rightrightarrows\cS$ is an admissible complex correspondence. By definition, it satisfies the weak hyperbolicity property.  If the limit set $\Lambda$ is a nonempty compact subset of $\cS$, we say that $\phi, \rho\colon \cT \rightrightarrows\cS$ is \emph{uniformly hyperbolic}.  Equivalently: there is a nonempty compact subset $K \subset \cS$ for which $F^{-1}(K) \subset K$. Between the weak and uniform hyperbolic correspondences, we single out the \emph{strongly subhyperbolic} ones: those for which there exists a complete length metric on the compact space $\overline{\cS}$ such that $F^{-1}$ uniformly contracts lengths of curves. 

\subsection{Bisets}\label{subsecn:introbisets}
In this section, we recast the notion of $\sfX$-ray using the natural, but perhaps less familiar to dynamicists, algebraic language of bisets associated to correspondences.

We draw heavily from~\cite{MR3781419}. For groups $G,H$, an \emph{$H$-$G$-biset} is a set equipped with a left $H$-action and a right $G$-action that commute.  Bisets generalize the homomorphism on fundamental group induced by a continuous map: given pointed path-connected topological spaces $(\cT,\star')$ and $(\cS,\star)$, any map $\psi\colon \cT\to\cS$, \emph{that need not preserve basepoints}, gives rise to a $\pi_1(\cT,\star')$-$\pi_1(\cS,\star)$-biset
\[B(\psi)=\{\gamma\colon[0,1]\to\cS:\gamma(0)=\psi(\star'),\gamma(1)=\star\}/{\simeq},\]
where `$\simeq$' denotes homotopy with respect to the endpoints. The actions are $[\beta]\cdot[\gamma]\cdot[\alpha]=[(\psi\circ\beta).\gamma.\alpha]$, for based loops $\alpha,\beta$ in $(\cS,\star),(\cT,\star')$ respectively.

We shall not recall the general definition of the biset of a correspondence $F=\phi,\rho:\cT\rightrightarrows\cS$, but restrict to the case in which $\phi$ is a covering, when (see~\cite{MR3781419}*{Lemma~4.4}) it may be defined by
\[B(F)=\{(\gamma\colon[0,1]\to\cS,p'\in \cT):\gamma(0)=\phi(p')=\star,\gamma(1)=\rho(p')\}/{\simeq},\]
again with `$\simeq$' denoting homotopy that preserves the relations indicated between $\gamma$ and $p'$. It is a $G$-$G$-biset for the group $G=\pi_1(\cS,\star)$. The left action is by pre-catenation, and the right action of $[\alpha]$ on $[\gamma,p']$ is computing by lifting $\alpha$ to a path $\widetilde\alpha\coloneqq\phi^*(\alpha)$ starting at $p'$ and setting $[\gamma,p']\cdot[\alpha]=[\gamma.(\rho\circ\widetilde\alpha),\widetilde\alpha(1)]$.

In particular, $B(F)$ is \emph{left-free}: there is a subset $\sfX$ of $B(F)$ such that, when only considering the left action, $B(F)\cong G\times \sfX$.  We call such $\sfX$ a \emph{basis} for the biset; it is of course not unique. We have $|\sfX|=\deg(\phi)$, and $\sfX$ will be of the form $\sfX=\{(\sfx_{p'},p'):p'\in \phi^{-1}(\star)\}$ for some choice of paths $\sfx_{p'}$ from $\star$ to $\rho(p')$ for every possible $p'\in \phi^{-1}(\star)$.

Given an $H$-$G$-biset $B$ and a $K$-$H$-biset $C$, their \emph{composition} is the biset
\[C\otimes_H B=\frac{C\times B}{(c h,b)=(c,h b)\;\forall c\in C,h\in H,b\in B},\]
The assignment of bisets to topological correspondences is functorial under composition, hence iteration. If $B$ is a $G-G$ biset, so are the products $B^{\otimes n}$ for $n\in \N$.  The element $\sfb_1 \otimes \cdots \otimes \sfb_n\in B^{\otimes n}$ is represented by a concatenation of paths $\undertilde{b}_1.\undertilde{b}_2.\ldots .\undertilde{b}_n$ of paths, with $\undertilde{b}_1$ based at $\star$, where $F^{(i-1)}(\undertilde{b}_i)=\sfb_i\in B$, $i=1, \ldots, n$, and where each $\undertilde{b}_i$ starts at the endpoint of $\undertilde{b}_{i-1}$.  The actions of $G$ are again by pre-concatenation (which is free), and by lifting under $F^{(n)}$ and post-concatenation. If $\sfX$ is a basis of $B$ then $\sfX^n$, the set of words of length $n$ in the alphabet $\sfX$, is a basis for $B^{\otimes n}$. 

\subsection{Examples: post-critically finite rational maps}\label{subsecn:pcf}
A post-critically finite rational map $f\colon(\rs,P)\righttoleftarrow$ may be viewed as an admissible complex correspondence, with $\cT=\rs\setminus f^{-1}(P)$ and $\cS=\rs\setminus P$ and $\phi,\rho\colon\cT\rightrightarrows\cS$ induced respectively by $f$ and the identity. Our main result covers their dynamics. However, our proof of Corollary~\ref{cor:attractor} relies on an interplay between the dynamical plane of such a correspondence and an associated correspondence on moduli space. This is the main topic of~\S\ref{subsec:thmap}.  

\section{Thurston maps and correspondences on moduli space}\label{subsec:thmap}
Fix a subset $P \subset S^2$ with $\#P\geq 4$, and identify $S^2$ with the Riemann sphere $\rs$.  Here, the restriction on the cardinality is to eliminate mention of uninteresting special cases.  Suppose $f\colon S^2\to S^2$ is a non-Latt\`es Thurston map with postcritical set contained in $P$ where $f(P) \subset P$. The moduli space $\Moduli(S^2, P)$ is defined to be the set of injections $\iota\colon P \hookrightarrow \rs$ modulo the action of $\Aut(\rs)$ by post-composition; it comes with a basepoint $\star$ represented by the inclusion $P \hookrightarrow \rs$. Teichm\"uller space $\Teich(S^2,P)$ is defined to be the universal cover of moduli space; it too comes with a basepoint $\star$ represented by the identity map $S^2\to\rs$.  Equivalently, it is the space of marked conformal structures on $(S^2, P)$ up to isotopy relative to $P$.  The pure mapping class group $\text{PMod}(S^2, P)$ is the group of deck transformations of the universal covering map $\pi\colon\Teich(S^2,P)\to\Moduli(S^2,P)$ and is canonically identified with the fundamental group $\pi_1(\Moduli(S^2, P),\star)$.  We set $G\coloneqq\Mod(S^2, P)=\pi_1(\Moduli(S^2, P),\star)$. 

Via pullback of complex structures, $f$ induces a self-map $\sigma_f\colon \Teich(S^2, P)\to\Teich(S^2,P)$ of Teichm\"uller space, which lies above an algebraic covering self-correspondence of moduli space in the sense that we have the following diagram 
\begin{equation}\label{eq:modcorr}
  \begin{tikzcd}
  \Teich(S^2,P) \arrow[dd,"\pi"]\arrow[rr,"\sigma_f"] \arrow[dr,"\omega"] & & 
  \Teich(S^2, P) \arrow[dd,"\pi"] \\
  &\mathcal{T}\arrow[dl,"\phi" above] \arrow[dr,"\rho"]\\
  \Moduli(S^2, P)=\cS & & \cS=\Moduli(S^2, P). 
\end{tikzcd}
\end{equation}
The intermediate cover $\cT$ is $\Teich(S^2,P)/H_f$, for the subgroup $H_f<G$ of index $\deg(\phi)$ defined (see~\cite{kmp:kps}) as
\begin{equation}\label{eq:H_f}
  H_f=\{h | \exists \tilde{h}\in G, h\circ f \simeq f\circ \tilde{h}\}.
\end{equation}
We obtain a correspondence $\phi, \rho\colon \cT \rightrightarrows\cS$ with $\phi,\rho$ defined respectively by factoring $\pi=\phi\circ\omega$ and noting that $\pi\circ\sigma_f$ descends to a map $\rho$ on $\cT$. Furthermore, following the construction, we note that the correspondence associated with $f$ is isomorphic to the correspondence associated with $\eta_1\circ f\circ \eta_0$, for any homeomorphisms $\eta_0,\eta_1\in\text{Homeo}(S^2,P)$.

This is not a coincidence. Recall from~\S\ref{subsecn:introbisets} that the correspondence $\phi,\rho\colon\cT\rightrightarrows\cS$ may be encoded by a set $B=B^{\text{paths}}$ endowed with two commuting actions of $G=\pi_1(\cS,\star)$. The group $G$ is naturally identified with the modular group $\text{PMod}(S^2, P)$. Viewing $G$ as a group of isotopy classes of self-maps of $(S^2,P)$, we define
\[ B^{\text{maps}}\coloneqq\{\eta_1 \circ f \circ \eta_0 | \eta_0, \eta_1\in\text{Homeo}(S^2,P)\}/\text{isotopy relative to $P$},\]
and note that $B^{\text{maps}}$ is naturally a $G$-$G$-biset, with actions induced by pre- and post-composition. The action by pre-composition is free with $\deg(\phi)$ orbits. The identification $\Mod(S^2, P)=G=\pi_1(\Moduli(S^2,P))$ yields an identification $B^{\text{maps}}=B^{\text{paths}}$, see~\cite{MR4213769}*{Theorem~9.1}, as long as we act on $B^{\text{maps}}$ in \emph{algebraic} order: $g_0\cdot f\cdot g_1$ is represented by $\eta_1\circ f\circ\eta_0$, for any representatives $\eta_i$ of $g_i$:
\begin{prop}\label{prop:identify}
  The $(G,G)$-bisets $B^{\text{maps}}$ and $B^{\text{paths}}$ are canonically isomorphic.\qed
\end{prop}

For a Thurston map $f$, conjugacy up to isotopy relative to $P$ is identified with conjugacy in the biset: the Thurston map represented by $g \cdot f \cdot g^{-1}$ is conjugate up to isotopy to $f$ via $g$. 

It follows that the map $f$ is conjugate up to isotopy to a rational function if and only if $\sigma_f$ has a fixed point $\theta_f$ in $\Teich(S^2, P)$, and, equivalently, if the correspondence above has a fixed point $\tilde{\star}\coloneqq\omega(\theta_f)\in\cT$ with $\phi(\tilde\star)=\rho(\tilde\star)=\star\in\cS$.

\subsection{Examples of maps with four postcritical points}\label{secn:examples}
Let us detail some maps with $\#P=4$.  In this setting, we may identify $\Teich(S^2, P)$ with the upper-half plane $\IH$, the moduli space with a triply-punctured sphere, e.g.\ $\cS\coloneqq\rs \setminus \{0,1,\infty\}$, and $\pi$ with the well-known modular function so that its deck group becomes the principal congruence subgroup $\Gamma(2)\coloneqq\{ A\in \mathrm{PSL}_2(\R) : A \equiv 1 \pmod 2\}$ acting on $\IH$.

\subsubsection*{Rabbit polynomial}  When $f$ is the so-called \emph{rabbit} complex polynomial $f(z)=z^2+c$, whose critical point $z=0$ is periodic of period $3$ and $\Im(z)>0$, after suitable normalizations the correspondence becomes $F(x)=1-\frac{1}{x^2}$, a single-valued post-critically finite hyperbolic rational function with three postcritical points lying in a common superattracting three-cycle $0 \mapsto\infty \mapsto 1 \mapsto 0$. The correspondence on moduli space is therefore uniformly hyperbolic. This example is studied in detail in~\cite{bartholdi:nekrashevych:twisted}.

\subsubsection*{Dendrite polynomial} When $f$ is the so-called \emph{dendrite} complex polynomial $f(z)=z^2+i$, whose critical point $z=0$ is preperiodic with preperiod $1$ and period $2$, after suitable normalizations the correspondence becomes $F(x)=(-1+2/x)^2$, a post-critically finite rational function with Julia set equal to the whole sphere. The orbifold of this map in the sense of Douady-Hubbard~\cite{MR1251582} turns out to be Euclidean with signature $(2, 4, 4)$. The corresponding Euclidean length orbifold metric on $\cS$ shows that this correspondence on moduli space is strongly subhyperbolic, and thus uniformly expanding on the entire completion of moduli space with respect to this metric. This is also studied in detail in~\cite{bartholdi:nekrashevych:twisted}.

\subsubsection*{Lodge's example} When $f(z)=\frac{3z^2}{2z^3+1}$ is the example featured in the thesis of R. Lodge~\cites{MR3063048, MR2508269}, the correspondence on moduli space may be described as follows. Let $\omega\coloneqq\exp(2\pi i /3)$ denote a cube root of unity.  We have  $\cS\coloneqq\rs \setminus \{1, \omega, \overline{\omega}\}$, $\cT\coloneqq\rs \setminus \{\pm 1, \pm \omega, \pm \overline{\omega}\}$, $\phi(t)=t(t^3+2)/(2t^3+1)$, and $\rho(t)=t^2$. 

This example has several interesting features. The set of ends $\cusps{\cS}$ is totally invariant: $F^{-1}(\cusps\cS) = \cusps{\cS}$.  The maps $\phi$ and $\rho$ are surjective and are Galois branched covers; $s=0$ is a super-attracting fixed point of $F^{-1}$.  A direct computation shows that there are two branches of $F$ at the fixed-cusp $s=1$: one attracting, and one repelling. Thus this correspondence is likely not subhyperbolic with respect to any straightforward generalization of such a definition from single-valued maps to correspondences. 

\subsubsection*{Nearly Euclidean Thurston maps}  This family of maps provides a wealth of examples, though few with explicit formulas. When $\deg(f)=5$ there are examples where the genus of $\cT$ is equal to $1$; when $\deg(f)=7$ one finds examples of genus $4$, etc.; the genera are observed to grow as the degree increases. See~\cite{kmp:origami} for a survey.  

\subsubsection*{Critically fixed maps} Thurston maps each of whose critical points are fixed have been completely classified~\cites{kmp:fixed, Hlushchanka:2022aa, Hlushchanka:2019aa}. The correspondences on moduli space depend only on the set of local degrees at the critical points.  When they have four postcritical points, their admissible complex correspondences on moduli space are particularly tractable. We present two examples below. 

\subsubsection*{Critically fixed quintic with local degrees $2, 3, 3, 4$}  Up to planar isomorphism any planar graph with four vertices of valences $1, 2, 2, 3$ is connected.  There are two possibilities: (a) a triangle together with one additional edge joining a vertex of the triangle to a vertex of valence one, and (b) a segment of length three, with a non-middle edge doubled.  Given either of these graphs, there is up to analytic conjugacy a unique degree five critically fixed rational map $f$ obtained by ``blowing up'' its edges.
\[\begin{tikzpicture}[every node/.style={inner sep=0mm}]
    \node (a1) at (0:1) {$\bullet$};
    \node (a2) at (120:1) {$\bullet$};
    \node (a3) at (240:1) {$\bullet$};
    \node (a4) at (0:2.7) {$\bullet$};
    \draw (a4) -- (a1) -- (a2) -- (a3) -- (a1);
    \node at (1.4,-0.3) {(a)};
    
    \node (b1) at (4,0) {$\bullet$};
    \node (b2) at (5.7,0) {$\bullet$};
    \node (b3) at (7.4,0) {$\bullet$};
    \node (b4) at (9.1,0) {$\bullet$};
    \draw (b1) edge[bend left=20] (b2) edge[bend right=20] (b2)
    (b2) -- (b3) -- (b4);
    \node at (6.55,-0.3) {(b)};
  \end{tikzpicture}\]
Simple algebra (aided, obviously, by machine) yields the following correspondence on moduli space as the one induced by $f$.  Set $\widehat{\cT}=\widehat{\cS}=\widehat{\C}$, and define
\begin{align*} \phi(t) &=\frac{(t-3)^2 (5 t+3)^4 (5 t^2+18 t-3)}{331776 t^3},\\
  \rho(t)&=\frac{-5 t^2+12 t+9}{24 t},
\end{align*}
reaching $\cS=\widehat{\C}\setminus\{0,1,\infty\}$ from $\cT=\phi^{-1}(\cS)=\widehat{\C}\setminus\{0,\infty,\pm 3/5, \pm 3, \frac{-9\pm \sqrt{6}}{5}, \frac{9\pm 4\sqrt{6}}{5}\}$. The quadratic map $\rho$ has critical points at $\pm 3i/\sqrt{5}$, so is not a cover, and so this correspondence is admissible. The multivalued map $F\coloneqq\phi \circ \rho^{-1}$ has the property that $y \in F^{-1}(y)$ for each $y\in \{0,1,\infty\}$. Moreover, each branch of $F$ at such a fixed cusp $y$ is super-attracting.  For example, in the forward direction of $F$, we have $\rho^{-1}(\infty)=\{0,\infty\}$ both with local degree one, while $\phi$ has degree $3$ at $0$ and $5$ at $\infty$; in the case of the fixed cusps $0$ and $1$, the branches of $F$ in the forward direction have local degrees respectively $2$ and $4$. 

By B\"ottcher's theorem, the forward dynamics of $F$ near each point $0,1,\infty$ is analytically conjugate to $w \mapsto w^k$ near the origin for some $k \geq 2$.  It follows that the complement $K$ of sufficiently small open disks centered around $0,1,\infty$, round in B\"ottcher coordinates, has the property that $F^{-1}(K) \subset K$. 

This implies that this correspondence is \emph{uniformly hyperbolic}, see Definition~\ref{defn:uniformly hyperbolic}: its limit set is contained in $K$, hence is compact.

 \subsubsection*{Critically fixed cubic with local degrees $2,2,2,2$} Let $f$ be the cubic critically fixed Thurston map obtained by blowing up the edges of a graph which is the disjoint union of two edges. No twist $g_1\circ f\circ g_0$ of $f$ can be rational, by the analytic Fixed Point Index Formula~\cite{milnor:dynamics}.  This observation suggests that the admissible complex correspondence associated to $f$ need not have a fixed point.  We now verify this by explicit calculation. 

Define a correspondence with $\widehat{\cT}=\widehat{\cS}=\widehat{\C}$ as follows. Set 
\begin{align*}
  \rho(t)&=\frac{-1+2t+3t^2}{4t},\\
  \phi(t)&=\frac{(1+t)(-1+3t)^3}{16t},
\end{align*}
reaching $\cS=\widehat{\C}\setminus\{0,1,\infty\}$ from $\cT=\phi^{-1}(\cS)=\widehat{\C}\setminus\{0,1,\infty,\pm 1/3, \pm 1\}$. The quadratic map $\rho$ has critical points at $\pm i/\sqrt{3}$, so is not a cover, and so this correspondence is admissible. The multivalued map $F=\phi\circ \rho^{-1}$ has a single-valued branch near the fixed point at infinity whose dynamical germ is, by B\"ottcher's theorem, analytically conjugate to $w \mapsto w^3$ near the origin.  However, direct calculation shows that the only fixed points are at the ideal points $0$ and $1$, which do not lie in $\cS$. 

There is a repelling $2$-cycle of $F$ at $\pm \sqrt{5}/3$, with multiplier $9/4$. 

\section{Hyperbolic geometry}\label{secn:hyperbolic}
In this section, we assume that $\cS$ is a hyperbolic Riemann surface which is conformally isomorphic to the complement of a finite nonempty subset $\cusps{\cS}$ contained in a compact surface $\widehat{\cS}$, and $\cS$ is equipped with its hyperbolic metric of curvature $-1$. Choose a basepoint $\star\in \cS$, and write $G\coloneqq\pi_1(\cS, \star)$.  Fix a universal cover $\pi\colon (\IH, \tilde{\star})\to (\cS, \star)$. For a rectifiable path $\gamma$ in $\cS$, we recall that $|\gamma|$ denotes its hyperbolic length.

\subsection{Hyperbolic metric on $G$}
Consider $g\in G$. The geodesic joining $\tilde{\star}$ to $g \cdot \tilde{\star}$ projects to a closed curve on $\cS$ which, upon removing $\star$, is an open geodesic segment.  While a closed curve, we emphasize that this is typically not a closed geodesic; there is typically an angle formed at $\star$.  Abusing notation, we will often denote this geodesic representative by the same symbol, $g$.  We denote by $|g|$ the hyperbolic length of this segment.  This defines the \emph{hyperbolic norm} on $G$.  Since the covering group acts properly discontinuously, the hyperbolic norm on $G$ is proper: for any constant $L>0$, the set of $g\in G$ with $|g|<L$ is finite. 

Caution is in order: though we have used the term ``norm'', the hyperbolic norm on $G$ is not bi-Lipschitz equivalent to a word metric norm induced by a finite generating set. For example, suppose $g$ is a primitive peripheral loop about some cusp of $\cS$, and consider the sequence of powers, $g^n$.  The word norms of $g^n$ grow linearly, while the hyperbolic norms of $g^n$ grow logarithmically.

\subsection{Cusp neighborhoods}\label{subsecn:nbhds}
Consider a cusp $x$ of $\cS$.  For a cusp circumference parameter $\delta>0$ sufficiently small, there is a neighborhood $B(x,\delta)$ of $x$ isometric to $\{z\in\IH:\Im(z)>1/\delta\}/\langle z \mapsto z+1\rangle$.  Note that the length of the bounding horocyclic curve $\partial B(x,\delta)$ is the parameter $\delta$. 

\subsection{$\delta$-thick-thin decompositions}\label{subsecn:deltathickthin}
We next formulate a version of a thick-thin decomposition of a Riemann surface, adapted to our setting. We are not interested in short simple closed geodesics, so our formulation focuses on cusps. 
\begin{prop}[$(\delta,\zeta)$-thick-thin decomposition for pointed surfaces]\label{prop:thickthinsurface}
There exists a \emph{cusp circumference parameter} $0<\delta<1$ small, and a \emph{separation parameter} $\zeta\gg1$ large, with the following properties.
\begin{enumerate}
\item for each $x\in \cusps{\cS}$, there is a cusp neighbourhood $B_x\coloneqq B(x,\delta)$ of $x$ isometric to $\{z\in\IH:\Im(z)>1/\delta\}/\langle z \mapsto z+1\rangle$, and the collection $\{B_x: x\in \cusps{\cS}\}$ have pairwise disjoint closures;
\item for each $x\in \cusps{\cS}$, the $\zeta$-neighborhood of $B_x$ is again a cusp neighborhood; 
\item for any two distinct $x_1, x_2\in \cusps{S}$, the closed $\zeta$-neighborhoods of $B_{x_1}, B_{x_2}$ are disjoint; 
\item for each $x\in \cusps{\cS}$, the distance from $B_x$ to the basepoint $\star$ is at least $\zeta$; 
\item $\delta$ is much shorter than the length of the smallest closed geodesic on $\cS$ (this isn't really necessary, but it helps for fixing intuition and justifying the terminology). 
\end{enumerate}
\end{prop}

The proof is elementary, and we omit it. 

\begin{prop}[$(\delta,\zeta)$-thick-thin decomposition for pointed loops]\label{prop:thickthinloop}
Given the setup of Proposition~\ref{prop:thickthinsurface}, for each $g\in G$ there exist a possibly empty sequence of cusps $x_1, x_2, \ldots, x_k$ such that the geodesic representative of $g$ decomposes uniquely into a concatenation of geodesic segments
\[g=\ell_0.r_1.\ell_1\ldots r_k.\ell_k,\]
with the $r_i$ maximal segments of $g$ entirely contained in the cusp neighborhoods $B(x_i,\delta)$, and the $\ell_i$ in their complement. In particular, $|\ell_i|\ge\zeta$ for $0\le i\le k$.  
\end{prop}

The last conclusion holds because for each $x\in \cusps{\cS}$, the cusp neighborhood $B(x,\delta)$ is a convex subset of its $\zeta$-neighborhood. 

We call the $\ell_i$ \emph{thick segments} since they lie in the thick part.  We call the $r_i$ \emph{roundabouts} since when $|r_i|$ is large (the case in which we will be most interested, later), they wind around the corresponding cusp.  

\subsection{Winding numbers}\label{subsecn:windingnumbers}
We continue, assuming that we are in the setup of~\S\ref{subsecn:deltathickthin}.  Consider $g\in G$ and its thick-thin decomposition $g=\ell_0.r_1.\ell_1\ldots r_k.\ell_k$.

A cusp neighborhood $B(x,\delta)$ is convex; a geodesic $r$ joining two points on its boundary winds around the cusp some number of times.  To define this winding number formally, consider the universal cover $\pi\colon \IH\to\cS$ that maps $\sqrt{-1}\infty$ to the cusp, and let $H_\delta\coloneqq\{\Im(\tau)>1/\delta\}$ be the corresponding horoball as in Proposition~\ref{prop:thickthinsurface}. A lift of $r$ enters $H_\delta$ at a point $a+\sqrt{-1}/\delta$, $a\in \R$, and then exits $H_\delta$ at some other point $a+\sphericalangle(r)+\sqrt{-1}/\delta$.  We call $\sphericalangle(r)$ the \emph{winding number} of $r$.  Defined in this way, the winding number coincides with the classical winding number of $r$ about $x$ from elementary complex analysis.

Turning to the thick-thin decomposition, each roundabout $r_i$  enters and exits once the cusp neighborhood $B(x_i,\delta)$, and so has an associated winding number $\sphericalangle(r_i)$.  
\begin{lem}[Winding numbers]\label{lem:lengthest} 
  Assume the setup of Propositions~\ref{prop:thickthinsurface} and~\ref{prop:thickthinloop}.
\begin{enumerate}
\item For any geodesic segment $r$ contained in a cusp neighborhood $B(x,\delta)$ and with endpoints on $\partial B(x,\delta)$, 
\[ |r|=2\log_+|\sphericalangle r| + \mathcal O(1)\]
where the implicit constant depends only on $\delta$.
\item 
 For any $g\in G$ with $(\delta,\zeta)$-thick-thin decomposition $g=\ell_0.r_1.\ell_1\ldots r_k.\ell_k$, the length of its geodesic representative may be estimated as 
  \[|g|=\sum_{i=0}^k|\ell_i|+2\sum_{i=1}^k\log_+|\sphericalangle(r_i)|+\mathcal O(k)\]
  where the implicit constant depends only on $\delta$.
  \end{enumerate}
\end{lem}
\begin{proof}
  It suffices to note that the length of $r_i$ is given, by hyperbolic geometry, by the formula $|r_i|=2\arcsinh(\delta \cdot |\sphericalangle(r_i)|/2)$. For $x \geq 1$ we may approximate $\arcsinh(x)$ by $\log(x)+\mathcal O(1)$.  Thus for $|\sphericalangle r_i|\geq 2/\delta$ we may approximate $|r_i|$ as $\log\sphericalangle(r_i)$. If $|\sphericalangle r_i| < 2/\delta$ then $|r_i|$ is bounded between zero and a constant depending on $\delta$. The conclusion follows.
\end{proof}

\subsection{Almost geodesics}
We will later need to compare a loop based at $\star$ whose length is close to that of its geodesic representative (based at $\star$) with this geodesic representative.  The proposition below makes this precise. 

\begin{prop}[Small slippage implies nearby quasigeodesic]\label{prop:slippage} 
Suppose that
\begin{itemize}
\item $L, C$ are positive constants satisfying $L>100+100C$ ;
\item $g$ is a hyperbolic geodesic segment in $\IH$ joining points $P$ and $Q$, and $|g|=L$. \item  $\undertilde{g}$ is a rectifiable curve joining $P$ and $Q$ of length $\tilde{L} \leq L+C$.
\end{itemize}
Then
\begin{enumerate}
\item when parameterized by arc length starting from $P$, the orthogonal projection $\pi\colon \undertilde{g}\to g$ defines a $(1, C)$-quasi-isometry;
\item there exist constants $C'>C$ and $a>1$ such that for all $C_0 \leq C$, the proportion $p$ of the curve $\undertilde{g}$ that is at distance at least $C_0$ away from $g$ satisfies
\[ p < \frac{C}{L}a^{-C_0}\]
where $a>1$ is a universal constant. 
\end{enumerate}
\end{prop}

\begin{proof} (1) If $\pi(\undertilde{g}(s))=g(t)$ then $t \leq s$ and $\tilde{L}-s \geq L-t$, implying $t \leq s \leq t+C$. Write $S\coloneqq\undertilde{g}(s)$ and $T\coloneqq g(t)$ and consider the triangles $\Delta PTS$ and $\Delta QTS$. Set $x\coloneqq\overline{PT}=t$, $x'\coloneqq\overline{PS}$, $y\coloneqq\overline{TQ}=L-t$, $y'\coloneqq\overline{QS}$, $z\coloneqq\overline{ST}$. From laws of triangles we have 
\[ \cosh x' = \cosh x \cosh z,\qquad\cosh y' = \cosh y \cosh z.\]
Taking ratios and applying $x'+y'\leq x+y+C$ yields $\cosh z \leq 2e^{C/2}$. In turn this gives $z \leq \log 4 + C/2$. Thus $|\undertilde{g}(s)-g(s)| \leq |\undertilde{g}(s)-\pi(\tilde{g}(s)| + |\pi(\undertilde{g})(s)-g(s)| \leq z+|s-t|\leq \log 4 +C/2 + C \lesssim 3C\eqqcolon C'$ if $C\geq 2$. 

(2) Suppose $0<p<1$.  Consider the portion of $\tilde{g}$ lying at distance at least $C_0$ from $g$, and suppose that its projection to $g$ has length $pL$. Since geodesics diverge exponentially, there is a universal constant $a>1$ such that this portion has length at least $a^{C_0}\cdot pL$. The remainder of $\undertilde{g}$  has length at least $(1-p)L$. We obtain $a^{C_0}pL + (1-p)L \leq L+C$ 
and so $p \leq (C/L)a^{-C_0}$ as required.
\end{proof}

\subsection{The behavior of complex correspondences near cusps}\label{subsecn:behavior}
In this section, we collect elementary facts about the behavior of complex correspondences near the cusps. 

We begin by introducing some notation.  Consider $x\in \cusps{\cS}$.  We denote by $\beta_x\colon \ID-\{0\}\to\mathcal{\cS}$ the infinite-degree covering map induced by the subgroup of $\pi_1(\mathcal{\cS})$ generated by a peripheral loop about $x$.  Abusing notation, we let $\beta_x$ also denote the extension of $\beta_x$ to an analytic map of pointed Riemann surfaces $\beta_x\colon (\ID,0)\to (\cS \cup\{x\},x)$.  We denote by $s\coloneqq\exp(-2\pi/\delta)$ the Euclidean radius of the model cusp neighborhood given by the image of the horoball $\{\Im(z)>1/\delta\}$ under the quotient map by the stabilizer of $\sqrt{-1}\infty$, uniformized by the exponential map.   With these conventions, the restriction of $\beta_x$ to $\{|z|<s\}$ is an isomorphism onto $B(x,\delta)$ if $\delta$ is sufficiently small.  We call the coordinate $z\in \ID$ on $B(x,\delta)$ induced by $\beta_x$ a \emph{natural cusp coordinate}.   

We now suppose we are in the dynamical setting of an admissible complex correspondence $\phi, \rho\colon \cT \rightrightarrows\cS$.  Let us choose parameters $\delta, \zeta$ that specify thick-thin decompositions on $\cS$.  We fix an iterate $n\in \N$, and consider $x, y\in \cusps{\cS}$ with $F^{(-n)}(x)=y$.  By definition, this means that there is an associated cusp $\tilde{x}\in \cusps{\cT^{(n)}}$ with $\phi^{(n)}(\tilde{x})=x$ and $\rho^{(n)}(\tilde{x})=y$.  We denote by $d\coloneqq\deg(\phi^{(n)},\tilde{x})$ and $c\coloneqq\deg(\rho^{(n)},\tilde{x})$ the local degrees of $\phi^{(n)}$ and $\rho^{(n)}$ at $\tilde{x}$, respectively.  

\begin{setup}\label{setup}
By finiteness of the set of cusps of $\cT^{(n)}$, and continuity, there exist sufficiently small cusp circumferences $0<\delta'' < \delta' < \delta$ such that the following hold:
\begin{enumerate}
\item the restriction $\rho^{(n)}\restriction B(\tilde{x}, d\cdot \delta')$ is unramified outside of $\tilde{x}$;
\item the image $\rho^{(n)}(B(\tilde{x}, d\cdot \delta'))$ satisfies 
\[ B(y, \delta'') \subset \rho^{(n)}(B(\tilde{x}, d\cdot \delta')) \subset B(y, \delta);\]
\item upon setting $s'\coloneqq\exp(-2\pi / \delta')$, there is a local model $h\coloneqq h^x_y$ for this branch of $F^{(-n)}$, defined, single-valued, and analytic on the disk $\{|z|<(s')^{1/d}\}$ so that the cusps $x$ and $y$ correspond to the origin, and satisfy
\begin{equation}\label{eqn:dwind}
h(z)=az(1-zg(z))\text{ on the disk }\{|z|< (s')^{1/d}\}
\end{equation}
for some $0 \neq a\in \C$, and $|z\cdot g(z)| < 1/2$ on $\{|z|<(s')^{1/d}\}$.
\item For $\hat{\delta}\leq \delta'$, setting $\hat{s}\coloneqq s(\hat{\delta})\coloneqq\exp(-2\pi/\hat{\delta})^{1/d}$, there exist positive monotone functions $A_-, A_+\colon [0,\delta')\to [0,\delta)$ with the property that as  $\hat{\delta}\to 0^+$, the values satisfy $A_\pm(\hat{\delta})\to 0^+$, and for $w\coloneqq h(z)$
  \[ \{|w|<(1-A_-(\hat{\delta}))\cdot \hat{s}\} \subset h(\{|z|<\hat{s}\}) \subset  \{|w|<(1+A_+(\hat{\delta}))\cdot \hat{s}\}.\]
  \label{setup:deltahat}
\end{enumerate} 
\end{setup}

Figure~\ref{fig:cuspdiagram} explains the construction of $h$.   
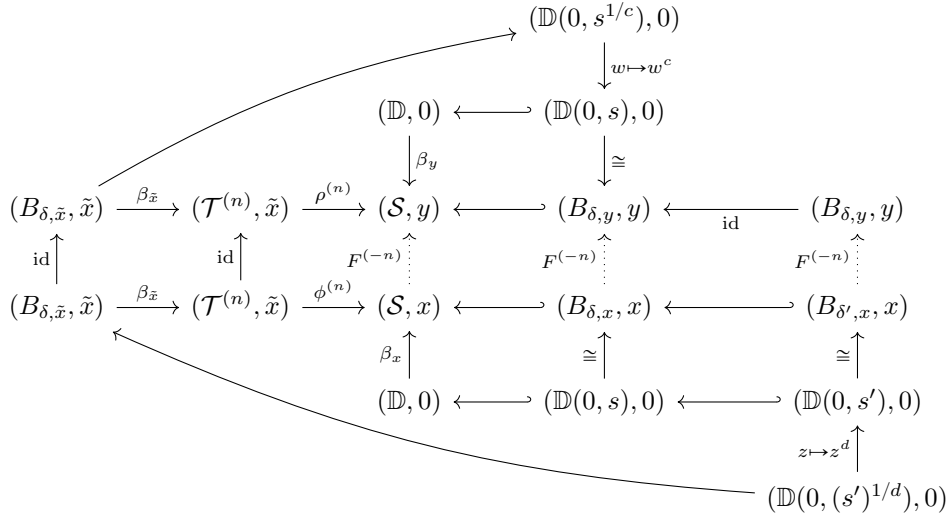
\begin{figure}
\[\begin{tikzcd}
&&  & (\ID(0,s^{1/c}),0) \arrow[d,"w \mapsto w^c"]\\
&& (\ID,0) \arrow[d,"\beta_y"] & (\ID(0,s),0) \arrow[d,"\cong"]\arrow[l,hook']\\
(B_{\delta, \tilde{x}},\tilde{x}) \arrow[r,"\beta_{\tilde{x}}"]  \arrow[rrruu,bend left=10] & (\cT^{(n)}, \tilde{x}) \arrow[r,"\rho^{(n)}"] & (\cS, y) & (B_{\delta, y},y)\arrow[l,hook'] &(B_{\delta,y},y) \arrow[l,"\text{id}"]\\
(B_{\delta, \tilde{x}},\tilde{x})\arrow[u,"\text{id}"] \arrow[r,"\beta_{\tilde{x}}"] & (\cT^{(n)}, \tilde{x}) \arrow[r,"\phi^{(n)}"]\arrow[u,"\text{id}"] & (\cS, x) \arrow[u,dotted,"F^{(-n)}"] & (B_{\delta, x},x)\arrow[u,dotted,"F^{(-n)}"]\arrow[l,hook'] & (B_{\delta',x},x)\arrow[l,hook']\arrow[u,dotted,"F^{(-n)}"]\\
& & (\ID,0) \arrow[u,"\beta_x"] & (\ID(0,s),0)\arrow[u,"\cong"]\arrow[l,hook'] &(\ID(0,s'),0)\arrow[u,"\cong"] \arrow[l,hook']\\
&&  & &(\ID(0,(s')^{1/d}),0)\arrow[u,"z \mapsto z^d"]\arrow[uullll,bend left=10]\\
\end{tikzcd}
\]
\caption{\sf Modeling the behavior of $F^{(-n)}$ near the cusp $x$, with image $y$. Arrows labeled $\cong$ are isomorphisms and those whose tails are curved are inclusions.  Following the arrows along the left-hand side from bottom-right to top yields a univalent, single-valued map $h\coloneqq h^x_y$ locally of the form $z \mapsto az(1+zg(z))$ with $|g|<1/2$ if $\delta'$ is chosen sufficiently small.}
\label{fig:cuspdiagram}
\end{figure}

We emphasize that $\delta$ is independent of $n$, while $\delta', \delta''$ do depend on $n$.  

In cusp coordinates, multiplication by a constant (such as $a$, or $1\pm A_\pm(\cdot))$) corresponds in the hyperbolic metric to a translation (such as $\tau \mapsto \tau+\log a$, etc.) whose displacement (such as e.g.\ $d_\IH(\tau, \tau+\log a)$) tends to zero as $\Im(\tau)\to\infty$.  The above discussion yields 

\begin{lem}\label{lem:t}
  If $x, y$ are cusps of $\cS$ such that $F^{(-n)}(x)=y$ and such that in the above natural cusp coordinates $\beta_x, \beta_y$ the corresponding branch of $F^{(-n)}$ has a Puiseux series of the form $z \mapsto az^t+\cdots$ with $t=c/d$, then in the Hausdorff metric on closed subsets of $\cS$ equipped with the hyperbolic metric, we have 
\[ F^{(-n)}(\overline{B(x,\hat{\delta})}) \approx \overline{B(y,\hat{\delta}/t)}\]
with additive error tending to zero  as $\hat{\delta}\to 0^+$.
\end{lem}

Note that the appearance of $t$ in the denominator makes sense: if e.g.\ $t=1/2$,  then the branch of $F^{(-n)}$ is strongly repelling in the Euclidean metric on the natural cusp coordinates, so the circumference of the image horoball should become larger. 

\section{$\sfX$-rays from correspondences}\label{secn:xrays}
In this section, we assume throughout that $\phi, \rho\colon \cT \rightrightarrows\cS$ is a covering correspondence of path-connected spaces.  As usual, $F\coloneqq\phi \circ \rho^{-1}$ denotes the associated multivalued map.  

\begin{defn}[$\sfX$-rays]\label{def:xray}
  Suppose $\star\in \cS$ is a base point (not necessarily fixed under $F$), $G\coloneqq\pi_1(\cS, \star)$, and $f$ is a path in $\cS$ joining $\star$ to an element of $F^{-1}(\star)$. Suppose $\sfX$ is a finite set of paths in $\cS$, each connecting $\star$ to an element of $F^{-1}(\star)$. An \emph{$\sfX$-ray of loops in the direction of  $f$ above $g\in G$} is an infinite sequence $g=g^{(0)},g^{(1)},\dots$ in $G$ such that for all $n\in\N$ there is $\sfx_n\in \sfX$ with
\[ f\cdot g^{(n)}\simeq g^{(n+1)}\cdot \sfx_{n+1}\text{ for all }n=0, 1,\ldots\]
\end{defn}
Equivalently: as a concatenation of paths beginning with $f$, we have $f.\widetilde{g}^{(n)}.\overline{\sfx_n} = g^{(n+1)}$ up to homotopy, where  the symbol $\overline{\sfx_n}$ denotes the reverse of the path $\sfx_n$.

We pause for a few clarifying comments about notation and terminology. $\sfX$-rays only depend on the homotopy classes relative to endpoints of the curves $f$ and $\sfx\in \sfX$.  Abusing terminology, we also say that $g^{(n+1)}$ is a ``lift'' of $g^{(n)}$, instead of a lift-followed-by-concatenation, and that the sequence $g^{(n)}$, $n=1, 2, 3, \ldots$ is a ``sequence of lifts of $g$''.  With this terminology, the relation on $\pi_1(\cS, \star)$ given by ``being a lift under the multivalued map $F$'' is a one-to-finite multivalued map, and an $\sfX$-ray is an orbit under iteration of this map.  It is also useful to think of $f$ and the collection $\sfX$ as together inducing a collection of self-maps of $G$---akin to an iterated function system---given by $g\mapsto(\text{any }h\text{ with }f\cdot g=h\cdot\sfx\text{ for some }x\in\sfX)$.

\begin{defn}[Contraction along $\sfX$-rays in the direction of $f$]\label{defn:bxr}
  We call the correspondence $F$ is \emph{contracting along $\sfX$-rays in the direction of $f$} if, for every choice of curves $\sfX$ as in the definition of $\sfX$-ray, there is a finite attractor $A\subset G$ such that for any $g\in G$, and any $\sfX$-ray $(g^{(n)})$ in the direction of $f$ above $g$, there is an integer $N\in \N$ such that $g^{(n)}\in A$ for all $n>N$.
\end{defn}

Note that in this definition, the finite set $A$ depends on $\sfX$, but not on $g$ and not on the choice of an $\sfX$-ray above $g$.  

Here is an example. More general notions could be formulated; the one below is simple enough to illustrate our themes.

\begin{defn}\label{defn:uniformly hyperbolic}
A covering correspondence from $\phi, \rho\colon \cT \rightrightarrows\cS$ is \emph{uniformly hyperbolic} if 
\begin{enumerate}
\item there exists a complete length metric on $\cS$ (yielding, via lifting by $\phi$, a locally isometric length metric on $\cT$);
\item there exists a compact rectifiably connected subset $K \subset \cS$ such that $\rho\circ\phi^{-1}(K) \subset K$;
\item the inclusion $K \hookrightarrow \cS$ induces a surjection on fundamental groups;
\item the norm on $\pi_1(\cS)$ induced by the length metric on $K$ is proper;
\item $\rho$ is a contraction everywhere, so that the restriction $\rho\colon \phi^{-1}(K)\to K$ is a uniform contraction, with factor $c<1$.
\end{enumerate}
\end{defn}
For example, an admissible complex correspondence that is uniformly hyperbolic according to its analytic definition, is also uniformly hyperbolic with this definition, when the hyperbolic metric is used; we may take $K$ to be any $\epsilon$-neighborhood of its compact limit set.

We give the proof of the elementary result below in order to illustrate the themes in our development.

\begin{prop}\label{prop:uh}
  If $F$ is uniformly hyperbolic, then for any  path $f$ joining any basepoint $\star$ to any element of $F^{-1}(\star)$, the correspondence $F$ is contracting along $\sfX$-rays in the direction of $f$.
\end{prop}

\begin{proof}
  Suppose $f, \sfX$ are as in the definition of $\sfX$-ray; we think of them as paths in $\cS$.  We may assume $f \subset K$ and $\sfx \subset K$ for each $\sfx \in \sfX$. For $g\in G$ arbitrary we denote by $|g|_K$ the infimum of the lengths of paths in $K$ representing $g$; similarly for a rectifiable path $\sfx \subset K$ we denote by $|\sfx|_K$ its length.  Put $\xi\coloneqq\max\{|\sfx|_K : \sfx\in \sfX\}$.   Suppose $g=g^{(0)}\in G$; we may represent it by a loop in $K$ based at $\star$.  From the definition of $\sfX$-ray in the direction of $f$ above $g$, and the backward-invariance of $K$, we have that for each $n\in \N$, the loop representing $g^{(n)}$ lies in $K$ and has length at most $(|f|_K+\xi)/(1-c)+c^{n}|g|_K$.  Thus when $n$ is sufficiently large, the length $|g^{(n)}|_K$ is bounded by a constant, $C$, independent of $g$. By properness of the norm, the set of elements $A\coloneqq\{h\in G : |h|_K<C\}$ is finite.
\end{proof}

Corollary~\ref{cor:contraction} below, applied to e.g.\ the correspondence on moduli space induced by Lodge's example of~\S\ref{secn:examples}, shows that the converse of Proposition~\ref{prop:uh} need not hold.  Thus, being ``contracting along  $\sfX$-rays in the direction of $f$'' is a form of non-uniform hyperbolicity property enjoyed by the pair consisting of the correspondence $\phi, \rho\colon \cT \to \cS$ and the element $f$.  

We next recall the statement of our main result:
\setcounter{mainthm}{2}
\begin{mainthm}
  Suppose $F=\phi, \rho\colon \cT \rightrightarrows\cS$ is an admissible correspondence, and $\star\in \cS$ is a fixed point of $F$.  Let $\star$ also denote the constant path at $\star$.  Fix a finite collection $\sfX$ of curves as in~\eqref{eq:genset}. Then there are constants $\kappa>0,N\in\N,\epsilon>0,\xi>0$ such that for every $g\in \pi_1(\cS, \star)$ and every $\sfX$-ray of loops $(g^{(n)})$ in the direction of $\star$ above $g$, we have
  \begin{enumerate}
  \item $|g^{(n+1)}|\le|g^{(n)}|+\xi$ for all $n$;\label{mainthm:1}
   \item if $|g|>\kappa$ then $|g^{(n)}|\le |g|-\epsilon$ for some $n\in\{0,\dots,N\}$. \label{mainthm:2}
  \end{enumerate}
\end{mainthm}
We recall that the norm $|\cdot|$ on the fundamental group $\pi_1(\cS, \star)$ is that induced by the hyperbolic metric.

Estimate~\eqref{mainthm:1} in Theorem~\ref{thm:main} is an immediate consequence of the definitions and the fact that admissible analytic correspondences are weakly hyperbolic.  Proving estimate~\eqref{mainthm:2} requires much more work.

We next show that the additive contraction estimates~\eqref{mainthm:1} and~\eqref{mainthm:2} in Theorem~\ref{thm:main} imply contraction along $\sfX$-rays at the fixed point $f=\star$.

\begin{cor}[Contraction along $\sfX$-rays at $\star$]\label{cor:contraction}
  In the setting of Theorem~\ref{thm:main}, the correspondence $F$ is contracting along $\sfX$-rays in the direction of the fixed point $\star$.

More precisely: for all $n$,
\begin{equation}\label{eq:1cor:contraction}
\big| g^{(n)}\big| \le \max\left\{ \left| g\right| - \epsilon\left\lfloor \frac{n}{N}\right\rfloor , \ \kappa\right\}+N\xi,
\end{equation}
where $\lfloor \frac{n}{N}\rfloor $ is the integer part of $\frac{n}{N}$.  Thus for all $n \gg 1$,
\begin{equation}\label{eq:2:cor:contraction}
  g_n\in \big\{g : |g| \le \kappa +N\xi\big\}.
\end{equation}
\end{cor}
It follows that the set $\{g : |g| \le \kappa +N\xi\}$ is a finite global attractor for $\sfX$-rays in the direction of $\star$.

\begin{proof}
  Write $n_0\coloneqq 0$ and define inductively the indices \[I=\{n_1 <n_2<n_3<\dots <n_s<\dots \}\subseteq\N\] 
 using the following rules:
 \begin{itemize}
 \item if $\left|g^{n_{s-1}}\right|>\kappa$, then $n_s\le n_{s-1}+N$ is the first index such that  ${\left|g^{n_{s}}\right|\le\left|g^{n_{s-1}}\right|-\epsilon}$;
 \item if $\left|g^{n_{s-1}} \right|\le\kappa$, then set $n_{s}\coloneqq n_{s-1}+1$ and observe that $\left|g^{n_{s}}\right|\le \kappa+\xi$, by Conclusion~(2) of Theorem~\ref{thm:main}.
\end{itemize}

For each $s$, we have
\[\big| g^{(n_s)}\big| \le \max\left\{|g| - \epsilon s, \kappa+\xi\right\}. \]
Since for each $s$ we have $n_{s+1}-n_{s}\le N$, we obtain from Conclusion~(2) of Theorem~\ref{thm:main} that 
 \[\big| g^{(n)}\big| \le \max\left\{|g| - \epsilon s, \kappa+\xi\right\}+(N-1)\xi \] 
 for each $s$ and for every $n$ with $n_{s}< n <n_{s+1}$. Equation~\eqref{eq:1cor:contraction} holds because $s\le \lfloor \frac{n}N\rfloor$, and the corollary follows.
\end{proof}

\subsection{$\sfX$-rays via bisets}\label{subsec:basiscase}
The notion of $\sfX$-ray may naturally be cast in the algebraic language of bisets, see~\S\ref{subsecn:introbisets}. Indeed the definition only depends on the homotopy classes, relative to endpoints, of the paths $f$ and $\sfx\in \sfX$, so may be viewed entirely algebraically. This subsection is not necessary for the proofs of the main theorems of this article.

\begin{defn}[$\sfX$-ray, biset version]
  Let $B$ be a left-free covering $G$-$G$-biset; choose a finite subset $\sfX\subseteq B$ and an element $\sfb\in B$.

  An \emph{$\sfX$-ray in the direction of $\sfb$ above $g=g^{(0)}\in G$} is an infinite sequence $(g^{(n)})$ such that for each $n\in \N$, there exist $\sfx_n\in \sfX$ with 
\[ \sfb^{\otimes n} \cdot g = g^{(n)}\cdot \sfx_n\sfx_{n-1}\ldots \sfx_1\]
as an equality in $B^{\otimes n}$. Equivalently, since the biset $B$ is left-free,
\begin{equation}\label{eq:square}
\sfb \cdot g^{(n)}=g^{(n+1)}\cdot \sfx_{n+1}
\end{equation}
\end{defn}

Equivalently, it is an orbit under iteration of the multivalued operator (determined by $\sfb$ and $\sfX$) given by $g\mapsto\{h\in G: \sfb\cdot g\in h\cdot\sfX\}$.

\begin{defn}[Contracting along $\sfX$-rays at $\sfb$, biset version]
We say the biset $B$ is \emph{contracting along $\sfX$-rays in the direction of $\sfb$} if, for any choice of finite set $\sfX \subset B$ containing a basis, as in the definition of $\sfX$-ray, there is a finite attractor  $A\subset G$ such that for any $g\in G$, and any $\sfX$-ray $(g^{(n)})$ in the direction of $\sfb$ above $g$, there is an integer $N\in \N$ such that $g^{(n)}\in A$ for all $n>N$.
\end{defn}

Equivalently, the above multivalued operator $g \mapsto \{h\in G: \sfb\cdot g\in h\cdot\sfX\}$ has a finite global attractor.

The attractor $A$ depends on $\sfX$ and $\sfb$.  Contraction is clearly invariant under isomorphisms $(B,b)\to (B', b')$.  It is also invariant under conjugation of $b$.  To see this, suppose $\sfa=h\cdot \sfb\cdot h^{-1}$ for some $h\in G$.   Then  $\sfa^{\otimes n}=h \cdot \sfb^{\otimes n}\cdot h^{-1}$, and so $\sfa^{\otimes n}\cdot g=h\cdot \sfb^{\otimes n}\cdot (h^{-1}g)$.  For $n$ large enough we have $\sfb^{\otimes n}\cdot h^{-1}g\in A$ and so we see that $B$ is also contracting along $\sfX$-rays in the direction of $a$, with attractor $hA$.  

Formulated in this way, this contracting property is an invariant of the pair $(B,\sfb^G)$, where $\sfb^G$ denotes the conjugacy class of $\sfb$ in $B$.

\section{Proof of Theorem~\ref{thm:main}}\label{secn:proof}
Let $\phi, \rho\colon \cT \rightrightarrows\cS$ be an admissible complex correspondence having a fixed point $\star\in \cS$; as usual we set $F\coloneqq\phi \circ \rho^{-1}$.  We set $G\coloneqq\pi_1(\cS, \star)$ and equip $G$ with its hyperbolic norm, $|g|$. Choose parameters $\delta, \zeta$ to specify the notion of thick-thin decomposition of $\cS$, as in Proposition~\ref{prop:thickthinsurface}.  By compactness, local inverse branches of $F$ are uniformly contracting on the thick part of $\cS$; we put $c\coloneqq\max\{\|dF^{-1}(z)\|_{\cS, \cS} : z\in \cS^{\text{thick}}\}$, so that $c<1$. 

We recall that $\sfX$ is a finite collection of rectifiable paths from $\star$ to $F^{-1}(\star)$, and $\xi$ is defined as $\max\{|\sfx| : x\in \sfX\}$.

\subsection{Basic contraction estimate}\label{subsec:basic}
In this subsection, we study the relationship between the lengths of consecutive terms in an $\sfX$-ray.

Let $g\in G$ be given, and consider $g=g^{(0)}, g^{(1)}, \ldots$ an $\sfX$-ray of loops above $g$, represented by geodesic loops based at $\star$. 

We consider now the associated thick-thin decomposition of $g$, as in Proposition~\ref{prop:thickthinloop}:
\[ g=\ell_0.r_1\ldots r_k\ell_k.\]
We introduce the notation $L\coloneqq\max\{|\ell_i| : 1 \leq i \leq k\}$.  

\begin{prop}[Basic contraction]\label{prop:basic}
  We have 
\begin{enumerate}
\item $|g^{(1)}|<|g|+\xi$; more generally $|g^{(n+1)}|<|g^{(n)}|+\xi$ and $|g^{(n)}| \leq |g|+n\xi$;
\item $|g^{(1)}| < |g| - (k+1)(1-c)\zeta + \xi$;
\item $|g^{(1)}| < |g| - (1-c)L+\xi$.
\end{enumerate}
\end{prop}

\begin{proof} (1) follows immediately from the definition of $g^{(1)}$ and   contraction of $F^{-1}$. (2) follows from (1) by applying 
uniform contraction of $F^{-1}$ on each of the $\ell_i$'s.  (3) follows from (1) and applying uniform contraction of $F^{-1}$ on one of the $\ell_i$'s realizing the maximum length $L$.
\end{proof}

\subsection{Tight sequences} 
Conclusion~(1) of Theorem~\ref{thm:main} follows immediately from Proposition~\ref{prop:basic}(1).  Our proof of Conclusion (2) of Theorem~\ref{thm:main} is  by contradiction. We assume Conclusion (2) does not hold.   Then, flipping all quantifiers, for every $\kappa\in\R,N\in\N,\epsilon>0$ there are $g\in G$ with $|g|>\kappa$ and $|g^{(n)}|>|g|-\epsilon$ for all $n\in\{0,\dots,N\}$.

\begin{defn}\label{defn:tight}
  A \emph{tight sequence} is a sequence $(g_m)$ in $G$ such that $|g_m|>m$ and $|g_m^{(n)}|>|g_m|-1/m$ for all $n\in\{0,\dots,m\}$.
\end{defn}
Thus we assume that there exists a tight sequence, and we will derive a contradiction.

\begin{prop}[Basic facts about tight sequences]\label{prop:tightfacts}
Suppose $(g_m)_m$ is a tight sequence, $m\in \N$, and 
\[ g_m = \ell_{0,m}.r_{1,m}\ldots r_{k_m,m}. \ell_{k_m,m}\]
is the $(\delta, \zeta)$-thick-thin decomposition of $g_m$ as in Proposition~\ref{prop:thickthinloop}. Let $L_m=\max\{|\ell_{i,m}|: 1 \leq i \leq k_m\}$ be the length of the longest segment in the thick part of the decomposition of $g_m$.  Then $L_m$ and $k_m$ are uniformly bounded independently of $m$: for all $m$ we have 
\begin{enumerate}
\item $k_m \leq 1+ \frac{\xi}{(1-c)\zeta}\eqqcolon K'$;
\item $L_m \leq 1+ \frac{\xi}{(1-c)}\eqqcolon L'$;
\item as $m\to\infty$, 
 \[ |g_m| = \sum_{i=1}^{k_m} |r_i| + \mathcal O(1) = 2\sum_{i=1}^{k_m}\log_+|\sphericalangle(r_i)|+\mathcal O(1)\]
 where the implicit constant is independent of the given tight sequence;
\item moreover,
\begin{equation}\label{eq:unbounded}
  \limsup_{m\to\infty}\max_{1 \leq i \leq k_m}|r_i| = +\infty.
\end{equation}
\end{enumerate}
\end{prop}

\begin{proof}
The condition $|g_m^{(1)}|>|g_m|-\frac{1}{m}$ in the definition of a tight sequence and Proposition~\ref{prop:basic} immediately imply~(1) and~(2). The estimate~(3) follows from Lemma~\ref{lem:lengthest} and the estimate $k_m \leq K'$. The condition $|g_m| > m$ and the uniform bound on $L_m$ from the second above inequality (2) yield the last conclusion. 
\end{proof}

For a tight sequence $(g_m)$, a picture thus emerges of the terms $g_m$ as $m\to\infty$.  In the complement of the cusp neighborhoods, the thick parts $\ell_{i,m}$ range over finitely many possible free boundary homotopy classes (in which the endpoints are allowed to slide along the boundaries of the cusp neighborhoods). Additionally, there must also be some roundabouts $r_{i,m}$ that spend larger and larger amounts of time (length) in some fixed cusp neighborhood.  An index $i\in \{1, \ldots, K'\}$ for which $\limsup_{m\to\infty} |r_{i,m}|=\infty$ will be called an \emph{unbounded roundabout}. Thus every tight sequence has a nonempty set of unbounded roundabouts. 

The property of being a tight sequence is not \emph{a priori} preserved under lifting under $F$.  However, given a sequence $(g_m^{(n)})_m$ of $\sfX$-rays in the direction of $f$, from the two-dimensionally-indexed array $(g^{(n)}_m)_{n,m\in \N}$, it is always possible to extract a tight sequence --- and one may in fact prescribe a lower bound on the levels of lifts involved. In the statement below, the sequences referred to are sequences of natural numbers.  The precise form of the lower bound on $m(t)$ in the statement is not important here; we include it merely for concreteness. 

\begin{lem}[Robust extraction]\label{lemma:extraction}
Suppose $(g_m)$ is a tight sequence.  Given
\begin{enumerate}
\item an increasing sequence $N(t), t=1, 2, \ldots$  with $N(t)>t$, and
\item an increasing sequence $m(t), t=1, 2, \ldots$  with $m(t)\gg N(t)$ sufficiently large (specifically, $m(t)>M(t)\coloneqq N(t)+t^2(N(t)\xi + 1)+t$),
\end{enumerate}
there exists a sequence $n(t), t=1, 2, \ldots$ with $n(t)\in [N(t), M(t)]$ such that the sequence 
\[ h_t\coloneqq g_{m(t)}^{(n(t))}\in G\]
is tight.
\end{lem} 

\begin{proof}
  We first verify that the lengths of the $h_t$'s grow as required. Since the sequence $(g_m)$ is tight, applying both conditions of Definition~\ref{defn:tight} we find
\[ |h_t|=|g_{m(t)}^{(n(t))}| > |g_{m(t)}|-\frac{1}{m(t)} > m(t)-\frac{1}{m(t)}>N(t)>t.\]

We next verify that the lengths of the lifts of the $h_t$ satisfy the required lower bound. Fix $t\in \N$.  To ease notation, put $h\coloneqq h_{m(t)}^{(N(t))}$. Consider the orbit segment of length $t$ under lifting given by 
\[ h, h^{(1)}, h^{(2)}, \ldots, h^{(t)}.\]
If for each $j\in \{1, \ldots, t\}$ we have 
\[ |h^{(j)}|>|h|-\frac{1}{t},\]
then we may set $n(t)\coloneqq N(t)$. Otherwise, there exists $j_1\in \{1, \ldots, t\}$ with 
\[ |h^{(j_1)}|\leq |h|-\frac{1}{t}.\]
We now repeat this argument, starting with $h^{(j_1)}$.  That is, we consider the orbit segment of length $t$ under lifting given by 
\[ h^{(j_1)}, h^{(j_1+1)}, h^{(j_1+2)}, \ldots, h^{(j_1+t)}.\]
If for each $j\in \{1, \ldots, t\}$ we have 
\[ |h^{(j_1+j)}|>|h|-\frac{1}{t},\]
then we may set $n(t)\coloneqq N(t)+j_1$. Otherwise, we find an integer $j_2\in \{1, \ldots, t\}$ for which 
\[ |h^{(j_1+j_2)}|\leq |h^{(j_1)}|-\frac{1}{t} \leq |h|-\frac{2}{t}.\]
Continuing in this manner, if we fail to find the desired bounds some number of times, say $s$, then we have the upper bound on lengths
\[ |h^{(j_1+j_2+\ldots+j_s)} \leq |h|-\frac{s}{t}.\]

We now proceed to bound $s$. Unwinding the definition of $h$, the previous inequality yields
\[ \left|g_{m(t)}^{(N(t)+j_1+\ldots+j_s)}\right| \leq \left|g_{m(t)}^{(N(t))}\right|-\frac{s}{t}.\]
The basic length upper bound (Proposition~\ref{prop:basic}) and the fact that $(g_m)_m$ is a tight sequence then imply 
\[ |g_{m(t)}|-\frac{1}{m(t)} < \left|g_{m(t)}^{(N(t)+j_1+\ldots+j_s)}\right| \leq |g_{m(t)}|+N(t)\xi - \frac{s}{t}.\]
This implies
\[ s \leq t(N(t)\xi + 1).\]
This bound on $s$ yields that among the sequence of lifts 
\[ h^{(1)}, h^{(2)}, \ldots, h^{(t^2(N(t)\xi+1)+t)}\]
there must be an index $i=i(t)\in \{1, 2, \ldots, t^2(N(t)\xi+1)+t\}$ such that the subsequence of consecutive lifts 
\[ h^{(i)}, h^{(i+1)}, h^{(i+2)}, \ldots, h^{(i+t)}\]
has the property that the desired lower bounds
\[ |h^{(i+j)}|>|h^{(i)}|-\frac{1}{t}, j=1\ldots t\]
hold. Putting $n(t)\coloneqq N(t)+i(t)$ and $h_t\coloneqq g_{m(t)}^{(n(t))}$, the Lemma is proven.
\end{proof}

\subsection{$m$-roundabouts}
For a tight sequence $(g_m)$, Proposition~\ref{prop:tightfacts} implies that as $m\to\infty$, the number of roundabouts is universally bounded independently of the chosen sequence by the constant $K'$, and that the lengths of some---but perhaps not all---of these roundabouts must become unbounded, and so (by Lemma~\ref{lem:lengthest}) spiral more and more around some cusps. We wish our decompositions to focus only on these cusps, and not on cusps for which the spiraling is bounded. 

Choose and fix a large \emph{roundabout parameter} $\mu\gg\zeta$.

\begin{defn}[$\mu$-roundabout]\label{defn:muroundabout}
  Consider $g\in G$. The \emph{$(\delta,\zeta)$-thick-thin decomposition with roundabout parameter $\mu$} is given by 
\[ g = \ell_0.r_1\ldots r_k.\ell_k\]
where 
\begin{enumerate}
\item each $r_i$ is a roundabout in the $(\delta, \zeta)$-thick-thin decomposition;\label{defn:muroundabout:1}
\item $|r_i| \geq \mu$ for each $1 \leq i \leq k$;\label{defn:muroundabout:2}
\item the $r_i$ are maximal with respect to~\eqref{defn:muroundabout:1} and~\eqref{defn:muroundabout:2}.
\end{enumerate}
\end{defn} 
We call the $r_i$ the $\mu$-\emph{roundabouts} or $\mu$-\emph{thin} parts, and the $\ell_i$'s the $\mu$-thick parts. In this notation, we say that $g$ has \emph{$\mu$-roundabouts} $r_1,\dots,r_k$. The estimate used in the proof of Lemma~\ref{lem:lengthest} then implies that if $|r_i| \ge \mu$, then $r_i$ winds at least $\sphericalangle(r_i)\approx \exp(\mu)$ times around the corresponding puncture.

Now suppose $(g_m)$ is a tight sequence and $\mu\gg1$ is a roundabout parameter.   Applying Proposition~\ref{prop:tightfacts}, we find the following. Since every $\mu$-roundabout is an ordinary roundabout, and the number of ordinary roundabouts is uniformly bounded by the constant $K'$ of that Proposition, we have 
\[ \limsup_{m\to\infty}\#\{\text{$\mu$-roundabouts in }g_m\} <\infty.\]
Obviously this limit decreases as $\mu$ increases, and it cannot decrease to zero, since every tight sequence has at least one unbounded roundabout.  Thus 
\begin{equation}
\label{eq:k_g}
k((g_m))\coloneqq\lim_{\mu\to\infty}\limsup_{m\to\infty}\#\{\text{$\mu$-roundabouts in }g_m\}\in [1, K'].
\end{equation}
In other words, $k=k((g_m))$ is the smallest integer such that the $g_m$'s eventually have at most $k$ unbounded roundabouts. From Proposition~\ref{prop:tightfacts}, we know $k \leq K'$. 

\subsection{Very tight sequences}
Proposition~\ref{prop:tightfacts} provides a universal upper bound on the number of unbounded roundabouts in a tight sequence.  To capture the extremal case, we set 
\[ K\coloneqq\sup\{k((g_m)) : (g_m) \text{ is tight}\} \leq K'<\infty.\]
The definition below captures the extremal tight sequences with respect to this bound. 

\begin{defn}[Very tight sequences]\label{defn:verytight}
  A \emph{very tight sequence} is a tight sequence $(g_m)$ in $G$ such that, additionally, the number of $m$-roundabouts in $g_m$ is precisely the stable extremal number, $K$, of unbounded roundabouts whenever $m$ is sufficiently large.
\end{defn}

Note that if $(g_m)_m$ is any tight sequence for which $k((g_m))=K$ realizes the extremal number of unbounded roundabouts, then by passing to subsequences we may assume that in fact the sequence $(g_m)_m$ is very tight. Summarizing: if Theorem~\ref{thm:main} fails, then there is a very tight sequence.  

We will often make claims on terms $g_m$ that hold whenever $m$ is sufficiently large; we abbreviate this as ``something holds eventually for $g_m$''. Thus in a very tight sequence eventually the number of $m$-roundabouts of $g_m$ is precisely $K$ .

\begin{lem}\label{lemma:lambda}
  There is a constant $\lambda>0$ such that, for every very tight sequence $(g_m)_m$, eventually all the thick segments $\ell_{m,i}$, for $i\in\{0,\dots,K\}$, of the $(\delta,\zeta)$-thick-thin decomposition of $g_m$ with roundabout parameter $m$ have length at most $\lambda$.
\end{lem}
\begin{proof}
  Assume the contrary.  Let $K$ be the stable bound on the number of roundabouts as in Definition~\ref{defn:verytight}. Then for every lower length bound $n\in\N$, there is a very tight sequence $(g_{m,n})_m$ satisfying the condition:  \emph{for all sufficiently large $m$, in the thick-thin decomposition of $g_{m,n}$ with roundabout parameter $m$,
   \[ g_{m,n}=\ell_{m,n,0} . r_{m,n,1} \ldots r_{m,n,i} . \ell_{m,n,i}\ldots \ell_{m,n,K},\]
some $m$-thick segment $\ell_{m,n,i}$ has length at least $n$.}  A word about notation here: $n$ does not denote a level of iteration. 

Fix now $n$ large, and consider  the ordinary $0$-thick decomposition of Proposition~\ref{prop:thickthinloop}, applied to $g_{m,n}$. We find that the long subsegment $\ell_{m,n,i}$ must be further decomposed, say as 
\[ \ell_{m,n,i}=L_0 . R_0\ldots L_q .R_q.\]
By Proposition~\ref{prop:tightfacts}, we have $q\leq K'$, a universal constant, and $|L_j| \leq L'$ for each $j=0, \ldots, q$. Since $|\ell_{m,n,i}|>n$ we have then that $\sum_{j=0}^q |R_j| > n-K'L'$ and so some $R_j$ satisfies $|R_j| > \frac{n-K'L'}{K'}\to\infty$ as $n\to\infty$.  

Now we extract a diagonal subsequence $(g_p)_p, p\in \N$ from the $(g_{m,n})$. Given $p\in \N$, choose first $n=n(p)$ so large that $\frac{n-K'L'}{K'} > p$, and then from the very tight sequence $(g_{m,n})_m$, choose $m=m(n)$ so large that in the $p$-thick-thin decomposition of the $m$th term $g_{m,n}$, we have some $p$-thick-subsegment $\ell_{m,n,i}$ (as in the previous paragraph) with length at least $p$. But then $(g_p)_p$ is a tight sequence with at least $K+1$ unbounded roundabouts: there are $K$ unbounded roundabouts in the original decomposition, and one more developing in the $\ell_{m(p), n(p), i}$'s. This contradicts the assumption that the sequence $(g_m)$ was assumed to be very tight, thus realizing the maximal number of unbounded roundabouts. 
\end{proof}

\noindent Summarizing:
\begin{lem}[Lengths in decomposition of very tight sequences]\label{lemma:ldvt}
  There are constants $\lambda>0$ and $K\in \{1, 2, 3, \ldots\}$ such that the following hold. Suppose $(g_m)$ is very tight, and 
\[ g_m = \ell_{0,m}.r_{1,m}.\ell_{1,m}\ldots r_{K,m}.\ell_{K,m}\]
is the $(\delta, \zeta)$-thick-thin decomposition with roundabout parameter $m$ of its $m$th term. Then for each $m$, 
\[ |\ell_{i,m}| < \lambda\text{ for all } i=0,\ldots,K,\]
and since $(g_m)$ is very tight,
\[ |r_{i,m}| >m\text{ for all } i=1,\ldots, K.\]
\end{lem}

That is, in a very tight sequence, eventually the terms consist of $K+1$ ``short'' segments in the thick part (each of length at most $\lambda$), separated by precisely $K$ roundabouts that are becoming unbounded at rate at least $m$. We next show that, to some extent, this structure persists along the terms $g_m^{(n)}$ of the $\sfX$-ray above each $g_m$, at least when $m\gg n$.

\subsection{Roundabouts lift to roundabouts}\label{subsecn:liftingunder}
Let $(g_m)$ be a very tight sequence; it has precisely $K$ unbounded roundabouts, by definition.  In this subsection, we show that for each iterate $n=1, 2, 3, \ldots$, the sequence $(g_m^{(n)})_m$ also has precisely $K$ unbounded roundabouts.

Fix $n\in \N$. Let $\phi^{(n)}, \rho^{(n)}\colon \cT^{(n)} \rightrightarrows\cS$ be the $n$th iterate of the correspondence.  Consider $m\gg n$.  In the remainder of this paragraph, to ease notation, when denoting paths, we temporarily drop sub- and super-scripts indicating the dependence on $m$ and $n$.  Suppose $g\coloneqq g_m$ has $(\delta, \zeta)$-thick-thin-decomposition with roundabout parameter $m$ given by a concatenation of geodesic segments 
 \[ g=\ell_0.r_1\ldots r_K . \ell_K.\]
 When we lift $g$ as a path under $\phi^{(n)}$, since  $\phi^{(n)}$ is a cover, it is a local isometry, so we get a lifted path which is also given by a concatenation of geodesic segments
 \[ \tilde{g} =\tilde{\ell}_0 . \tilde{r}_1\ldots \tilde{r}_K . \tilde{\ell}_K.\]
 Since $\phi^{(n)}$ is a local isometry, and $(g_m)$ is very tight, we have 
 \[ \lambda \geq |\ell_i|=|\tilde{\ell}_i|\text{ for all } i=0\ldots K\]
 and
 \[ m \leq |r_i|=|\tilde{r}_i|\text{ for all } i=1\ldots K\]
 where $\lambda$ is the constant of Proposition~\ref{lemma:lambda} and the indicated inequalities hold for all $m\geq n$. The above decomposition of the path $\tilde{g}$ will typically  not be a $(\delta, \zeta)$-thick-thin decomposition in $\cT^{(n)}$ --- indeed, $\tilde{g}$ need not be a closed path.  We delay until the next subsection the analysis of these effects, in favor of giving first some preliminary, simple observations.  
 
 For this paragraph, we re-introduce the notation giving the dependence on $m$.  Recalling that we have fixed the number of iterates $n$, we use super- and sub-tilde's to indicate the dependence on this fixed level. When we map $\tilde{g}_m \subset \cT^{(n)}$ to $\cS$ under $\rho^{(n)}$, we get an expression as a concatenation of segments --- now not necessarily geodesics:
 \[ \rho^{(n)}(\tilde{g}_m)\eqqcolon\undertilde{g}_m=\undertilde{\ell}_{m,0} . \undertilde{r}_{m,1} . \cdots . \undertilde{r}_{m,K} . \undertilde{\ell}_{m,K}\] 
In the next subsection, we will explore the difference. By definition, the element $g_m^{(n)}\in G$ in the definition of $\sfX$-ray is represented by the loop $\undertilde{g}_m.\overline{\undertilde{\sfx}_m}$, where $\undertilde{\sfx}_m$ is a concatenation of lifts of elements of $\sfX$ under iterates $F^{j}$ for $j=1,\ldots, n$.  Since the maps $\rho^{(j)}\colon \cT^{(j)}\to\cS$ are contractions for all $j=1,\ldots, n$,  the length of the path $\undertilde{\sfx}_m$ is at most $n\xi$.  Since $(g_m)$ is very tight, it is tight, so we conclude 
\begin{equation}\label{eq:nth}
m-\frac{1}{m} \leq |g_m|-\frac{1}{m} \leq |g_m^{(n)}| \leq |\undertilde{g}_m.\overline{\undertilde{\sfx}_m}| \leq |g_m|+n\xi.
\end{equation}
Summarizing: for fixed $n$, as $m\to\infty$, the loops $\undertilde{g}_m.\undertilde{\sfx}_m\subset \cS$ join $\star$ to itself, are getting longer and longer, and when pulled tight to a geodesic joining $\star$ to itself, yield a curve that is at most an additive constant (say $C\coloneqq n\xi + 1)$, independent of $m$, shorter. Proposition~\ref{prop:slippage} then implies that when both $\undertilde{g}_m.\undertilde{\sfx}_m$ and its geodesic representative $g^{(1)}_m$ are parameterized by arc length in $\cS$, the image points under these parametrizations are at most distance $C$ apart, and in fact most of the time they are actually much closer. 

In this paragraph, we set up some notation that we will need later. Let $x_1,\ldots, x_K$ be the sequence of cusps surrounded by the $m$-roundabouts $r_1, \ldots, r_K$ of $g_m$, in the order in which they occur.  Note that some cusps may be visited more than once. Lifting under $\phi^{(n)}$, we obtain a sequence of punctures $\tilde{x}_1, \ldots, \tilde{x}_K$ in $\cusps{\cT^{(n)}}$.  Mapping under $\rho^{(n)}$, we obtain a sequence of points $\undertilde{x}_1, \ldots, \undertilde{x}_K$ in $\overline{\cS}$.  

Our next goal is to show that each point $\undertilde{x}_i$ is in fact a cusp of $\cS$.  Along the way, we introduce further notation.

Write $d^{(n)}_i=\deg(\phi^{(n)}, \tilde{x}_i)$, $c^{(n)}_i=\deg(\rho^{(n)}, \tilde{x}_i)$, and put $t^{(n)}_i\coloneqq c^{(n)}_i/d^{(n)}_i$.  Let $B_i\coloneqq B(x_i,\delta) \subset \cS$ be the cusp neighborhood of $x_i$;  by the definition of $(\delta, \zeta)$-thick-thin decomposition, its bounding horocycle curve has length $\delta$.  
Let $\tilde{B}_i \subset \cT^{(n)}$ be the component of the preimage of $B_i$ under $\phi^{(n)}$ surrounding $\tilde{x}_i$, and $\undertilde{B}_i \subset \cS$ its image under $\rho^{(n)}$.  We now claim that in fact each point $\undertilde{x}_i$ is a cusp of $\cS$.  For suppose for some $i$ that $\undertilde{x}_i$ were a regular point of $\cS$, instead of a cusp. The we can factor $\rho^{(n)}$ as a composition of maps of hyperbolic Riemann surfaces:
\[ \rho^{(n)}\colon \cT^{(n)} \hookrightarrow \cT^{(n)}\cup\{\tilde{x}_i\}\to\cS.\]
The second map is a   contraction; let us show that near $\tilde{x}_i$, the inclusion is very strongly contracting on $\tilde{B}_i$. It is a general fact that when including one Riemann surface into another, the amount of contraction incurred at a point $z$ is a continuous function bounded above away from $1$ by the hyperbolic distance from $z$ to a point in the complement (compare~\cite{ctm:renorm}*{Theorem~2.25}). In the surface $\cT^{(n)} \cup \{\tilde{x}_i\}$, the Jordan region $\tilde{B}_i \cup\{\tilde{x}\}$ has boundary of length at most $d_i^{(n)}\delta$, and hence it has diameter at most $d^{(n)}_i\delta$ as well. It follows that there is a constant $0<u<1$ such that $|\undertilde{r}_i| < u|\tilde{r}_i|=u|r_i|$.  Equivalently, $|\undertilde{r}_i| < |r_i| - (1-u)|r_i|$. Summing over indices $i$, we have by~\ref{eq:nth}  that up to additive constants depending on $n$ but not $m$, we have the estimate 
\[ |g_m|=|r_1|+\ldots+|r_k| + \mathcal O_n(1)  = |g_m^{(n)}|=|\undertilde{r}_1| + \ldots + |\undertilde{r}_k| + \mathcal O_n(1)\]
as $m\to\infty$. 
But this is impossible if for some $i$ we have $|\undertilde{r}_i| < |r_i|-(1-u)|r_i|$, since very tight implies $|r_i|\to\infty$ as $m\to\infty$. 

The fact that the $\undertilde{x}_i$'s are cusps leaves open the possibility that the lifts 
$\undertilde{x}_i$, $\undertilde{x}_{i+1}$ of two consecutive cusps might be the same; that is, the lift $\undertilde{\ell}_{m, i+1}$ might be contained in a cusp neighborhood of $\undertilde{x}_i$. We now rule this out. Were this to happen, we would have  
  \[ \log(\sphericalangle |r_{i,m}^{(n)}| + \sphericalangle |r^{(n)}_{i+1,m}|) \ll \log(\sphericalangle |r_{i,m}^{(n)}| )+ \log(\sphericalangle |r^{(n)}_{i+1,m}|) \]
again implying substantial drop in length, which is impossible. 

In the next section, we analyze the lengths of the thick and thin parts in $g_m^{(n)}$.  

\subsection{How thick and thin lengths change under lifting}\label{subsecn:change}
We continue the discussion and notation of the previous subsection.
The estimate~\eqref{eq:nth} implies that $|g_m^{(n)}| = |\undertilde{g}_m.\overline{\undertilde{\sfx}_m}|+\mathcal O_n(1)$ as $m\to\infty$, where the implicit constant depends on $n$ but not $m$.  Proposition~\ref{prop:slippage} then implies that the loop $\undertilde{g}_m.\overline{\undertilde{\sfx}_m}$ is uniformly $C_n$-close, independently of $m$ (but depending on $n$) to the geodesic $g_m^{(n)}$.


\begin{lem}[Lifting very tight sequences]\label{lem:tightlift}
  Suppose $(g_m)$ is a very tight sequence, and fix an iterate $n\in \N$.  Suppose that the $(\delta, \zeta)$-thin-thick decomposition with roundabout parameter $m$ of the geodesic loop $g_m$ is
  \[ g_m=\ell_{0,m}.r_{1,m}\ldots r_{K,m}.\ell_{K,m}\]
  and that the one of the geodesic loop $g_m^{(n)}$ is
    \[g_m^{(n)}=\ell_{0,m}^{(n)}.r_{1,m}^{(n)}\ldots r_{K,m}^{(n)}.\ell_{K,m}^{(n)}.\]
Then
\begin{enumerate}
\item the lengths of the roundabouts $r_{i,m}^{(n)}$ at level $n$ satisfy 
\[ |r_{i,m}^{(n)}| = |r_{i,m}| + 2\log t^{(n)}_{i,m} + o_n(1)\]
as $m\to\infty$, with implicit constant tending to zero depending on $n$;
\item after passing to subsequences (specifically, by choosing any increasing sequence $m(t)$ satisfying  $m(t)>t+C_n$ where $C_n=\max_i 2\log t^{(n)}_{i,m}$), for each $i$ we have
\[ |r_{i,m}^{(n)}| > m;\]
\item the lengths of the thick parts $\ell_{i,m}^{(n)}$ at level $n$ satisfy
\[ |\ell_{i,m}^{(n)}| < |\ell_{i,m}|+\nu_n+\mathcal O_n(1)<\lambda+\nu_n+\mathcal O_n(1),\]
as $m\to\infty$, with $\nu_n\coloneqq2\log(\delta/\delta')+n\xi$.
\end{enumerate}
\end{lem}

\begin{proof} We begin by proving (1).  We first prove a distortion estimate for classical winding numbers under application of an $h$ as in~\eqref{eqn:dwind}.  Suppose $r$ is a rectifiable path in $\{|z|<s'\}$ joining two points $A$ and $B$. We have 
\[ 2\pi \sphericalangle(h\circ r) =\int_{h\circ r}d(\arg(z)) =\Im\int_r d(\log(h(z))=\Im\left(\int_r dz/z + d\log(1+zg(z))\right).\]
The second term of the last integral is exact, so we obtain 
\begin{equation}\label{eq:univwind}
2\pi\sphericalangle(h\circ r)=2\pi\sphericalangle(r)+\Im\log\left(\frac{1+Bg(B)}{1+Ag(A)}\right)=2\pi\sphericalangle(r)+\mathcal O(1),
\end{equation}
where the implicit constant may be taken to be at most $\log(3)$ in magnitude and is thus universal.  

Fix an index $m \in \N$ for a term in a tight sequence. In this paragraph, we introduce a refined decomposition of the geodesic $g_m$, focusing on the behavior near the cusps.   
We suppress dependence on $m$, to ease notation.  Let $x_1, \ldots, x_K$ be the cusps visited by $g_m$. 
Consider the $(\delta,\zeta)$-thick-thin decomposition of the geodesic $g_m$ with roundabout parameter $m$ given by 
\[ g_m=\ell_{0}.r_{1}\ldots r_{K} . \ell_{K}.\]
Let $\delta'<\delta$ be the constant defined in~\S\ref{subsecn:behavior}; it depends on $n$.
For each $i\in \{1, \ldots, K\}$, we further decompose 
\[ r_{i} = r_{i'} . r_{i''} . r_{i'''}\]
where $r_{i''}$ is the intersection of $r_{i}$ with the cusp neighborhood $B(x_i,\delta')$ and the $r_i', r_i'''$ are the two complementary segments.  For large $m$, the segments $r_i', r_i'''$ in the annular cusp region $B(x_i, \delta) \setminus \overline{B}(x_i,\delta')$, viewed in the natural coordinates in the disk $\ID$ given by $\beta_{x_i}$, are nearly radial; in the upper-half-plane model, these subsegments are nearly vertical. Thus as $m\to\infty$ we have 
\begin{equation}\label{eqn:rprimes}
 |r_i'|=|r_i'''|=\log(\delta/\delta')+o(1)\; \text{and} \; |r_i''|=|r_i|-2\log(\delta/\delta')+o(1).
 \end{equation}
 
 We now lift this decomposition by $F^{(n)}$.  Applying the definition of $\sfX$-ray, and suppressing the superscript $n$, we get  
\[g_m^{(n)} \simeq\undertilde{g}_m.\undertilde{\sfx}_m = (\undertilde{\ell}_{0} .\undertilde{r}_1').\undertilde{r}_{1}''.(\undertilde{r}_1'''.\undertilde{\ell}_1.\undertilde{r}_2') \ldots \undertilde{r}_{K}'' . (\undertilde{r}_K'''.\undertilde{\ell}_{K}.\undertilde{\sfx}).\]

In the next paragraph, we estimate the winding numbers of the $r_i^{(n)}$ in terms of the winding numbers of the $\undertilde{r}_i''$. See Figure~\ref{fig:pencil}.

\begin{figure}
\begin{center}
\begin{tikzpicture}[>=stealth']
    \foreach\y in {0,2.5,4,5.5} { \draw (-0.5,\y) -- ++(3,0) +(-0.15,-0.15) -- +(0.15,0.15) +(0.85,-0.15) -- +(1.15,0.15) +(1,0) -- +(4,0); }
    \node at (6.8,0) {\LARGE $\R$};
    \node at (-1,7) {\LARGE $\IH$};
    \foreach\x in {0,6} {\draw[dashed] (\x,-0.3) -- (\x,2.5); \draw (\x,2.5) -- (\x,5.5); 
    \draw[dashed] (\x,5.5) -- (\x,6.3);
    \draw (\x,0) -- +(-2,6) (\x,0) -- +(2,6);
    }
    \draw[dashed] (1,1.5) -- ++(0,2.2);
    \draw[dashed] (5,1.5) -- ++(0,2.2);
    \draw[decorate,decoration={zigzag,amplitude=0.5mm}] (1,3.7) -- +(0.5,2.2) (5,3.7) -- +(-0.5,2.2);
    \node (A) at (3,0.75) {$\sphericalangle r_i^{(n)}$};
    \draw[->] (A) -- +(-3,0); \draw[->] (A) -- +(3,0);
    \node (B) at (3,1.5) {$\sphericalangle {\undertilde r''_i}^{(n)}$};
    \draw[->] (B) -- +(-2,0); \draw[->] (B) -- +(2,0);
    \node (C) at (3,6.5) {${\undertilde r''_i}^{(n)}$};
    \node (D) at (3,7.5) {$r_i^{(n)}$};
    \draw[->] (C) .. controls +(-0.7,0) and +(0,0.7) .. (1.5,5.8);
    \draw[->] (C) .. controls +(0.7,0) and +(0,0.7) .. (4.5,5.8);
    \draw[->] (D) .. controls +(-1,0) and +(0,1) .. (0,6.3);
    \draw[->] (D) .. controls +(1,0) and +(0,1) .. (6,6.3);
    \foreach\y/\t in {2.5/,4/',5.5/''} {\node at (6.5,\y+0.3) {$1/\delta\t$}; }
\end{tikzpicture}
\end{center}
\caption{} 
\label{fig:pencil}
\end{figure}

Fix $i\in \{1, \ldots, K\}$.  Set $x\coloneqq x_i$ and $y\coloneqq\undertilde{x}_i^{(n)}$.  The implicit constants below hold as $m\to\infty$, though they do depend on $n$. Recall  that for large $m$, the segments $r_i', r_i'''$ in the annular cusp region $B(x, \delta) \setminus \overline{B}(x,\delta')$, viewed in the natural coordinates in the disk $\ID$ given by $\beta_x$, are nearly radial.  Thus as $m\to\infty$ their classical winding numbers, which coincide with those as defined in~\S\ref{subsecn:windingnumbers}, satisfy 
\[ \sphericalangle r_i=\sphericalangle r_i'' + o(1)_{n\to\infty}\]
as $m\to\infty$.   The distortion estimate~\eqref{eq:univwind} implies that the classical winding number of the lift $ \undertilde{r}_i''$ satisfies 
\[ \sphericalangle \undertilde{r}_i''=t^{(n)}_i\sphericalangle r_i''+\mathcal O_n(1)\]
as $m\to\infty$, with a universal implicit constant (at most $\log(3)$); here $t^{(n)}_i$ is the exponent of the leading term in the series expansion of the branch of $F^{(-n)}$ sending $x$ to $y$.   
  By Proposition~\ref{prop:slippage}, the curve $\undertilde{r}_i''$ lies within a $C$-neighborhood of  a subsegment of the geodesic $g_m^{(n)}$, for a constant $C_n$ depending on $n$ but not $m$.  Also, the endpoints of $\undertilde{r}_i''$ lie within the annular region between cusp neighborhoods $B(y, \delta) \setminus \overline{B}(y,\delta'')$.  When $m$ is large, the portion of $r_i^{(n)}\subset g_m^{(n)}$  that lies within the annular cusp region $B(y, \delta) \setminus \overline{B}(y,\delta'')$ consists of two  subsegments that, in the natural coordinates in the disk $\ID$ given by $\beta_y$, are nearly radial; so that again, in the upper-half-plane model, the subsegments are nearly vertical.  It follows  that the winding number $\sphericalangle r_i^{(n)}$, both classically and as defined in~\S\ref{subsecn:windingnumbers}, satisfies 
\[ \sphericalangle r_i^{(n)}=t^{(n)}_i \sphericalangle r_i + \mathcal O_n(1)\]
with an implicit constant depending on $n$ but not $m$. 
Since the winding numbers tend to infinity as $m$ tends to infinity, we conclude upon applying the estimate in Lemma~\ref{lem:lengthest}  that 
\[ |r_i^{(n)}| = |r_i| + 2\log t^{(n)}_i + o(1)_{n\to\infty}.\]

Conclusion (2) is straightforward to verify.

We now prove Conclusion (3). For each $i\in \{1, \ldots, K-1\}$, the geodesic $\ell_i^{(n)}$ lies in a $C$-neighborhood of the curve $\undertilde{r}_{i}'''.\undertilde{\ell}_i.\undertilde{r}_{i+1}'$, which has length at most $\lambda+2\log(\delta/\delta')+o(1)_{n\to\infty}$.  Thus $|\ell_i^{(n)}| < \lambda+2\log(\delta/\delta')+C+o(1)_{n\to\infty}$.  When $i=0$, the curve $\ell_0^{(n)}$ lies in the $C$-neighborhood of $\undertilde{\ell}_0.\undertilde{r}_1'$, so has length at most $\lambda+\log(\delta/\delta')+C+o(1)_{n\to\infty}$.  When $i=K$, the curve $\ell_K^{(n)}$ lies in the $C$-neighborhood of the curve $\undertilde{r}_K'''.\undertilde{\sfx}$, and so has length at most $\log(\delta/\delta')+C+n\xi+o(1)_{n\to\infty}$. 
\end{proof}

\begin{cor}[Modelling lengths of lifts along a subsequence]\label{cor:m_n}
  For every very tight sequence $(g_m)$ there are unbounded sequences $(m(s)),(n(s))$ such that $(g_{m(s)})^{(n(s))}$ is very tight, and such that the roundabouts of $g_{m(s)}^{(n(s))}$ are the lifts of roundabouts of $g_{m(s)}$.

  In particular, $|\ell_{0,m(s)}^{(n(s))}|+\cdots+|\ell_{K,m(s)}^{(n(s))}|$ is bounded, and \eqref{eq:extraction} holds:
  \[|g_{m(s)}^{(n(s))}|=|g_{m(s)}|+2\log\Big(t_{1,m(s)}^{(n(s))}\cdots t_{K,m(s)}^{(n(s))}\Big)+\mathcal O(1).\]
\end{cor}
\begin{proof}
  The first claim follows from Lemma~\ref{lemma:extraction}, upon choosing $m(s)$ sufficiently sparse. Namely, we first select any $N(s)$ and $M(s)$ satisfying the hypotheses of Lemma~\ref{lemma:extraction}, and then a sufficiently large $m(s)$ so that lifts of roundabouts of the $g_{m(s)}$ under $n(s)\in[N(s),M(s)]$ iterations are long roundabouts by Lemma~\ref{lem:tightlift}(1). This guarantees that $g_{m(s)}^{(n(s))}$ still has $K$ roundabouts that are lifts of $g_{m(s)}$'s roundabouts, and therefore no other, new roundabouts. The second claim follows from Lemma~\ref{lemma:ldvt}, and the third one follows from Lemma~\ref{lem:tightlift}.
\end{proof}

%
%

We are still considering a very tight sequence $(g_m)$.  Given a fixed iterate $n$, we now know that eventually the thick-thin decompositions of $g_m$ and $g_m^{(n)}$ approximately match: that is, they are both made of $K+1$ ``short'' (now, however, with an upper bound depending on $n$) paths in the thick part, separated by $K$ very long roundabouts. Moreover, we have bounds relating the length of each thick and thin part of $g_m$ to the length of the corresponding thick and thin part of $g_m^{(n)}$. In the remainder of this section, we show that as the number of iterates $n$ increases, contraction on the thick parts must be offset by expansion on the thin part, yielding bounds on the exponents $t_i^{(n)}$ as $n\to\infty$. 

For $m\gg0$ let $x_{1,m},\dots,x_{K,m}$ denote the cusps that the roundabouts in $g_m$ surround, and for $m\gg n$ let $x_{1,m}^{(n)},\dots,x_{K,m}^{(n)}$ denote the corresponding cusps surrounded by the roundabouts in $g_m^{(n)}$. By our assumptions, there are branches of $F^{(-n)}$ that map $x_{m,i}$ to $x_{m,i}^{(n)}$, and have local behaviour, in natural cusp coordinates, modeled by a Puiseux series $z\mapsto a^{(n)}_{i,m}z^{t^{(n)}_{i,m}}+\ldots$ for some $a^{(n)}_{i,m}\in\C\setminus \{0\}$ and $t^{(n)}_{i,m}\in\Q$.

\begin{lem}\label{lem:tdown}
  For any very tight sequence $(g_m)$, there is $\eta > 0$ and a sequence of exponent bounds $d^{(n)}=\mathcal O(\exp(-\eta n))\to 0$ as $n\to\infty$, such that for any sufficiently large iterate $n\in\N$ and index $m\in \N$, we eventually  have
\begin{equation}
\label{eqn:exps} 
t^{(n)}_{1,m}\le d^{(n)}\text{ and }t^{(n)}_{i,m}t^{(n)}_{i+1,m}\le d^{(n)}\text{ for all }i\in\{1,\dots,K-1\}.
\end{equation}
In particular, 
\begin{equation}\label{eq:lem:tdown}
t^{(n)}_{1,m}t^{(n)}_{2,m}\ \dots\  t^{(n)}_{K,m} \le d^{(n)}\to 0
\end{equation}
as $n\to\infty$.
\end{lem}

\begin{proof} Fix an iterate $n\in \N$.  Let $\delta'<\delta$ be as in~\S\ref{subsecn:change}.  By Setup~\ref{setup}\eqref{setup:deltahat} we may choose $\hat{\delta}\ll \delta'$, so that for each branch of $F^{(-n)}$ sending a cusp $x\in \cusps{\cS}$ to a cusp $y\in \cusps{\cS}$ given in local natural cusp coordinates by $z \mapsto az^t+\cdots$, the image of the cusp neighborhood $B(x, \hat{\delta})$  is approximately the cusp neighborhood $B(y,  \hat{\delta}/t)$, with an additive error in the cusp circumference parameter that tends to zero as $\hat{\delta}\to 0$. In the analysis below, we will obtain estimates in which the dependence on $\hat{\delta}$ cancels out. 

Suppose now that $(g_m)_m$ is a tight sequence, and fix $i\in \{1, \ldots,K-1\}$.

Given a thick part $\ell_{i,m}$ of the geodesic $g_m$ connecting a point $Q_{i,m}\in \partial B(x_{i,m}, \delta)$ with a point $Q_{i+1}\in \partial B(x_{i+1},\delta)$, consider the unique bi-infinite geodesic in $\cS$ that is homotopic to the curve $[x_{i,m}Q_{i,m}].\ell_{i,m}.[Q_{i+1,m}x_{i+1,m}]$, where the subcurves in brackets are segments joining the indicated points that are radial in the natural cusp coordinates.  (Formally, one lifts to the universal cover $\IH$.   The endpoints of the lift to $\IH$ of the curve $[x_{i,m}Q_{i,m}].\ell_{i,m}.[Q_{i+1,m}x_{i+1,m}]$ must join two cusps corresponding to $x_{i,m}$ and $x_{i+1,m}$.  The geodesic we seek is the image of the geodesic in $\IH$ joining these two cusps.) This geodesic is a concatenation of an infinite-length segment in $B(x_{i,m},\hat{\delta})$ that is radial in natural cusp coordinates, a finite-length segment $p$ lying outside the union of each of the $\hat{\delta}$-cusp neighborhoods of $\cS$ (i.e., in the $\hat{\delta}$-thick-part), followed by another radial infinite-length segment in $B(x_{i+1,m},\hat{\delta})$.  The length of $p$ is thus at most $\lambda+2\log(\delta/\hat{\delta})$, since, firstly, the length of the $\delta'$-thick part of $p$ is at most the length of $\ell_{i,m}$ which is bounded by $\lambda$ and, secondly, the distance between the boundaries of the $\delta$-horoball and the $\hat{\delta}$-horoball is $\log(\delta/\hat{\delta})$; here, we have used the fact that $p$ is a subsegment of a geodesic joining two cusps. 
 
We now look at a distinguished lift $\undertilde{p}$ of $p$ under $F^{(-n)}$.  See Figure~\ref{fig:homotopy}.
 
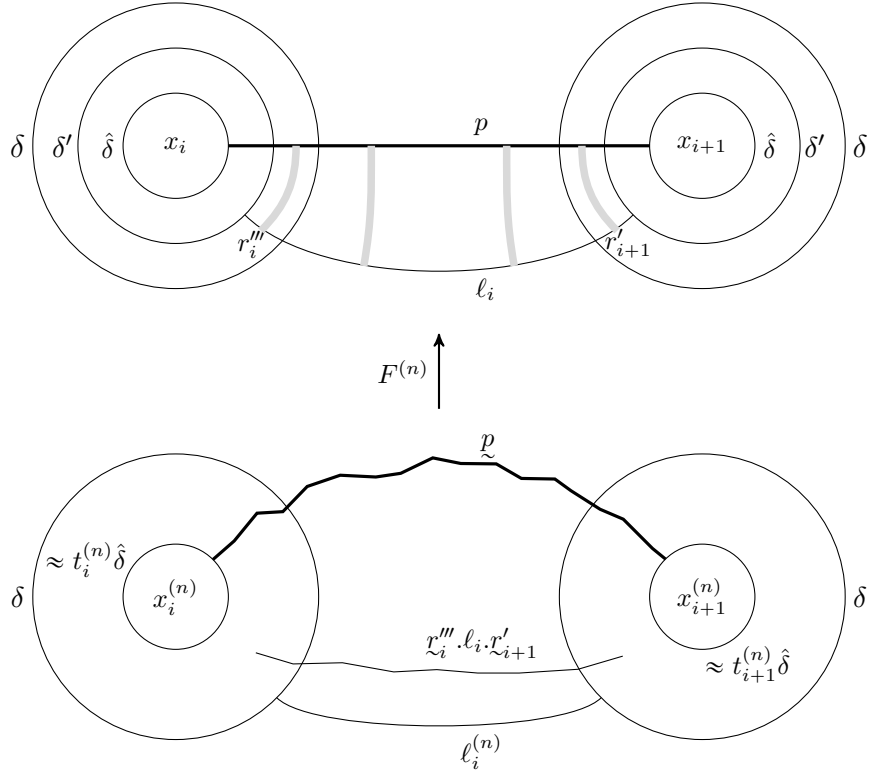
\begin{figure}
\begin{center}
  \begin{tikzpicture}[>=stealth',decoration={markings,mark=at position 0.5 with {\arrow{>}}}]
    \node (I) at (-3.5,0) {$x_i$};
     \draw (I) circle (7mm) circle (13mm) circle (19mm) +(-0.9,0) node {$\hat\delta$} +(-1.5,0) node {\large $\delta'$} +(-2.1,0) node {\Large $\delta$};
     \node (J) at (3.5,0) {$x_{i+1}$};
     \draw (J) circle (7mm) circle (13mm) circle (19mm) +(0.9,0) node {$\hat\delta$} +(1.5,0) node {\large $\delta'$} +(2.1,0) node {\Large $\delta$};
     \draw [very thick,postaction={decorate}] (I.center) +(0.7,0) -- node[pos=0.6,above] {$p$} ($(J.center)+(-0.7,0)$);
     \draw[postaction={decorate}] (I.center) +(-45:1.3) .. controls +(1,-1) and +(-1,-1) .. node [pos=0.03,below] {$r_i'''$} node [pos=0.6,below] {$\ell_i$} node[pos=0.98,below] {$r_{i+1}'$} ($(J.center)+(225:1.3)$);
     \draw[line width=1mm,gray!30] (I.center)+(-43.4:1.65) to[out=58,in=-90] ($(I.center)+(1.6,0)$)
     (-1,-1.58) to[out=80,in=-90] (-0.9,0)
     (1,-1.58) to[out=100,in=-90] (0.9,0)
     (J.center)+(223.4:1.65) to[out=123,in=-90] ($(J.center)+(-1.6,0)$);
     
     \draw[thick,<-] (0,-3.5) -- node [left] {$F^{-n}$} (0,-2.5);
     
     \node (K) at (-3.5,-6) {$x_i^{(n)}$};
     \draw (K) circle (7mm) circle (19mm) +(1.1,-0.5) node[] {$\approx \hat\delta/t_i^{(n)}$} +(-2.1,0) node {\large $\delta$};
     \node (L) at (3.5,-6) {$x_{i+1}^{(n)}$};
     \draw (L) circle (7mm) circle (19mm) +(1.05,-0.6) node[] {$\approx \hat\delta/t_{i+1}^{(n)}$} +(2.1,0) node {\large $\delta$};
     \draw [very thick,decorate,decoration={random steps,segment length=4mm,amplitude=1mm}] (K.center) +(45:0.7) to[out=45,in=135] node[pos=0.6,above] {$\undertilde p$} ($(L.center)+(135:0.7)$);
     \draw [decorate,decoration={random steps,segment length=6mm,amplitude=1mm}] (K.center) +(-35:1.3) .. controls +(1,-0.3) and +(-1,-0.3) .. node[pos=0.6,above] {${\undertilde r}_i'''.\undertilde{\ell}_i.{\undertilde r}'_{i+1}$} ($(L.center)+(215:1.3)$);
     \draw (K.center) +(-45:1.9) .. controls +(0.5,-0.5) and +(-0.5,-0.5) .. node[below,pos=0.6] {$\undertilde{\ell}_i^{(n)}$} ($(L.center)+(225:1.9)$);
 \end{tikzpicture}
 \end{center}
 \caption{Top: dependence on index $m$ suppressed.  The geodesic $p$ joins the cusps $x_i$ and $x_{i+1}$.  The symbols involving $\delta$ are circumferences of cusp neighborhoods.  Some traces of a homotopy between a subsegment of $p$ and $\ell_i$ are indicated in thick gray.  Bottom: the lifts under $F^{-n}$ are indicated with under-tilde's.}
 \label{fig:homotopy}
\end{figure}
 
By construction, $p$ contains a subsegment homotopic to the curve $r_{i}''' . \ell_{i,m}.r_{i+1,m}'$ relative to the union of the $\delta'$-cusp neighborhoods.  Lifting $r_{i}''' . \ell_{i,m}.r_{i+1,m}'$, we get $\undertilde{r}_{i}''' . \undertilde{\ell}_{i,m}.\undertilde{r}_{i+1,m}'$. The curve $\undertilde{r}_{i}''' . \undertilde{\ell}_{i,m}.\undertilde{r}_{i+1,m}'$ joins $B(x_{i}^{(n)},\delta)$ to $B(x_{i+1}^{(n)},\delta)$.   By lifting this homotopy, we obtain a unique lift $\undertilde{p}$ of $p$ that also joins $B(x^{(n)}_{i},\delta)$ to $B(x_{i+1}^{(n)},\delta)$.  From this observation it follows that the endpoints of $p$ are, up to an additive error tending to zero in $\hat{\delta}$, actually in the smaller neighborhoods $B(x^{(n)}_{i},\hat{\delta}/t_{i, n, m})$ and $B(x_{i+1}^{(n)},\hat{\delta}/t^{(n)}_{i+1, m})$. The lift $\undertilde{r}_{i}''' . \undertilde{\ell}_{i,m}.\undertilde{r}_{i+1,m}'$ contains a subcurve homotopic to $\ell_{i,m}^{(n)}$ relative to the $\delta$-cusp neighborhoods, and so the same is true of the lift $\undertilde{p}$.  

We will now estimate the length of $\undertilde{p}$ in two different ways. 

Focusing on the long parts of $\undertilde{p}$ in the thin parts near the cusps, we note that by Lemma~\ref{lem:t}, the first end has a portion joining a horocycle of circumference $\delta$ to one of circumference approximately $\hat{\delta}/t^{(n)}_{i,m}$ and similarly the second end has a portion joining a horocycle of width $\delta$ to one of circumference approximately $\hat{\delta}/t^{(n)}_{i+1,m}$ (these estimates hold up to additive error tending to zero in $\hat{\delta}$). Hence 
  \[|\undertilde{p}|\ge\log(t^{(n)}_{i,m}\delta/\hat{\delta})+\log(t^{(n)}_{i+1,m}\delta/\hat{\delta}).\]
  
Recall that the constant $0<c<1$ denotes an upper bound on the contraction factor of $\rho$ on the thick part of $\cT$ (defined at the start of~\S\ref{secn:proof}), and that the constant $\lambda > 0$ is an upper bound on the length of the thick parts of the $g_m$, from Lemma~\ref{lemma:lambda}.  Focusing now on the thick part of $\undertilde{p}$, we note that $p$ and each of its $F^{(j)}$-lifts, $j=1, \ldots, n-1$, contains a subcurve of length at least $\zeta$ in the thick part of $\cS$, namely one which is homotopic to $\ell_{i,m}^{(j)}$.  Each time we lift by a single iterate $F^{-1}$, the length of such a subcurve is shortened by at least the additive constant $\eta\coloneqq(1-c)\lambda$. Therefore,
  \[|\undertilde{p}|\le|p|-\eta n \le \lambda+2\log(\delta/\hat{\delta})-\eta n.\]
Combining these upper and lower estimates, the dependence on $\hat{\delta}$ vanishes by cancellation, and we conclude $\lambda-\eta n\ge\log(t^{(n)}_{i,m}t^{(n)}_{i+1,m})$ and so we obtain the sharper estimate
  \[t^{(n)}_{i,m}\cdot t^{(n)}_{i+1,m}\le\exp(\lambda-\eta n)\eqqcolon d^{(n)}\to 0\]
  as $n\to\infty$. The same argument gives the estimate $t^{(n)}_{0,m} \leq d^{(n)}\to 0$, by considering the geodesic connecting the basepoint $\star$ to the first cusp $x_{1,m}$.  Looking now at the product of exponents and focusing on their indices $1, \ldots, K$, if $K$ odd, we group the terms as $(1)(2\cdot 3)\cdots ((K-1)\cdot K)$; if $K$ is even, we group them as $(1\cdot 2)\cdots ((K-1)\cdot K)$. We conclude the product of exponents satisfies $t^{(n)}_{1,m}\cdots t^{(n)}_{K,m} < d^{(n)}\to 0$ exponentially fast in $n \gg 1$.
\end{proof}

\subsection{Conclusion of proof of Theorem~\ref{thm:main}}
\label{subsecn:conclusion}

\begin{proof}[Proof of Theorem~\ref{thm:main}]
Consider a very tight sequence $(g_m)$, a corresponding sequence  $(g_m^{(n)})$ of $\sfX$-rays, and for each term $g_m^{(n)}$ its $m$-thick-thin decomposition.  Out of this doubly-indexed sequence $(g_m^{(n)}), m, n\in \N$, we will extract another very tight sequence $h_s\coloneqq g_{m(s)}^{(n(s))}$ in which the lengths of thick parts become unbounded.  This will contradict Lemma~\ref{lemma:ldvt}, completing the proof.

From the asymptotics as $m \to \infty$ for the lengths of $g_m$ and $g_m^{(n)}$, we obtain
\begin{align*}
  |g_m^{(n)}| &= |r_{1,m}^{(n)}|+\cdots+|r_{K,m}^{(n)}| + |\ell_{0,m}^{(n)}|+\cdots+|\ell_{K,m}^{(n)}|, \text{by definition of decomposition}\\
        &= \begin{aligned}[t]
             |r_{1,m}|+\cdots+|r_{K,m}|&+2\log(t^{(n)}_{1,m}\cdot t^{(n)}_{2,m}\cdots t^{(n)}_{K,m})\\
             &+ |\ell_{0,m}^{(n)}|+\cdots+|\ell_{K,m}^{(n)}| + \mathcal O(1)\text{ by~Lemma~\ref{lem:tightlift}(1)}
           \end{aligned}\\
              &= |r_{1,m}|+\cdots+|r_{K,m}|-2\eta n+|\ell_{0,m}^{(n)}|+\cdots+|\ell_{K,m}^{(n)}| + \mathcal O(1)\text{ by Lemma~\ref{lem:tdown}\eqref{eq:lem:tdown}}\\
              &= |g_m|-2\eta n+|\ell_{0,m}^{(n)}|+\cdots+|\ell_{K,m}^{(n)}| + \mathcal O(1)\text{ by Proposition~\ref{prop:tightfacts}}.
\end{align*}
Write $L_m^{(n)}=|\ell_{0,m}^{(n)}|+\cdots+|\ell_{K,m}^{(n)}|$. Since $(g_m)$ is tight, $|g_m^{(n)}|\ge|g_m|-\mathcal O(1)$ for all $n\le m$, so $L_m^{(n)}-2\eta n$ is bounded. On the other hand, Corollary~\ref{cor:m_n} implies the existence of subsequences $(m(s)),(n(s))$ such that $L_{m(s)}^{(n(s))}$ is bounded. This contradiction concludes the proof.
\end{proof}

\section{Modeling graph pullbacks via $\sfX$-rays}\label{ss:gga}
We consider in this section a Thurston map $f\colon(S^2,P)\to (S^2, P)$, which we assume is \emph{nonexceptional}: it has degree $\ge2$, and $(S^2,P)$ is not double-covered by a torus on which $f$ lifts to an affine map.  We denote by $G\coloneqq\Mod(S^2,P)$ the pure mapping class group and by $B\coloneqq B^{\text{maps}}$ its Hurwitz biset of Thurston maps; recall from Proposition~\ref{prop:identify} that it is the biset of the correspondence on moduli space associated with $f$.

\subsection{Planar graphs} 
\begin{defn}\label{def:essentialgraph}
  A \emph{planar graph} is a finite graph $\Gamma=(V,E)$ embedded in $S^2$, with $V$ nonempty, such that $\Gamma \cap P \subset V$. Its underlying subset of $S^2$ is denoted $|\Gamma|$.  Its \emph{complexity} $\#\Gamma$ is defined as its total number of vertices and edges: $\#\Gamma\coloneqq\#V+\#E$.   Two planar graphs are \emph{equivalent} if there is an isotopy of $S^2$ fixing $P$ sending one to the other and inducing a bijection between their vertices and between their edges. 
\end{defn}
The definition of equivalence does not allow for any of the following: a collapse of inessential loops, coalescence of two isotopic components into a single component, or addition or deletion of vertices of valence two. We put no other restrictions on planar graphs.  

Here is an example. Suppose $C$ is a simple closed curve in $S^2\setminus P$.  Choose any point of $C$, and decree it to be a vertex.  This gives a planar graph $\Gamma$ with one vertex and one edge, which is a loop. 

We denote by $\graphs(P,M)$ the set of equivalence classes of graphs, of complexity at most $M$. Note that this is typically an infinite set. 

\begin{prop}\label{prop:gorb}
For each $M\in \N$, the group $G$ acts on the set $\graphs(P,M)$  with finitely many orbits. 
\end{prop}

\begin{proof} We use induction on the complexity.  The conclusion is clear for the base case of a graph consisting of a single vertex. The inductive step breaks down into cases according to whether we add a vertex or an edge.  The only case of substance is that of adding an edge between two existing vertices. A standard application of the Alexander method~\cite{farb:margalit:mcg}*{\S2.3} yields the result.
\end{proof}

\begin{defn}\label{subgraphs}
A \emph{subgraph} of a graph $\Gamma=(V,E)$ is a graph $\Gamma' = (V', E')$ for which $|\Gamma'| \subset |\Gamma|$.  We denote this by $\Gamma' \subset \Gamma$.
\end{defn}

The full $f$-preimage $f^{-1}(\Gamma)$ of a planar graph $\Gamma$ is again a planar graph, and  its complexity satisfies $\#f^{-1}(\Gamma) \leq \deg(f)\cdot \#\Gamma$. The isotopy class of $f^{-1}(\Gamma)$ depends only on that of $\Gamma$. A \emph{sublift} of the planar graph $f^{-1}(\Gamma)$ is a subgraph $\Delta \subset f^{-1}(\Gamma)$. If $\Delta$ is any sublift of $\Gamma$ then its complexity satisfies $\#\Delta \leq \deg(f) \cdot \#\Gamma$. 

\subsection{Bounded orbits under pullback}
Fix an integer $M \geq 1$.  An \emph{$M$-bounded $f$-orbit of graphs} is a sequence $(\Gamma_0,\Gamma_1,\dots)$ of planar graphs such that for all $n\in\N$, $\Gamma_n\in \graphs(P,M)$, and $\Gamma_{n+1}$ is isotopic to a subgraph of $f^{-1}(\Gamma_n)$. In our terminology, $\Gamma_{n+1}$ is a \emph{sublift} of $\Gamma_n$.

For example, if $C_0, C_1,\ldots $ is a sequence of iterated preimages of the curve $C_0$ under pullback by $f$, then by decorating each $C_i$ with a single vertex, we obtain an associated $2$-bounded-orbit of graphs $\Gamma_0, \Gamma_1, \ldots$.

We next relate $M$-bounded $f$-orbits of graphs to $\sfX$-rays in the direction of $f$. 

Let $F=\phi, \rho\colon \cT \rightrightarrows\cS$ be the covering on moduli space determined by $f$; we denote also by $f$ the path in moduli space determined by $f$; it joins the basepoint $\star=[P \hookrightarrow \Pone]$ to the element of $F^{-1}(\star)$ determined by $f$.

\begin{prop}\label{prop:getxray}
Given a complexity bound $M$, let $\graphs(P,M)_0 \subset \graphs(P,M)$ be a finite set of graphs comprising a right orbit transversal for the action of $G$ on $\graphs(P,M)$.  Then given a Thurston map $f$, there is a finite set $\sfX$ of Thurston maps with the following property.

For any $M$-bounded $f$-orbit of graphs $\Gamma_0, \Gamma_1, \ldots$, there exists $g=g^{(0)}\in G$ with $\Gamma^{(0)}\coloneqq g(\Gamma_0)\in \graphs(P,M)/G$, and an $\sfX$-ray $(g^{(n)})$ in the direction of $f$ above $g$, such that for each $n\in \N$, we have $\Gamma^{(n)}\coloneqq g^{(n)}(\Gamma_n)\in \graphs(P,M)/G$.
\end{prop}    

\begin{proof}
  Recall from~\eqref{eq:H_f} the subgroup $H_f\le G$ of ``liftable elements'' under $f$, namely with the property that for every $h\in H_f$ there exists a $g\in G$ with $g \circ f=f \circ h$. Furthermore, this $g$ is unique, and is denoted by $\widetilde h$. Note that in the biset $B$ this equation becomes $f \cdot g = \widetilde{h} \cdot f$.  Choose a right transversal $T$ to $H_f$ in $G$.
  
By Proposition~\ref{prop:gorb}, there exists a finite orbit transversal $\graphs(P,M)/G$ to the action of $G$ on graphs of complexity at most $M$.  For each pair $(t, \Gamma)\in T \times \graphs(P,M)/G$, consider the graph $t^{-1}\Gamma$.  Suppose $\Delta$ is a sublift of $t^{-1}\Gamma$ under $f$ whose complexity satisfies $\#\Delta \leq M$. Then there exists $s=s(t, \Gamma, \Delta)\in G$ such that $s\Delta\in \graphs(P,M)/G$.  We fix a choice of such an $s$ for each triple $(t, \Gamma, \Delta)$. The set of triples $(t, \Gamma, \Delta)$ is finite, so the set of such $s(t,\Gamma,\Delta)$ is finite as well; we denote it by $S$. We set 
\[ \sfX\coloneqq\{s^{-1}\cdot f \cdot t | t\in T, s\in S\} \subseteq B\]
and we denote its elements as usual by symbols $\sfx$. 
  
 Consider now an $M$-bounded $f$-orbit of graphs $\Gamma_0, \Gamma_1, \ldots$.  We will construct an $\sfX$-ray of loops from this data. Using Proposition~\ref{prop:gorb} we first choose $g^{(0)}\in G$ with $\Gamma^{(0)}\coloneqq g^{(0)}\Gamma_0\in \graphs(P,M)/G$.  Once this choice is made, we will show that the rest of the sequence $g^{(n)}, n \geq 1$ is determined by the sequence of graphs $\Gamma_1, \Gamma_2, \ldots$ comprising the $M$-bounded $f$-orbit of graphs that we are given and the choices of $T$ and $S$.   
  
 For the base case, we factor $g^{(0)}=t_0\circ h_0$ where $h_0\in H_f$ is liftable under $f$ and $t_0\in T$.  We then lift inductively by iterates of $f$. At the $n$th stage, we  have the following commutative diagram:

  \[\begin{tikzcd}
(S^2,\Gamma_{n+1}) \arrow[d,"f"] \arrow[r,"\tilde{h}_n" below] \arrow[rr,bend left=15,"g^{(n+1)}"] &(S^2,\Delta) \arrow[d,"f"] \arrow[r,"s_n" below] & (S^2,\Gamma^{(n+1)}) \arrow[d,"\sfx_n"] \\
(S^2,\Gamma_n) \arrow[r,"h_n"] \arrow[rr,bend right=15,"g^{(n)}" below] &(S^2,t_n^{-1}\Gamma^{(n)}) \arrow[r,"t_n"] & (S^2,\Gamma^{(n)}).
\end{tikzcd}
\]
In the above diagram, we assume $g^{(n)}$ is given, with the property that $\Gamma^{(n)}\coloneqq g^{(n)}\Gamma_n$ belongs to $\graphs(P,M)/G$.  We factor  $g^{(n)}=t_n\circ h_n$ uniquely, where $h_n\in H_r$ is liftable and $t_n\in T$.  Thus the left square commutes, where $\tilde{h}_n\in G$ is the lift of $h_n$ under $f$.  By assumption, $\#\Gamma_{n+1}\leq M$ and so $\tilde{h}_n\Gamma_{n+1}$ is one of the graphs $\Delta$ arising in the definition of the set $S$.  Thus there exists a unique $s_n\in S$ with $s_n\Delta\in \graphs(P,M)/G$.  By definition, the composition $s_n^{-1}\cdot f \cdot t_n$ defines an element of $\sfX$ which we denote by $\sfx_n$.  Thus the the right-hand square commutes by definition. Finally, we put $g^{(n+1)}\coloneqq s_n\circ \tilde{h}_n$.  Then by construction $f\cdot g^{(n)}=g^{(n+1)}\cdot \sfx_n$ as elements of $B$. Inductively, this holds for each $n$, so the sequence $g^{(0)}, g^{(1)}, \ldots$ is an $\sfX$-ray above $g^{(0)}$.  
\end{proof}

\subsection{Finite graph attractor}
We are now ready to state the main result of this section. 

\begin{thm}[Graph attractors] \label{thm:gga}
Suppose $\#P=4$ and $f\colon (\Pone,P)\to (\Pone, P)$ is a nonexceptional post-critically finite rational map whose postcritical set lies in $P$.  Then for every $M\in\N$ there exists a finite collection $\graphs(P,M,f)\subset \graphs(P,M)$ of planar graphs, each of complexity at most $M$, closed under sublifting by $f$,   such that, for every $M$-bounded orbit $(\Gamma_1, \Gamma_2, \ldots)$ of graphs, we have $\Gamma_n\in\graphs(P,M,f)$ for all $n$ large enough.

In particular, $\graphs(P,M,f)$ naturally forms the stateset of a finite automaton, with states given by $G\coloneqq\Mod(\Pone, P)$-orbits, and transitions given by taking all possible sublifts.
\end{thm}
\begin{proof}
We denote by $F=\phi, \rho\colon \cT \rightrightarrows\cS$ the correspondence on moduli space associated to $f$ and by $\star$ the fixed point corresponding to $f$.  
 By Proposition~\ref{prop:getxray}, there is a finite set $\graphs(P,M)/G$ and a collection of curves (Thurston maps) $\sfX$ such that for any $M$-bounded $f$-orbit of graphs $\Gamma_0, \Gamma_1, \ldots $ there is an $\sfX$-ray $(g^{(n)})$ with $g^{(n)}(\Gamma_n)\in \graphs(P,M)/G$ for each $n$.  
 
 By Corollary~\ref{cor:contraction} , there is a finite set $A \subset G$ independent of $g$ and of $\Gamma_0$ such that $g^{(n)}\in A$ for all $n$ sufficiently large. We put $\graphs(P,M,f)\coloneqq A^{-1}\graphs(P,M)/G$.  The definitions give $\Gamma_n\in \graphs(P,M,f)$ for all $n$ sufficiently large. 
\end{proof}

\subsection{Proof of Corollary~\ref{cor:attractor}}\label{ss:appinvgr}
We recall the statement:
\setcounter{mainthm}{1}
\begin{maincor}
  Assume that $\#P=4$ and $f$ is a non-Latt\`es rational map. Then there is a finite attractor for the pullback iteration on multicurves, on trees, and on spines.
\end{maincor}

For general $\#P$, we will show that in each of the settings of multicurves, trees, and spines, we can associate a $G$-invariant collection of graphs and a notion of pullback such that the graph complexity remains bounded under iteration.  Upon restricting to the case $\#P=4$, the Corollary then follows from Theorem~\ref{thm:gga}.

\begin{proof}[Proof for multicurves]
By adding a vertex to each component, the set of isotopy classes of multicurves in $S^2-P$ embeds into the family $\graphs(P,M)$ where $M=2(\#P-3)$. The set of multicurves is $G$-invariant; with finitely many orbits.  We define pullback of such graphs as follows. Given a multicurve $\Gamma$, we consider $f^{-1}(\Gamma)$, and modify it  as follows: (i) delete all inessential and peripheral elements; (ii) consolidate essential components of the preimage which are isotopic in $S^2-P$ into a single component; (iii) lift the graph structure from $\Gamma$; and finally (iv) delete all but one vertex.  Defined in this way, the terms of pullback orbits have complexity at most $M$. 
\end{proof}

\begin{proof}[Proof for spines]
  A \emph{spine} for $S^2-P$ is a finite 3-valent graph in $S^2-P$ which, as a subset of the sphere, arises as the image of a deformation retract of $S^2-P$; thus each complementary region in $S^2$ is topologically a disk containing exactly one element of a $P$. From Euler's Formula and being 3-valent, it is easy to see that the complexity of a spine is $M\coloneqq5\#P-10$.  The set of spines is $G$-invariant; thus by Proposition~\ref{prop:gorb}, there are finitely many $G$-orbits of spines.  

\begin{lem}\label{lemma:eraseone}
    The following properties of spines hold:
\begin{enumerate}
\item Suppose $Q \subset S^2$ is finite, with $\#Q \geq 3$, and $q\in Q$. Then any spine $\Gamma$ for $S^2-Q$ contains a spine for $S^2-(Q - \{q\})$. 
\item Suppose $P \subset Q$ are finite, with $\#P \geq 2$. Then any spine $\Gamma$ for $S^2-Q$ contains a spine for $S^2-P$.  
\item If $f\colon (S^2,P)\to (S^2,P)$ is a Thurston map and $\Gamma$ is a spine for $S^2-P$, then $f^{-1}(\Gamma)$ contains a spine for $S^2-P$. 
\end{enumerate}
\end{lem}

\begin{proof}
  For (1), let $U_q$ be the unique face of $\Gamma$ containing $q$. Since $\Gamma$ is a spine, there exists $p\in Q - \{q\}$ such that the face $U_p$ containing $p$ has the property that $\overline{U}_p \cap \overline{U}_q$ contains an edge $e \subset \Gamma$.  Since $p \neq q$, the set $\Gamma'\coloneqq\Gamma - \{\text{interior}(e)\}$ is connected. If $e$ is not a loop, we declare the common endpoint of its ends to be an ordinary point, and not a vertex; we obtain a new spine $\Gamma'$. If $e$ is a loop at $v$, we further prune $\Gamma'$ by also deleting the dangling edge $e'$ joining $v$ (which is now valence $1$ in $\Gamma'$) to say $w$, and finally declaring $w$ to be an ordinary point, and not a vertex; we obtain a new spine $\Gamma''$. Statement (2) then follows by (1) and induction. Statement (3) follows from (2) upon setting $Q\coloneqq f^{-1}(P)$.
\end{proof}

Thus for a spine $\Gamma$, we define pullback by applying Lemma~\ref{lemma:eraseone}, choosing any subgraph of $f^{-1}(\Gamma)$ which is again a spine.
\end{proof}

\begin{proof}[Proof for trees]
Given $P \subset S^2$ finite with $\#P \geq 2$, an \emph{admissible tree containing $P$} is a tree containing $P$ such that any vertex of valence $1$ or $2$ lies in $P$.  The complexity of such a tree is at most $M\coloneqq4\#P-1$.  To see this: for a tree $T$, we denote by $C(T)$ its complexity, and by $V_k(T)$ the set of vertices of valence $k$.  Let $T'$ be the tree obtained from $T$ by deleting vertices of valence $2$.  Thus $C(T) \leq C(T')+2\#V_2(T) \leq C(T')+2\#P$.  An easy induction argument shows that the number of interior vertices of $T'$ is at most the number of leaves of $T'$, yielding $\sum_{k \geq 3}V_k(T') \leq \#P$. Thus $\#V(T')=\sum_{k \geq 1}V_k(T') \leq 2\#P$. Euler's Formula gives $\#E(T')=\#V(T')-1$ and so $C(T') \leq 2\#P-1$. Combined with the first estimate, this yields the claimed complexity bound $M$. In particular  the set of such trees is $G$-invariant, with finitely many orbits. 

\begin{lem}\label{lem:trees}
  The following properties of trees hold:
\begin{enumerate}
\item Suppose $Q \subset S^2$ is finite, with $\#Q \geq 3$, and $q\in Q$. Then any admissible tree $\Gamma$ containing $Q$ contains an admissible tree containing $Q-\{q\}$. 
\item Suppose $P \subset Q$ are finite, with $\#P \geq 2$. Then any admissible tree $\Gamma$ containing $Q$ contains an admissible tree containing $P$.  
\item If $f\colon (S^2,P)\to (S^2,P)$ is a Thurston map and $\Gamma$ is an admissible tree containing $P$, then $f^{-1}(\Gamma)$ contains an admissible tree containing $P$.
\end{enumerate}
\end{lem}
\begin{proof}
  (1) If $\text{val}(q)\geq 3$ we set $\Gamma'\coloneqq\Gamma$ as graph, and we are done. If $\text{val}(q)=2$ we set $\Gamma'\coloneqq\Gamma$ as graph but decree $q$ to be an ordinary point and not a vertex. If $\text{val}(q)=1$ we delete $q$ and the unique edge $e$ incident to it to obtain $\Gamma'$ as a set.  Denoting $p$ the other vertex of $e$, consider the valence $v$ of $p$ in the original graph $\Gamma$. We have $v \geq 2$ since $\#Q \geq 3$. If $v=2$ then $p\in Q$ and we stop with $\Gamma'$. If $v=3$ and $p\in Q$ then we stop with $\Gamma'$. If $v=3$ and $p \not\in Q$ then we decree $p$ to be an ordinary point of $\Gamma'$, and stop. If $v \geq 4$ then we stop with $\Gamma'$. Statement (2) follows from (1) and induction.  For statement (3), we apply (2) to any spanning tree containing $Q\coloneqq f^{-1}(P)$ in the connected graph $f^{-1}(\Gamma)$. 
\end{proof}

Thus for an admissible tree $T$ containing $P$, we define pullback by applying Lemma~\ref{lem:trees}, choosing any subtree of $f^{-1}(T)$ which is again an admissible tree.
\end{proof}

\bibliographystyle{amsalpha}
\bibliography{lcrefs.bib}

\begin{bibdiv}
\begin{biblist}

\bib{MR3781419}{article}{
      author={Bartholdi, Laurent},
      author={Dudko, Dzmitry},
       title={Algorithmic aspects of branched coverings {I}/{V}. {V}an
  {K}ampen's theorem for bisets},
        date={2018},
        ISSN={1661-7207},
     journal={Groups Geom. Dyn.},
      volume={12},
      number={1},
       pages={121\ndash 172},
         url={https://doi-org.proxyiub.uits.iu.edu/10.4171/GGD/441},
      review={\MR{3781419}},
}

\bib{BD4}{article}{
      author={Bartholdi, Laurent},
      author={Dudko, Dzmitry},
       title={Algorithmic aspects of branched coverings {IV}/{V}. {E}xpanding
  maps},
        date={2018},
        ISSN={0002-9947,1088-6850},
     journal={Trans. Amer. Math. Soc.},
      volume={370},
      number={11},
       pages={7679\ndash 7714},
         url={https://doi.org/10.1090/tran/7199},
      review={\MR{3852445}},
}

\bib{MR4213769}{article}{
      author={Bartholdi, Laurent},
      author={Dudko, Dzmitry},
       title={Algorithmic aspects of branched coverings {II}/{V}: sphere bisets
  and decidability of {T}hurston equivalence},
        date={2021},
        ISSN={0020-9910},
     journal={Invent. Math.},
      volume={223},
      number={3},
       pages={895\ndash 994},
         url={https://doi-org.proxyiub.uits.iu.edu/10.1007/s00222-020-00995-2},
      review={\MR{4213769}},
}

\bib{bartholdi:nekrashevych:twisted}{article}{
      author={Bartholdi, Laurent},
      author={Nekrashevych, Volodymyr},
       title={Thurston equivalence of topological polynomials},
        date={2006},
        ISSN={0001-5962},
     journal={Acta Math.},
      volume={197},
      number={1},
       pages={1\ndash 51},
      review={\MR{MR2285317}},
}

\bib{MR4510015}{article}{
      author={Belk, James},
      author={Lanier, Justin},
      author={Margalit, Dan},
      author={Winarski, Rebecca~R.},
       title={Recognizing topological polynomials by lifting trees},
        date={2022},
        ISSN={0012-7094},
     journal={Duke Math. J.},
      volume={171},
      number={17},
       pages={3401\ndash 3480},
         url={https://doi-org.proxyiub.uits.iu.edu/10.1215/00127094-2022-0043},
      review={\MR{4510015}},
}

\bib{Belk:2023aa}{article}{
      author={Belk, James},
      author={Margalit, Dan},
      author={Winarski, Rebecca~R.},
       title={Thurston's theorem and the nielsen-thurston classification via
  teichm{\"u}ller's theorem},
        date={202309},
      eprint={2309.06993},
         url={https://arxiv.org/pdf/2309.06993.pdf},
}

\bib{MR3727134}{book}{
      author={Bonk, Mario},
      author={Meyer, Daniel},
       title={Expanding {T}hurston maps},
      series={Mathematical Surveys and Monographs},
   publisher={American Mathematical Society, Providence, RI},
        date={2017},
      volume={225},
        ISBN={978-0-8218-7554-4},
         url={https://doi.org/10.1090/surv/225},
      review={\MR{3727134}},
}

\bib{MR2508269}{incollection}{
      author={Buff, Xavier},
      author={Epstein, Adam},
      author={Koch, Sarah},
      author={Pilgrim, Kevin},
       title={On {T}hurston's pullback map},
        date={2009},
   booktitle={Complex dynamics},
   publisher={A K Peters, Wellesley, MA},
       pages={561\ndash 583},
         url={http://dx.doi.org/10.1201/b10617-20},
      review={\MR{2508269 (2010g:37071)}},
}

\bib{MR1168195}{article}{
      author={Bullett, Shaun},
       title={Critically finite correspondences and subgroups of the modular
  group},
        date={1992},
        ISSN={0024-6115},
     journal={Proc. London Math. Soc. (3)},
      volume={65},
      number={2},
       pages={423\ndash 448},
         url={https://doi-org.proxyiub.uits.iu.edu/10.1112/plms/s3-65.2.423},
      review={\MR{1168195}},
}

\bib{MR4071411}{article}{
      author={Bullett, Shaun},
      author={Lomonaco, Luna},
       title={Mating quadratic maps with the modular group {II}},
        date={2020},
        ISSN={0020-9910},
     journal={Invent. Math.},
      volume={220},
      number={1},
       pages={185\ndash 210},
         url={https://doi-org.proxyiub.uits.iu.edu/10.1007/s00222-019-00927-9},
      review={\MR{4071411}},
}

\bib{MR4505392}{article}{
      author={Bullett, Shaun},
      author={Lomonaco, Luna},
       title={Dynamics of modular matings},
        date={2022},
        ISSN={0001-8708},
     journal={Adv. Math.},
      volume={410},
      number={part B},
       pages={Paper No. 108758, 43},
         url={https://doi-org.proxyiub.uits.iu.edu/10.1016/j.aim.2022.108758},
      review={\MR{4505392}},
}

\bib{MR1370680}{incollection}{
      author={Bullett, Shaun},
      author={Penrose, Christopher},
       title={Dynamics of holomorphic correspondences},
        date={1995},
   booktitle={X{I}th {I}nternational {C}ongress of {M}athematical {P}hysics
  ({P}aris, 1994)},
   publisher={Int. Press, Cambridge, MA},
       pages={261\ndash 272},
  url={https://mathscinet-ams-org.proxyiub.uits.iu.edu/mathscinet-getitem?mr=1370680},
      review={\MR{1370680}},
}

\bib{kmp:fixed}{article}{
      author={Cordwell, Kristin},
      author={Gilbertson, Selina},
      author={Nuechterlein, Nicholas},
      author={Pilgrim, Kevin~M.},
      author={Pinella, Samantha},
       title={On the classification of critically fixed rational maps},
        date={2015},
        ISSN={1088-4173},
     journal={Conform. Geom. Dyn.},
      volume={19},
       pages={51\ndash 94},
  url={http://dx.doi.org.proxyiub.uits.iu.edu/10.1090/S1088-4173-2015-00275-1},
      review={\MR{3323420}},
}

\bib{MR4423808}{article}{
      author={Cui, Guizhen},
      author={Gao, Yan},
      author={Zeng, Jinsong},
       title={Invariant graphs of rational maps},
        date={2022},
        ISSN={0001-8708},
     journal={Adv. Math.},
      volume={404},
      number={part B},
       pages={Paper No. 108454, 50},
         url={https://doi-org.proxyiub.uits.iu.edu/10.1016/j.aim.2022.108454},
      review={\MR{4423808}},
}

\bib{cui2024invariantgraphsjuliasets}{misc}{
      author={Cui, Guizhen},
      author={Gao, Yan},
      author={Zeng, Jinsong},
       title={Invariant graphs in julia sets and decompositions of rational
  maps},
        date={2024},
         url={https://arxiv.org/abs/2408.12371},
        note={https://arxiv.org/abs/2408.12371},
}

\bib{MR4104599}{article}{
      author={Dinh, Tien-Cuong},
      author={Kaufmann, Lucas},
      author={Wu, Hao},
       title={Dynamics of holomorphic correspondences on {R}iemann surfaces},
        date={2020},
        ISSN={0129-167X},
     journal={Internat. J. Math.},
      volume={31},
      number={5},
       pages={2050036, 21},
         url={https://doi-org.proxyiub.uits.iu.edu/10.1142/S0129167X20500366},
      review={\MR{4104599}},
}

\bib{MR3667901}{article}{
      author={Dinh, Tien-Cuong},
      author={Sibony, Nessim},
       title={Equidistribution problems in complex dynamics of higher
  dimension},
        date={2017},
        ISSN={0129-167X,1793-6519},
     journal={Internat. J. Math.},
      volume={28},
      number={7},
       pages={1750057, 31},
         url={https://doi.org/10.1142/S0129167X17500574},
      review={\MR{3667901}},
}

\bib{MR1251582}{article}{
      author={Douady, Adrien},
      author={Hubbard, John~H.},
       title={A proof of {T}hurston's topological characterization of rational
  functions},
        date={1993},
        ISSN={0001-5962},
     journal={Acta Math.},
      volume={171},
      number={2},
       pages={263\ndash 297},
         url={https://doi-org.proxyiub.uits.iu.edu/10.1007/BF02392534},
      review={\MR{1251582}},
}

\bib{farb:margalit:mcg}{book}{
      author={Farb, Benson},
      author={Margalit, Dan},
       title={A primer on mapping class groups},
      series={Princeton Mathematical Series},
   publisher={Princeton University Press},
     address={Princeton, NJ},
        date={2012},
      volume={49},
        ISBN={978-0-691-14794-9},
      review={\MR{2850125 (2012h:57032)}},
}

\bib{kmp:origami}{article}{
      author={Floyd, William},
      author={Kelsey, Gregory},
      author={Koch, Sarah},
      author={Lodge, Russell},
      author={Parry, Walter},
      author={Pilgrim, Kevin~M.},
      author={Saenz, Edgar},
       title={Origami, affine maps, and complex dynamics},
        date={2016},
     journal={Arnold Mathematical Journal},
      volume={3},
      number={3},
       pages={365\ndash 395},
      eprint={1612.06449},
         url={https://arxiv.org/abs/1612.06449},
}

\bib{MR4278340}{article}{
      author={Floyd, William},
      author={Parry, Walter},
      author={Pilgrim, Kevin~M.},
       title={Rationality is decidable for nearly {E}uclidean {T}hurston maps},
        date={2021},
        ISSN={0046-5755},
     journal={Geom. Dedicata},
      volume={213},
       pages={487\ndash 512},
         url={https://doi-org.proxyiub.uits.iu.edu/10.1007/s10711-020-00593-9},
      review={\MR{4278340}},
}

\bib{MR3774834}{article}{
      author={Gao, Yan},
      author={Ha\"{\i}ssinsky, Peter},
      author={Meyer, Daniel},
      author={Zeng, Jinsong},
       title={Invariant {J}ordan curves of {S}ierpi\'{n}ski carpet rational
  maps},
        date={2018},
        ISSN={0143-3857},
     journal={Ergodic Theory Dynam. Systems},
      volume={38},
      number={2},
       pages={583\ndash 600},
         url={https://doi.org/10.1017/etds.2016.47},
      review={\MR{3774834}},
}

\bib{Hlushchanka:2019aa}{article}{
      author={Hlushchanka, Mikhail},
       title={Tischler graphs of critically fixed rational maps and their
  applications},
        date={201904},
      eprint={1904.04759},
         url={https://arxiv.org/pdf/1904.04759.pdf},
}

\bib{Hlushchanka:2022aa}{article}{
      author={Hlushchanka, Mikhail},
      author={Prochorov, Nikolai},
       title={Critically fixed thurston maps: classification, recognition, and
  twisting},
        date={202212},
      eprint={2212.14759},
         url={https://arxiv.org/pdf/2212.14759.pdf},
}

\bib{HSS}{article}{
      author={Hubbard, John},
      author={Schleicher, Dierk},
      author={Shishikura, Mitsuhiro},
       title={Exponential {T}hurston maps and limits of quadratic
  differentials},
        date={2009},
        ISSN={0894-0347,1088-6834},
     journal={J. Amer. Math. Soc.},
      volume={22},
      number={1},
       pages={77\ndash 117},
         url={https://doi.org/10.1090/S0894-0347-08-00609-7},
      review={\MR{2449055}},
}

\bib{kmp:kps}{article}{
      author={Koch, Sarah},
      author={Pilgrim, Kevin~M.},
      author={Selinger, Nikita},
       title={Pullback invariants of {T}hurston maps},
        date={2016},
        ISSN={0002-9947},
     journal={Trans. Amer. Math. Soc.},
      volume={368},
      number={7},
       pages={4621\ndash 4655},
         url={http://dx.doi.org/10.1090/tran/6482},
      review={\MR{3456156}},
}

\bib{MR3063048}{article}{
      author={Lodge, Russell},
       title={Boundary values of the {T}hurston pullback map},
        date={2013},
     journal={Conform. Geom. Dyn.},
      volume={17},
       pages={77\ndash 118},
  url={https://doi-org.proxyiub.uits.iu.edu/10.1090/S1088-4173-2013-00255-5},
      review={\MR{3063048}},
}

\bib{MR4543152}{article}{
      author={Lodge, Russell},
      author={Lyubich, Mikhail},
      author={Merenkov, Sergei},
      author={Mukherjee, Sabyasachi},
       title={On dynamical gaskets generated by rational maps, {K}leinian
  groups, and {S}chwarz reflections},
        date={2023},
     journal={Conform. Geom. Dyn.},
      volume={27},
       pages={1\ndash 54},
         url={https://doi-org.proxyiub.uits.iu.edu/10.1090/ecgd/379},
      review={\MR{4543152}},
}

\bib{ctm:renorm}{book}{
      author={McMullen, Curtis~T.},
       title={Complex dynamics and renormalization},
   publisher={Princeton University Press},
     address={Princeton, NJ},
        date={1994},
        ISBN={0-691-02982-2; 0-691-02981-4},
      review={\MR{96b:58097}},
}

\bib{milnor:dynamics}{book}{
      author={Milnor, John},
       title={Dynamics in one complex variable},
     edition={Third},
      series={Annals of Mathematics Studies},
   publisher={Princeton University Press},
     address={Princeton, NJ},
        date={2006},
      volume={160},
        ISBN={978-0-691-12488-9; 0-691-12488-4},
      review={\MR{MR2193309 (2006g:37070)}},
}

\bib{MR1390649}{article}{
      author={Minsky, Yair~N.},
       title={Extremal length estimates and product regions in
  {T}eichm\"{u}ller space},
        date={1996},
        ISSN={0012-7094},
     journal={Duke Math. J.},
      volume={83},
      number={2},
       pages={249\ndash 286},
  url={https://doi-org.proxyiub.uits.iu.edu/10.1215/S0012-7094-96-08310-6},
      review={\MR{1390649}},
}

\bib{MR3199801}{article}{
      author={Nekrashevych, Volodymyr},
       title={Combinatorial models of expanding dynamical systems},
        date={2014},
        ISSN={0143-3857},
     journal={Ergodic Theory Dynam. Systems},
      volume={34},
      number={3},
       pages={938\ndash 985},
         url={http://dx.doi.org/10.1017/etds.2012.163},
      review={\MR{3199801}},
}

\bib{MR2020454}{book}{
      author={Pilgrim, Kevin~M.},
       title={Combinations of complex dynamical systems},
      series={Lecture Notes in Mathematics},
   publisher={Springer-Verlag, Berlin},
        date={2003},
      volume={1827},
        ISBN={3-540-20173-4},
         url={https://doi.org/10.1007/b14147},
      review={\MR{2020454}},
}

\bib{kmp:tw:published}{article}{
      author={Pilgrim, Kevin~M.},
       title={An algebraic formulation of {T}hurston's characterization of
  rational functions},
        date={2012},
        ISSN={0240-2963},
     journal={Ann. Fac. Sci. Toulouse Math. (6)},
      volume={21},
      number={5},
       pages={1033\ndash 1068},
         url={http://dx.doi.org/10.5802/afst.1361},
      review={\MR{3088266}},
}

\bib{MR4472056}{incollection}{
      author={Pilgrim, Kevin~M.},
       title={On the pullback relation on curves induced by a {T}hurston map},
        date={[2022] \copyright 2022},
   booktitle={In the tradition of {T}hurston {II}. {G}eometry and groups},
   publisher={Springer, Cham},
       pages={385\ndash 399},
  url={https://doi-org.proxyiub.uits.iu.edu/10.1007/978-3-030-97560-9_11},
      review={\MR{4472056}},
}

\bib{MR2650982}{article}{
      author={Poirier, Alfredo},
       title={Hubbard trees},
        date={2010},
        ISSN={0016-2736},
     journal={Fund. Math.},
      volume={208},
      number={3},
       pages={193\ndash 248},
         url={https://doi-org.proxyiub.uits.iu.edu/10.4064/fm208-3-1},
      review={\MR{2650982}},
}

\bib{selinger:pullback}{article}{
      author={Selinger, Nikita},
       title={Thurston's pullback map on the augmented {T}eichm\"uller space
  and applications},
        date={2012},
        ISSN={0020-9910},
     journal={Invent. Math.},
      volume={189},
      number={1},
       pages={111\ndash 142},
         url={http://dx.doi.org/10.1007/s00222-011-0362-3},
      review={\MR{2929084}},
}

\bib{MR3071467}{article}{
      author={Selinger, Nikita},
       title={Topological characterization of canonical {T}hurston
  obstructions},
        date={2013},
        ISSN={1930-5311},
     journal={J. Mod. Dyn.},
      volume={7},
      number={1},
       pages={99\ndash 117},
         url={https://doi-org.proxyiub.uits.iu.edu/10.3934/jmd.2013.7.99},
      review={\MR{3071467}},
}

\bib{MR3443373}{article}{
      author={Selinger, Nikita},
      author={Yampolsky, Michael},
       title={Constructive geometrization of {T}hurston maps},
        date={2015},
        ISSN={0706-1994},
     journal={C. R. Math. Acad. Sci. Soc. R. Can.},
      volume={37},
      number={3},
       pages={100\ndash 113},
      review={\MR{3443373}},
}

\bib{MR3434502}{article}{
      author={Selinger, Nikita},
      author={Yampolsky, Michael},
       title={Constructive geometrization of {T}hurston maps and decidability
  of {T}hurston equivalence},
        date={2015},
        ISSN={2199-6792},
     journal={Arnold Math. J.},
      volume={1},
      number={4},
       pages={361\ndash 402},
         url={https://doi-org.proxyiub.uits.iu.edu/10.1007/s40598-015-0024-4},
      review={\MR{3434502}},
}

\bib{MR4068871}{article}{
      author={Shepelevtseva, Anastasia},
      author={Timorin, Vladlen},
       title={Invariant spanning trees for quadratic rational maps},
        date={2019},
        ISSN={2199-6792},
     journal={Arnold Math. J.},
      volume={5},
      number={4},
       pages={435\ndash 481},
         url={https://doi-org.proxyiub.uits.iu.edu/10.1007/s40598-019-00123-w},
      review={\MR{4068871}},
}

\bib{zbMATH07239271}{article}{
      author={Thurston, Dylan},
       title={A positive characterization of rational maps},
    language={English},
        date={2020},
        ISSN={0003-486X},
     journal={Ann. Math. (2)},
      volume={192},
      number={1},
       pages={1\ndash 46},
}

\end{biblist}
\end{bibdiv}

\end{document}